\documentclass[11pt,letterpaper]{article}

\usepackage{fullpage}

\usepackage{lmodern}
\usepackage[T1]{fontenc}
\usepackage{textcomp}

\usepackage{amsmath}
\usepackage{amssymb}
\usepackage{amsthm}
\usepackage{bm}
\usepackage{mleftright}\mleftright
\usepackage{mathtools}\mathtoolsset{showonlyrefs}
\usepackage{nicematrix}

\usepackage{thm-restate}
\usepackage{thmtools}

\DeclarePairedDelimiter{\prn}{(}{)}
\DeclarePairedDelimiter{\set}{\{}{\}}
\DeclarePairedDelimiterX{\Set}[2]{\{}{\}}{\,#1\mathclose{}\nonscript\;\delimsize|\nonscript\;\mathopen{}#2\,}

\DeclarePairedDelimiter{\floor}{\lfloor}{\rfloor}
\DeclarePairedDelimiter{\ceil}{\lceil}{\rceil}
\newcommand{\GF}{\mathrm{GF}}

\newcommand{\ones}{\mathbf{1}}
\DeclareMathOperator{\sgn}{sgn}
\DeclareMathOperator{\chr}{char}

\DeclareMathOperator{\rank}{rank}
\DeclareMathOperator{\supp}{supp}

\newcommand{\Mn}{M\texorpdfstring{${}^\natural$}{-natural}}

\usepackage{algorithm}
\usepackage[noend]{algpseudocode}

\usepackage[
    date=year,
    isbn=false,
    natbib,
    url=false,
    sorting=nyt,
    style=trad-abbrv
]{biblatex}
\AtEveryBibitem{\clearname{editor}}
\usepackage{booktabs}

\usepackage[
    bookmarksnumbered,
    colorlinks,
    hypertexnames=false,
    pdfdisplaydoctitle,
    pdfusetitle,
    unicode
]{hyperref}
\usepackage{url}

\usepackage{upref}
\usepackage{zref-clever}
\zcsetup{
    cap,
    nameinlink=false
}

\usepackage{enumitem}
\setlist[enumerate]{label={\upshape{(\arabic*)}}}

\usepackage{tikz}
\usetikzlibrary{arrows.meta,positioning,decorations.pathreplacing}

\usepackage{subcaption}
\usepackage{comment}
\usepackage[dvipsnames]{xcolor}

\newtheorem{theorem}{Theorem}[section]
\newtheorem{lemma}[theorem]{Lemma}
\newtheorem{proposition}[theorem]{Proposition}
\newtheorem{corollary}[theorem]{Corollary}
\newtheorem{conjecture}[theorem]{Conjecture}

\zcRefTypeSetup{conjecture}{
    Name-sg=Conjecture,
    Name-pl=Conjectures,
    name-sg=conjecture,
    name-pl=conjectures,
}
\zcRefTypeSetup{claim}{
    Name-sg=Claim,
    Name-pl=Claims,
    name-sg=claim,
    name-pl=claims,
}

\theoremstyle{definition}

\newtheorem{remark}[theorem]{Remark}

\newcommand{\cA}{\mathcal{A}}
\newcommand{\cB}{\mathcal{B}}
\newcommand{\cC}{\mathcal{C}}

\newcommand{\cH}{\mathcal{H}}

\newcommand{\cM}{\mathcal{M}}
\newcommand{\cP}{\mathcal{P}}
\newcommand{\cQ}{\mathcal{Q}}

\newcommand{\cT}{\mathcal{T}}
\newcommand{\cU}{\mathcal{U}}
\newcommand{\cX}{\mathcal{X}}

\newcommand{\bM}{\mathbf{M}}

\newcommand{\K}{\mathbb{K}}
\newcommand{\F}{\mathbb{F}}
\newcommand{\Q}{\mathbb{Q}}
\newcommand{\N}{\mathbb{N}}
\newcommand{\Z}{\mathbb{Z}}
\newcommand{\R}{\mathbb{R}}

\newcommand{\bmcU}{\bm{\cU}}
\newcommand{\bmPhi}{\bm{\Phi}}

\newcommand{\Zp}{\Z_{\ge0}}

\newcommand{\opt}{\mathrm{OPT}}

\newcommand{\given}{\mathrel{|}}

\DeclareMathOperator{\argmax}{arg\,max}

\DeclareMathOperator{\err}{err}

\newcommand{\symdif}{\mathbin{\triangle}}

\newcommand{\note}[3]{{\color{#1}[{\tiny\textbf{#2: #3}}]\marginpar{\color{#1}*}}}
\newcommand{\onote}[1]{\note{purple}{Taihei}{#1}}
\newcommand{\snote}[1]{\note{cyan}{Tamás}{#1}}

\makeatletter
\newcommand{\customlabel}[2]{%
\protected@edef\@currentlabel{#2}\label{#1}%
}
\makeatother

\title{Generalizing the Multiple Exchange Property for Matroid Bases}
\date{}

\author{
    Taihei Oki\thanks{Institute for Chemical Reaction Design and Discovery (ICReDD), Hokkaido University, Sapporo, Hokkaido, Japan. D3 Center, The University of Osaka, Osaka, Japan. Center for Advanced Intelligence Project, RIKEN, Tokyo, Japan. E-mail: \texttt{oki@icredd.hokudai.ac.jp}.}
    \and Tamás Schwarcz\thanks{
    HUN-REN Alfréd Rényi Institute of Mathematics, Budapest, Hungary. E-mail: \texttt{tamas.schwarcz@ttk.elte.hu}. Most of this work was done while this author was affiliated to the London School of Economics and Political Science.}
}

\begin{document}
\maketitle

\begin{abstract}
    The multiple exchange property for matroid bases states that for any bases $A$ and $B$ of a matroid and any subset $X\subseteq A\setminus B$, there exists a subset $Y\subseteq B\setminus A$ such that both $A-X+Y$ and $B+X-Y$ are bases.
    This classical result has found applications not only in matroid theory, but also in the analysis and design of various algorithms.
    This paper generalizes the multiple exchange property in two directions.
    First, we prove a common generalization of this and other known basis exchange properties by showing that for any subsets $X \subseteq A \setminus B$ and $Y \subseteq B \setminus A$, there exist subsets $U \subseteq A \setminus B$ and $V \subseteq B \setminus A$ such that $X\subseteq U$, $Y\subseteq V$, $A-U+V$ and $B+U-V$ are bases, and $|U|=|V|$ is at most the rank of $X+Y$.
    Second, we develop a general framework for deriving extensions of the Grassmann--Plücker identity.
    For matroids representable over fields of characteristic zero, this framework yields new exchange and reconfiguration properties, the latter requiring the resulting basis pair to be reachable from the original one by a sequence of symmetric exchanges.
    For this matroid class, we obtain an exchange theorem that simultaneously generalizes our first result and the recent Equitability Theorem (SODA 2026). 
    Within the same representable setting, we also derive a weighted equitability theorem, with an application to matroid-constrained EF1 allocations for two agents with additive valuations.
\end{abstract}

\section{Introduction} \label{sec:intro}

Matroids, introduced by Whitney~\cite{Whitney1935-kz} and Nakasawa~\cite{nakasawa1935zur} as a combinatorial abstraction of linear independence, play a fundamental role in combinatorial optimization~\cite{schrijver2003combinatorial}.
They provide a unifying framework for discrete structures that admit greedy algorithms for optimizing a linear objective function~\cite{edmonds1971matroids}.
More recently, matroids and their quantitative generalizations, valuated matroids, have been employed in mathematical economics to model constraints of allocations and agents’ utility functions satisfying the gross substitutes property~\cite{fujishige2003a,paes2017gross}.

Matroids can be defined in several equivalent ways, but the most basic characterization is through the \emph{exchange property} of bases. 
A matroid $\bM$ consists of a finite ground set $E = E(\bM)$ and a nonempty set family $\cB = \cB(\bM)$ over $E$, called \emph{bases}, satisfying the following \emph{exchange property}: 
for any $A, B \in \cB$ and $x \in A \setminus B$, there exists an element $y \in B \setminus A$ such that $A-x+y \in \cB$, where $A-x+y$ is a shorthand for $(A \setminus \set{x}) \cup \set{y}$. 
This simplest exchange property gives rise to more advanced exchange properties.
The \emph{symmetric} (or \emph{simultaneous}) \emph{exchange property}~\cite{brualdi1969comments} asserts that for any two bases $A$ and $B$ and $x \in A \setminus B$, there exists $y \in B \setminus A$ such that both $A-x+y$ and $B+x-y$ are bases.
This property, in turn, extends to the following \emph{multiple exchange property}, which allows multiple elements to be exchanged simultaneously.
Here, $A-X+Y$ denotes $(A \setminus X) \cup Y$ for sets $A, X, Y$ with $X \subseteq A$ and $Y \cap A = \emptyset$.

\begin{theorem}[{Multiple exchange property~\cite{brylawski1973some, greene1973multiple, woodall1974exchange}}]\label{thm:multiple}
  Let $A$ and $B$ be bases of a matroid.
  For any $X \subseteq A\setminus B$, there exists $Y\subseteq B\setminus A$ such that $A-X+Y$ and $B+X-Y$ are bases.
\end{theorem}

\zcref{thm:multiple} serves as a key ingredient in the analysis and design of various algorithms.
It has proved to be particularly useful for monotone submodular function maximization under matroid constraints, which is a central problem of combinatorial optimization.
For $k$ matroid constraints, Lee, Sviridenko, and Vondr\'ak~\cite{lee2010submodular} used a generalization of the multiple exchange property to prove the $1/(k+\varepsilon)$-approximation guarantee of a local search algorithm.
The current best approximation ratio of
$\frac{2k\ln 2}{1+\ln 2}+O(\sqrt{k})$ for this problem, due to Feldman and
Ward~\cite{feldman2026submodular}, was proved using the generalization of Greene
and Magnanti~\cite{greene1975some}.
Buchbinder, Feldman, and Garg~\cite{buchbinder2023deterministic} used the multiple exchange property to break the long-standing $2$-approximation barrier for a single matroid constraint.
In the single matroid case, it was also used 
in the setting of a noisy value oracle~\cite{huang2022efficient} and, more recently, for giving fixed-parameter tractable algorithms parameterized by the rank~\cite{nematollahi2026fixed}.
Besides submodular maximization, the generalization by Greene and Magnanti~\cite{greene1975some} was used in the breakthrough result of beating the $1.5$-approximation for $s$--$t$ path graph TSP~\cite{traub2023beating} and matching the best known approximation factor for Steiner Tree via a combinatorial local search algorithm~\cite{traub2025better}.
Further applications of the multiple exchange property include approximating maximum weight series-parallel subgraphs~\cite{calinescu2024approximation} and a certain signal processing problem~\cite{li2019scalable}.

In this paper, we generalize \zcref{thm:multiple} so that it encompasses various exchange properties beyond \zcref{thm:multiple} and yields several applications.
We first outline the related work and motivations in \zcref{sec:motivations}, followed by a detailed presentation of our contributions in \zcref{sec:contributions}.

\subsection{Related Work and Motivations}\label{sec:motivations}

\paragraph{Further exchange properties.}

Besides \zcref{thm:multiple}, extensions of the symmetric exchange property have been extensively studied.
One example is the following result of Greene~\cite{greene1974another}, which recovers the symmetric exchange property for $|X|=1$. 

\begin{theorem} [Greene~\cite{greene1974another}] \label{thm:greene}
    Let $A$ and $B$ be bases of a matroid.
    For any $X\subseteq A\setminus B$ and $Y\subseteq B\setminus A$ with $|X|+|Y|\ge |A\setminus B|+1$, there exist nonempty $U\subseteq X$ and $V\subseteq Y$ such that $A-U+V$ and $B+U-V$ are bases.
\end{theorem}

Kung~\cite{kung1978alternating} further observed that the condition $|X|+|Y|\ge |A\setminus B|+1$ in \zcref{thm:greene} can be replaced by the dependence of $X+Y$. 
While these results have not yet found as many applications in the literature as \zcref{thm:multiple}, we explain in \zcref{sec:relation} that they imply several interesting exchange properties, such as
a relaxed base orderability property of all matroids (\zcref{cor:relaxed}) and a uniqueness version of \zcref{thm:multiple} (\zcref{cor:unique-multiple}).

\paragraph{Proofs via algorithms.}
Besides the applications of \zcref{thm:multiple} to algorithm design mentioned above, the theorem also admits a short and natural algorithmic proof.
Specifically, Woodall~\cite{woodall1974exchange} reduced the problem of finding the set $Y$ in \zcref{thm:multiple} to matroid intersection and exploited Edmonds' matroid intersection theorem~\cite{edmonds1970submodular} to show the existence of the desired exchangeable set $Y$.
The algorithmic proof also has the advantage that it yields a polynomial-time algorithm for finding it.
Extensions of Edmonds' theorem and the primal-dual framework are known for the more general problem of weighted matroid intersection~\cite{frank1981weighted}, and it is a natural direction to develop an exchange property using such an advanced framework.

\paragraph{Sequences of symmetric exchanges.} 
An extensive line of research concerns sequences of symmetric exchanges between basis pairs.
For bases $A, A', B, B'$ of a matroid, we say that the ordered basis pairs $(A,B)$ and $(A',B')$ \emph{differ by a symmetric exchange} if there exist $x\in A\setminus B$ and $y\in B\setminus A$ such that $A'=A-x+y$ and $B'=B+x-y$.
For an integer $k\ge 0$, a $k$-step symmetric exchange sequence from $(A,B)$ to $(A',B')$ is a sequence $(A_0,B_0), (A_1,B_1),\dots, (A_k,B_k)$ of basis pairs such that $(A_0,B_0)=(A,B)$, $(A_k,B_k)=(A',B')$, and consecutive pairs differ by a symmetric exchange.
If such a sequence exists, we say that $(A',B')$ is \emph{obtained from $(A,B)$
using $k$ symmetric exchanges}; if such a sequence exists for some $k\ge 0$,
we say that $(A',B')$ is \emph{obtained from $(A,B)$ using symmetric exchanges}.
White~\cite{white1980unique} conjectured that for any pairs of bases $(A,B)$ and $(A',B')$ with $A\cap B =A'\cap B'$ and $A\cup B=A'\cup B'$, $(A',B')$ can be obtained from $(A,B)$ using symmetric exchanges.
This conjecture and the more general conjectures of White~\cite{white1980unique} on tuples of bases have received significant attention, in part due to their connection to toric ideals of matroids, see e.g.~\cite{blasiak2008toric,lason2014toric}.
While the conjecture for basis pairs was verified for regular matroids~\cite{berczi2026reconfiguration},  Larson~\cite{larson2026counterexamples} disproved it for binary matroids in a very recent breakthrough.

The counterexample of Larson~\cite{larson2026counterexamples} still leaves the case $(A',B')=(B,A)$ of White's conjecture open.
This case would also follow from the following extension of the multiple exchange property by taking $X=A\setminus B$.

\begin{conjecture}[Greene and Magnanti~\cite{greene1975some} and Kotlar and Ziv~\cite{kotlar2013serial}] \label{conj:serial-mep}
    Let $A$ and $B$ be bases of a matroid.
    For any $X\subseteq A\setminus B$, there exists $Y\subseteq B\setminus A$ such that $A-X+Y$ and $B+X-Y$ are bases and the basis pair $(A-X+Y,B+X-Y)$ can be obtained from $(A,B)$ using $|X|$ symmetric exchanges.
\end{conjecture}

\zcref{conj:serial-mep} can be formulated equivalently as the existence of $Y\subseteq B\setminus A$ in \zcref{thm:multiple} such that there exist orderings $(x_1,\dots, x_k)$ of $X$ and $(y_1,\dots, y_k)$ of $Y$ such that both $A-\set{x_1, \dots, x_i}+\set{y_1,\dots,y_i}$ and $B+\set{x_1,\dots, x_i}-\set{y_1,\dots, y_i}$ are bases for $i \in [k]$. Greene and Magnanti~\cite{greene1975some} and Kotlar and Ziv~\cite{kotlar2013serial} independently proved the case $|X|=2$ of the conjecture.
For $|X|=3$, McGuinness~\cite{mcguinness2022serial} verified the conjecture for binary matroids, while for general matroids 
Kotlar~\cite{kotlar2013circuits} showed the weaker statement that there exist serially exchangeable subsets of size three whenever $|A\setminus B|\ge 3$.
For $|X|\ge 4$, this weaker variant is already open.
If $X=A\setminus B$, then necessarily $Y=B\setminus A$.
This case of \zcref{conj:serial-mep} is the well-studied conjecture of Gabow~\cite{gabow1976decomposing}, which is known to hold for regular matroids~\cite{berczi2026reconfiguration}.

\paragraph{Grassmann--Plücker identity.}
The original idea of the multiple exchange property by Greene~\cite{greene1973multiple} comes from a determinant identity called the (multiple) \emph{Grassmann--Plücker identity}.
Specifically, for a matrix $K \in \F^{r \times n}$ over a field $\F$, subsets $A, B \subseteq [n] \coloneqq \set{1, \dotsc, n}$ with $|A| = |B| = r$, and any $X \subseteq A \setminus B$, the identity states that
\begin{align}\label{eq:multiple-gp}
    \det K[A] \det K[B] = \sum_{Y \in \binom{B \setminus A}{|X|}} \pm \det K[A - X + Y] \det K[B + X - Y],
\end{align}
where the sign is determined appropriately.
Here, $K[J]$ for $J \subseteq [n]$ denotes the submatrix of $K$ consisting of the columns indexed by $J$, and $\binom{B \setminus A}{|X|}$ denotes the family of subsets of $B \setminus A$ of size $|X|$.
The case $|X|=1$ of the identity~\eqref{eq:multiple-gp} has been extensively studied in projective and algebraic geometry, where it defines the Plücker relations characterizing the Grassmann variety~\cite{gelfand2008discriminants}.

If $K[A]$ and $K[B]$ are nonsingular, or equivalently, if $A$ and $B$ are bases of the matroid represented by $K$, then the right-hand side of~\eqref{eq:multiple-gp} has at least one nonzero term, implying the existence of $Y \in \binom{B\setminus A}{|X|}$ such that $A-X+Y$ and $B+X-Y$ are bases as well.
This proves \zcref{thm:multiple} for representable matroids.
Therefore, seeking extensions of~\eqref{eq:multiple-gp} is a promising approach to further exchange properties for representable matroids.

\paragraph{Equitability.}
Recently, Akrami, Liu, Raj, and Végh~\cite{equitability_journal,equitability_soda} proved the following \emph{Equitability Theorem}, originally conjectured by Fekete and Szabó.

\begin{theorem}[Equitability~\cite{equitability_journal}]\label{thm:equitability}
    Let $\bM$ be a matroid, and assume that the ground set $E = E(\bM)$ can be partitioned into $k \ge 1$ disjoint bases. Let $T \subseteq E$ be an arbitrary subset.
    Then, there exists a partition $(B_1, \dotsc, B_k)$ of $E$ into $k$ bases such that $\floor*{\frac{|T|}{k}} \le |B_i \cap T| \le \ceil*{\frac{|T|}{k}}$ holds for $i \in [k]$. Moreover, such bases can be found in polynomial time.
\end{theorem}

The proof of \zcref{thm:equitability} in~\cite{equitability_journal} is based on the repeated application of the following exchange property. 

\begin{theorem}[Akrami, Liu, Raj,  V\'egh~\cite{equitability_journal}]\label{thm:equitability-exchange}
  Let $A$ and $B$ be bases of a matroid.
  For any $X\subseteq A\setminus B$ and a bipartition $(Y_1, Y_2)$ of $B\setminus A$ with $|X|>|Y_1|$, there exist $U \subseteq X$ and $V \subseteq B \setminus A$ such that $|V \cap Y_2|=1$ and $A-U+V$ and $B+U-V$ are bases.
\end{theorem}

While \zcref{thm:equitability-exchange} contains the symmetric exchange property, \zcref{thm:multiple,thm:equitability-exchange} do not seem to imply each other.
Hence, it is natural to expect a common generalization of \zcref{thm:multiple,thm:equitability-exchange}.

\paragraph{EF1 allocations.}
    \zcref{thm:equitability} has a natural interpretation in the context of matroid-constrained fair division problems.
    This line of research was initiated by Barman and Biswas~\cite{biswas2018fair}.
    Formally, an instance of the \emph{matroid-constrained fair division problem} is a tuple $(n, E, (v_i)_{i \in [n]}, \bM)$ consisting of the number of agents $n$, a finite ground set $E$ of indivisible goods,  valuation functions $v_i\colon 2^E\to \R_{\ge 0}$ for $i \in [n]$, and a matroid $\bM$ on ground set $E$.
    A \emph{feasible allocation} is a partition $(A_1,\dots, A_n)$ of $E$ into $n$ bases of $\bM$.
    The allocation is \emph{envy-free up to one good (EF1)} if for all agents $i,j\in [n]$ with $i\ne j$, there exists $Y\subseteq A_j$ with $|Y|\le 1$ such that $v_i(A_i)\ge v_i(A_j-Y)$.
    Using this terminology, \zcref{thm:equitability} directly shows the existence of matroid-constrained EF1 allocations whenever $E$ can be partitioned into $n$ bases and the valuation functions are identical, additive, and binary, that is, they take values $0$ or $1$ on single elements.
    Akrami, Liu, Raj, and Végh~\cite{equitability_journal} used \zcref{thm:equitability-exchange} to extend this result to identical, additive, and tri-valued valuations, meaning that they take at most three different values on single elements.   
    By a cut-and-choose method, this also settles the case of two agents if at least one of the two valuations is additive and tri-valued~\cite{equitability_journal}.
    Besides these breakthrough results on general matroids, further positive results are known for regular matroids with any number of agents and identical additive valuations~\cite{berczi2026reconfiguration,equitability_journal}, and for base-orderable matroids in certain two- and three-agent cases~\cite{dror2023fair}.

\paragraph{Robustness of local search.}
An unexpected application of our exchange property lies in the \emph{robustness} of algorithms.
In practical optimization, we often encounter situations where only noisy or imperfect information is available.
Therefore, it is desirable for an algorithm to possess a robustness property, ensuring that small perturbations in the input lead to only minor deviations in the output.
For monotone submodular function maximization, the greedy algorithm that repeatedly appends an element with the maximal marginal gain achieves a $(1 - 1/\mathrm{e})$-approximation ratio and, moreover, is known to be robust in the sense that if, at each iteration, the algorithm selects an element whose marginal gain is not exactly maximal but deviates from the true maximizer by at most some local error, then the final function value deteriorates by at most the sum of these local errors compared with the ideal $1 - 1/\mathrm{e}$ approximation~\cite{niazadeh2023online,nie2022an,streeter2008an}.
For linear objective optimization over matroid bases, Oki and Sakaue~\cite{oki2024no} established an analogous robustness property of the greedy algorithm within the more general framework of \Mn-concave functions.

We turn to the local-search algorithm for linear optimization over matroid bases.
Starting from an arbitrary basis $B_0$, the algorithm iteratively updates it by replacing at most one element at a time: at the $k$th iteration, it finds a basis $B_k$ satisfying $|B_k \setminus B_{k-1}| \le 1$ that maximizes the objective value among all bases $B$ with $|B \setminus B_{k-1}| \le 1$.
It is known that, for any $k \ge 0$, the resulting $B_k$ is optimal among all bases $B \in \cB$ such that $|B \setminus B_0| \le k$~\cite{Shioura2022-tc}.
This property enables a warm-started local search from a near-optimal basis~\cite{oki2023faster}.
However, the theoretical guarantee of the performance of the algorithm when the maximization at each iteration is performed with some error has remained open.

\subsection{Our Contributions}\label{sec:contributions}

This paper makes two main contributions: one extending \zcref{thm:multiple} and another providing stronger extensions for representable matroids, in some cases under characteristic constraints of the ground field.
The relations among the exchange properties are summarized in \zcref{fig:relations}.

\begin{figure}[t]
    \centering
    \begin{tikzpicture}[
        >=Latex,
        box/.style={
            font=\small,
            draw,
            rounded corners,
            align=center,
            inner sep=4pt,
            minimum width=16mm,
        },
        rep/.style={
            box,
            dashed,
        },
        lbl/.style={
            font=\small,
            midway,
            above,
        },
        solid arrow/.style={
            ->,
        },
        dashed arrow/.style={
            ->,
            dashed,
        },
        bent/.style={
            rounded corners=5pt,
        },
    ]

        \node[rep] (master) at (4.4,6.0)
            {\zcref{thm:gp-partition} [\textbf{This paper}]};

        \node[rep] (twomax) at (0.8,4.9)
            {\zcref{thm:two-max} [\textbf{This paper}]};

        \node[rep] (withranks) at (0.8,3.8)
            {\zcref{thm:with_ranks} [\textbf{This paper}]};

        \node[rep] (blue) at (8,4.9)
            {\zcref{thm:blue} [\textbf{This paper}]};

        \node[rep] (reconfmain) at (0.8,2.5)
            {Reconfiguration version of \zcref{thm:main}~(i)\\
             (\zcref{cor:reconf-main}) [\textbf{This paper}]};

        \node[rep] (reconfARV) at (8,2.5)
            {Reconfiguration version of \zcref{thm:equitability-exchange}\\
             (\zcref{cor:reconf-arlv})~[\textbf{This paper}]};

        \node[rep] (reconfmep) at (4.5,0.8)
            {Reconfiguration version of MEP\\
             (\zcref{thm:reconf-mep}) [\textbf{This paper}]};

        \node[box] (main) at (-1.8,0.8)
            {\zcref{thm:main}~(i) [\textbf{This paper}]};

        \node[box] (ARV) at (10.3,0.8)
            {Akrami, Liu, Raj, Végh\\
             (\zcref{thm:equitability-exchange})~
             \cite{equitability_journal,equitability_soda}};

        \node[box] (Kung) at (-1.8,-0.5)
            {Kung~\cite{kung1978alternating}};

        \node[rep] (WEquit) at (7.3,-0.7)
            {Weighted Equitability\\
             (\zcref{thm:weighted-equitability} (i)) [\textbf{This paper}]};

        \node[box] (MEP) at (2.6,-1.2)
            {MEP (\zcref{thm:multiple})\\
             \cite{brylawski1973some,greene1973multiple,woodall1974exchange}};

        \node[box] (Greene) at (-1.8,-2.0)
            {Greene (\zcref{thm:greene})~\cite{greene1974another}};

        \node[box] (Equit) at (7.3,-2.4)
            {Equitability\\
             (\zcref{thm:equitability})~\cite{equitability_journal}};

        \node[box] (SEP) at (2.6,-3.2)
            {Symmetric exchange property~\cite{brualdi1969comments}};

        \draw[dashed arrow,bent]
            (master.west) -| (twomax.north);

        \draw[dashed arrow,bent]
            (master.east) -| (blue.north);

        \draw[dashed arrow]
            (twomax) -- (withranks);

        \draw[dashed arrow]
            (withranks.south)
            --
            (withranks.south |- reconfmain.north);

        \draw[dashed arrow,bent]
            (withranks.west)
            -|
            ([xshift=8mm]main.north west);

        \draw[dashed arrow,bent]
            (withranks.east) -| (reconfARV.north);

        \draw[dashed arrow,bent]
            (blue.west) -| ([xshift=28mm]reconfmain.north);

        \draw[dashed arrow]
            (blue) -- (reconfARV);

        \draw[dashed arrow]
            ([xshift=-17.8mm]reconfmain.south)
            --
            ([xshift=-17.8mm]reconfmain.south |- main.north)
            node[lbl,right,xshift=0.2cm,align=center]
                {$\chr=0$};

        \draw[dashed arrow]
            ([xshift=28mm]reconfmain.south)
            --
            ([xshift=28mm]reconfmain.south |- reconfmep.north)
            node[lbl,right,xshift=0.2cm,align=center]
                {$\chr=0$};

        \draw[dashed arrow]
            ([xshift=15mm]reconfmep.south)
            --
            ([xshift=15mm]reconfmep.south |- WEquit.north);

        \draw[dashed arrow]
            ([xshift=-9.1mm]reconfmep.south)
            --
            ([xshift=-9.1mm]reconfmep.south |- MEP.north)
            node[lbl,left,align=center]
                {$\chr \notin \set{2,\dotsc,|X|}$};

        \draw[dashed arrow]
            (WEquit)
            --
            (Equit)
            node[lbl,right,xshift=0.1cm,align=center]
                {$\chr=0$};

        \draw[dashed arrow]
            ([xshift=10mm]reconfARV.south)
            --
            ([xshift=10mm]reconfARV.south |- ARV.north)
            node[lbl,right,xshift=0.2cm,align=center]
                {$\chr=0$};

        \draw[solid arrow]
            (main) -- (Kung);

        \draw[solid arrow,bent]
            ([xshift=13mm]main.south) |- (MEP.west);

        \draw[solid arrow]
            (Kung) -- (Greene);

        \draw[solid arrow]
            (MEP) -- (SEP);

        \draw[solid arrow,bent]
            (Greene.south) |- (SEP.west);

        \draw[solid arrow,bent]
            (ARV.south) |- (Equit.east);

        \draw[solid arrow,bent]
            (ARV.south) |- (SEP.east);
    \end{tikzpicture}

    \caption{%
        Relations among exchange properties.
        Theorems surrounded by dashed lines hold only for representable matroids, in some cases under characteristic constraints.
        A solid arrow $A \to B$ means that $A$ contains $B$ as a special case, and a dashed arrow indicates that the implication is only for representable cases.
        ``MEP'' stands for ``Multiple Exchange Property''.
    }
    \label{fig:relations}
\end{figure}

As the first main contribution, we show the following exchange property.

\begin{restatable}{theorem}{main}\label{thm:main}
    Let $A$ and $B$ be bases of a matroid $\bM$ with rank function $r$, $X\subseteq A\setminus B$, and $Y\subseteq B\setminus A$.
    Then, (i) there exist $U \subseteq A\setminus B$ and $V \subseteq B \setminus A$ such that $X\subseteq U$, $Y \subseteq V$, $|U| = |V|\le r(X+Y)$, and $A-U+V$ and $B+U-V$ are bases, and (ii) such a pair $(U,V)$ can be found using a polynomial number of independence oracle queries to $\bM$.
\end{restatable}

\zcref{thm:main} specializes to \zcref{thm:multiple} if $Y=\emptyset$.
Like Woodall's proof~\cite{woodall1974exchange} of \zcref{thm:multiple}, our proof of \zcref{thm:main} is algorithmic.
We formulate the problem of finding the sets $U$ and $V$ as weighted matroid intersection, and bound the optimal value via Frank’s weight splitting theorem~\cite{frank1981weighted}.
While the proof idea stems from Woodall's proof, our proof is more involved as it uses the submodularity of rank functions several times and exploits the min-max theorem for the minimum-weight basis problem.
As our proof is algorithmic, it also provides a polynomial-time algorithm for finding the sets $U$ and $V$ through any weighted matroid intersection algorithm. 

Besides \zcref{thm:multiple}, \zcref{thm:main} also includes Kung's exchange property~\cite{kung1978alternating}, and hence \zcref{thm:greene}.
This inclusion is made explicit in \zcref{thm:ultra_kung} by formulating \zcref{thm:main} in a manner similar to \zcref{thm:greene} with the additional constraint that the common size of the symmetrically exchangeable subsets $U, V$ lies in an interval of length $r(X+Y)-\max\set{|X|,|Y|}$.
To show the relevance of these consequences, we observe that \zcref{thm:greene} already implies certain relaxed base-orderability properties of all matroids and a relaxation of Gabow's conjecture~\cite{gabow1976decomposing} on serial exchanges (\zcref{cor:relaxed}).

As an application of \zcref{thm:main}, we prove the following robustness property of the local-search algorithm of matroid bases, asserting that its global suboptimality is bounded by the sum of the local errors.

\begin{restatable}{theorem}{robustness}\label{thm:robustness}
  Let $\bM$ be a matroid with ground set $E$ and basis family $\cB$, and $w\colon E \to \R$ a weight function.
  Let $B_0, B_1, \dotsc, B_k$ be a sequence of bases with $|B_i \setminus B_{i-1}| \le 1$ for all $1 \le i \le k$.
  Then,
  \begin{align}
      w(B_k) \ge \max\Set{w(B)}{B \in \cB, |B \setminus B_0| \le k} - \sum_{i=1}^k \err(B_i \given B_{i-1})
  \end{align}
  holds, where $\err(B_i \given B_{i-1}) \coloneqq \max\Set{w(B)}{B \in \cB, |B \setminus B_{i-1}| \le 1} - w(B_i) \ge 0$ for $1 \le i \le k$.
\end{restatable}

While the robustness proof of the greedy algorithm in~\cite{oki2024no} depends solely on the symmetric exchange property, our proof of \zcref{thm:robustness} makes use of \zcref{thm:main} in the case $|X| = |Y| = 1$.

\medskip

The second main contribution of this paper is the development of a general framework for deriving extensions of the Grassmann--Plücker identity~\eqref{eq:multiple-gp} and thereby establishing exchange properties for representable matroids.
Our framework is described as follows.
For sets $A$ and $B$, let
\begin{align}
    \cX_{A,B}
    \coloneqq
    \Set{(X,Y)}{X\subseteq A,\ Y\subseteq B,\ |X|=|Y|}.
\end{align}
For a matrix $K \in \F^{r \times n}$ over a field $\F$ and $A,B \in \binom{[n]}{r}$, define the function $\mu^K_{A,B}\colon \cX_{A\setminus B, B\setminus A} \to \F$ by
\begin{align}\label{def:c-intro}
    \mu^K_{A,B}(X,Y) \coloneqq \pm \det K[A-X+Y] \det K[B+X-Y]
\end{align}
for $(X,Y) \in \cX_{A\setminus B,B\setminus A}$, where the sign in~\eqref{def:c-intro} is appropriately defined.
We exploit the \emph{harmonicity} of $\mu_{A,B}^K$, which states
\begin{align}\label{def:harmonic-mu}
    \sum_{a\in X} \mu_{A,B}^K(X-a,Y)
    =
    \sum_{b\in (B\setminus A)-Y} \mu_{A,B}^K(X,Y+b)
\end{align}
for any $X\subseteq A\setminus B$ and $Y\subseteq B\setminus A$ with $|X|=|Y|+1$.
The harmonicity of set functions and multilinear polynomials has been extensively studied in the representation theory of symmetric groups~\cite{ceccherini-silberstein2013representation}.
In particular, the harmonicity~\eqref{def:harmonic-mu} for $\mu_{A,B}^K$ is known as the \emph{Plücker relations} in projective and algebraic geometry~\cite{gelfand2008discriminants}.

We first use harmonicity to prove the following \emph{reconfiguration} version of \zcref{thm:multiple} for representable matroids under certain characteristic constraints, claiming that $(A-X+Y, B+X-Y)$ can be obtained from $(A,B)$ using symmetric exchanges.
Recall that the \emph{characteristic} $\chr(\F)$ of a field $\F$ is the smallest positive integer $k$ such that $k\cdot1=0$; if no such $k$ exists, then $\chr(\F)\coloneqq0$.

\begin{restatable}{theorem}{reconfmep} \label{thm:reconf-mep}
    Let $A$ and $B$ be bases of a matroid $\bM$ representable over a field $\F$ and let $X\subseteq A\setminus B$.
    If $\chr(\F) \not \in \set{2,3,\dots, |X|}$, then there exists $Y\subseteq B\setminus A$ such that $A-X+Y$ and $B+X-Y$ are bases of $\bM$ and the basis pair $(A-X+Y,B+X-Y)$ can be obtained from $(A,B)$ by a sequence of symmetric exchanges in $\bM$.
\end{restatable}

\zcref{thm:reconf-mep} verifies a relaxation of \zcref{conj:serial-mep} under a representability constraint in which the prescribed bound on the length of the sequence of symmetric exchanges is dropped.
Observe that \zcref{thm:reconf-mep} applies to all matroids representable over a field of characteristic zero. 
This class coincides with the class of complex-representable matroids~\cite[Proposition~6.8.11]{oxley2011matroid}, and it includes many important families, such as regular matroids and gammoids~\cite{mason1972on}, and thus their subclasses, including graphic and transversal matroids.

In the special case $X=A\setminus B$, we obtain for this matroid class the common relaxation of the conjectures of Gabow~\cite{gabow1976decomposing} and White~\cite{white1980unique}: the basis pair $(B,A)$ can be obtained from $(A,B)$ using a sequence of symmetric exchanges.
This relaxation already implies the following weighted variant of equitability.

\begin{restatable}{theorem}{weightedequitability} \label{thm:weighted-equitability} 
    Let $\bM$ be a matroid representable over a field $\F$ of characteristic zero.
    Assume that the ground set $E=E(\bM)$ can be partitioned into $k \ge 1$ disjoint bases.
    Let $w\colon E \to \R_{\ge 0}$ be a weight function. 
    (i) Then, there exists a partition $(B_1,\dots, B_k)$ of $E$ into $k$ bases such that \[w(B_1)\ge w(B_2) \ge  \dots \ge w(B_k) \ge w(B_1)-\max_{e \in B_1} w(e).\]
    (ii) Moreover, if $w\colon E \to \set{0,1,\dots, W}$ for some $W\in \Z_{\ge 0}$ and a matrix representing $\bM$ over $\F$ is given,
    then such a partition can be found by a randomized algorithm in expected time polynomial in $|E|$ and $W$, assuming unit-time field operations.
\end{restatable}

Using a cut-and-choose argument as in \cite[Corollary~5.2]{equitability_journal}, \zcref{thm:weighted-equitability} settles the existence of matroid-constrained EF1 allocations for two agents whenever the matroid is representable over a field of characteristic zero and the valuation of at least one of the agents is additive.
As already noted in \cite[Remark~A.2]{equitability_journal}, our argument using a sequence of symmetric exchanges does not seem sufficient to show the existence of EF1 allocations for an arbitrary number of agents with identical additive valuations.

Motivated by the role of harmonicity in the proof of \zcref{thm:reconf-mep}, we develop this framework further to derive additional basis exchange properties.
The vector space of harmonic functions admits an explicit basis whose elements are indexed by non-crossing bijections, or equivalently, by non-crossing perfect bipartite matchings~\cite{filmus2016an,filmus2019harmonicity}.
Building on this fact, we use expansions of $\mu^K_{A,B}$ and of its restriction to the pairs obtained by symmetric exchanges from $(A, B)$ to prove the following result.

\begin{restatable}{theorem}{gpmain}\label{thm:gp-main}
    Let $\F$ be a field, $n,r \in \Zp$, $A,B \in \tbinom{[n]}{r}$, and $c\colon \cX_{A\setminus B, B\setminus A} \to \F$ be such that
    \begin{align}\label{eq:c}
        \sum_{X \subseteq A\setminus B} c(X, \varphi(X)) = 1
    \end{align}
    holds for every bijection $\varphi\colon A\setminus B \to B\setminus A$.
    Then, for any matrix $K \in \F^{r \times n}$, we have
    \begin{align}\label{eq:gp-general}
        \mu_{A,B}^K(\emptyset, \emptyset) = \sum_{(X,Y) \in \cX_{A\setminus B,B\setminus A}} c(X,Y) \mu_{A,B}^K(X,Y).
    \end{align}
\end{restatable}

For bases $A$ and $B$ of a representable matroid, \zcref{thm:gp-main} immediately implies the existence of $(U,V)$ with $c(U,V) \ne 0$ for which both $A-U+V$ and $B+U-V$ are bases, just as \zcref{thm:multiple} follows from~\eqref{eq:multiple-gp}.
When the ground field $\F$ has characteristic zero, we further show the reconfiguration statement analogous to that of \zcref{thm:reconf-mep}.
The resulting matroidal statement is formally stated as follows.
Here, for a function $c\colon \cX_{A\setminus B, B\setminus A} \to \F$, let $\supp(c) \coloneqq \Set{(X,Y) \in \cX_{A\setminus B, B\setminus A}}{c(X,Y) \ne 0}$.

\begin{restatable}{theorem}{gpmatroids}\label{thm:gp-matroids}
    Let $A$ and $B$ be bases of a matroid $\bM$ representable over a field $\F$.
    Let $c\colon \cX_{A\setminus B, B\setminus A} \to \F$ satisfy~\eqref{eq:c} for any bijection $\varphi\colon A\setminus B \to B\setminus A$.
    Then, there exists $(U,V) \in \supp(c)$ such that $A-U+V$ and $B+U-V$ are bases of $\bM$.
    Moreover, if $\F$ has characteristic zero, then $(U,V)$ can be chosen so that $(A-U+V, B+U-V)$ is obtained from $(A,B)$ using symmetric exchanges in $\bM$.
\end{restatable}

Our framework yields several consequences.
We first discuss how it recovers exchange properties which are known or proved in this paper, namely \zcref{thm:multiple,thm:equitability-exchange,thm:main}.
By giving explicit choices of the function $c$ with~\eqref{eq:c}, we show that \zcref{thm:gp-matroids} recovers \zcref{thm:equitability-exchange} for matroids representable over a field of characteristic zero, and \zcref{thm:multiple,thm:main} for representable matroids over fields of arbitrary characteristic, though in case of \zcref{thm:main}, with a weaker upper bound on the common size of $U$ and $V$.
When $\chr(\F)=0$, these results also yield their reconfiguration versions, including \zcref{thm:reconf-mep}.

Beyond these exchange properties, we further prove several exchange properties for representable matroids.
Our first highlight is the following.
For bases $A$ and $B$ of a matroid $\bM$ representable over a field $\F$, let $\cU_{A,B}^{\bM}(\F)$ denote the collection of pairs $(U,V)\in \cX_{A\setminus B,B\setminus A}$ such that both $A-U+V$ and $B+U-V$ are bases of $\bM$, and, in addition, if $\chr(\F)=0$, then $(A-U+V,B+U-V)$ is reachable from $(A,B)$ by a sequence of symmetric exchanges in $\bM$.

\begin{restatable}{theorem}{twomax}\label{thm:two-max}
    Let $A$ and $B$ be bases of a matroid $\bM$ representable over a field $\F$, $X\subseteq A\setminus B$ a subset, $(Y_1, Y_2)$ a bipartition of $B\setminus A$, and $p_1, p_2\in \Z_{\ge 0}$ with $p_t \le \max\set{|X|-|Y_t|, 0}$ for $t \in [2]$.
    If $p_1=p_2=0$ or $\F$ has characteristic zero, then $|X| \le m_1+m_2$, where
    \begin{align}\label{def:m-t}
        m_t \coloneqq \max\Set{|V \cap Y_t|}{(U,V) \in \cU_{A,B}^\bM(\F),\, U \subseteq X,\, |V\cap Y_{3-t}| = p_t}
    \end{align}
    for $t \in [2]$.
\end{restatable}

\zcref{thm:two-max} reveals a trade-off between two different exchange properties, ensuring that at least one of them admits large exchange sets.
This type of exchange guarantee appears to be new, even for representable matroids.

We show that \zcref{thm:two-max} implies statement \ref{it:subsets_with_ranks}  of the following theorem. Its statement \ref{it:supersets_with_ranks} is obtained by a change of variables, combining with the reachability from $(A,B)$ to $(B,A)$ implied by \zcref{thm:reconf-mep} for the $\chr(\F)=0$ case.
These exchange properties are illustrated in \zcref{fig:with_ranks}. 

\begin{figure}
    \centering
    \begin{subfigure}{0.5\linewidth}
        \centering
        \begin{tikzpicture}[
          font=\small,
          line width=0.7pt,
          line cap=round,
          line join=round,
          brace above/.style={decorate,decoration={brace,amplitude=5pt}},
          brace below/.style={decorate,decoration={brace,mirror,amplitude=5pt}},
          myarrow/.style={<->, >=Stealth, line width=0.7pt}
        ]
          \def\W{5}
          \def\H{0.35}
          \def\WX{3.5}
          \def\WYone{2.9}
          \def\WU{2}
          \def\WVone{1.5}
          \def\WVtwo{0.5}
          \def\Boriginy{-2}
        
          \node[left=12pt] at (0, 0.5*\H) {$A\setminus B$};
          \draw (0,0) rectangle (\W, \H);
          \draw (\WX, 0) -- (\WX, \H);
          \draw[brace above] (0.05, \H+0.1) -- node[midway,above=6pt] {$X_1$} (\WX-0.05, \H+0.1);
          \draw[brace above] (\WX+0.05, \H+0.1) -- node[midway,above=6pt] {$X_2$} (\W-0.05, \H+0.1);
          \fill[gray!25] (0.07, 0.07) rectangle (0.07+\WU, \H-0.07);
        
          \node[left=12pt] at (0, \Boriginy+0.5*\H) {$B\setminus A$};
          \draw (0, \Boriginy) rectangle (\W, \Boriginy+\H);
          \draw (\WYone, \Boriginy) -- (\WYone, \Boriginy+\H);
          
          \draw[brace below] (0.05, \Boriginy-0.10) -- node[midway,below=6pt] {$Y_1$} (\WYone-0.05, \Boriginy-0.10);
          \draw[brace below] (\WYone+0.05, \Boriginy-0.10) -- node[midway,below=6pt] {$Y_2$} (\W-0.05, \Boriginy-0.10);
          
          \fill[gray!25] (0.07, \Boriginy+0.07) rectangle (\WVone+0.07, \Boriginy+\H-0.07);
          \draw[myarrow] (0.07, \Boriginy+\H+0.15) -- (\WVone+0.07, \Boriginy+\H+0.15) node[midway, above, align=right] {$\ge\max\{|A\setminus B|-r(X_2+Y_2),\qquad\quad\;$\\$|X_1+Y_1|-r(X_1+Y_1)\}\qquad$};
          
          \fill[gray!25] (\W-0.07-\WVtwo, \Boriginy+0.07) rectangle (\W-0.07, \Boriginy+\H-0.07);
          \draw[myarrow] (\W-0.07-\WVtwo, \Boriginy+\H+0.15) -- (\W-0.07, \Boriginy+\H+0.15) node[midway, above] {$p$};
        \end{tikzpicture}
        \caption{\zcref{thm:with_ranks}~\ref{it:subsets_with_ranks}}
    \end{subfigure}%
    \begin{subfigure}{0.5\linewidth}
        \centering
        \begin{tikzpicture}[
          font=\small,
          line width=0.7pt,
          line cap=round,
          line join=round,
          brace above/.style={decorate,decoration={brace,amplitude=5pt}},
          brace below/.style={decorate,decoration={brace,mirror,amplitude=5pt}},
          myarrow/.style={<->, >=Stealth, line width=0.7pt}
        ]
          \def\W{5}
          \def\H{0.35}
          \def\WX{2.1}
          \def\WYone{1.5}
          \def\WUone{1.6}
          \def\WUtwo{1.4}
          \def\WVtwo{1.5}
          \def\Boriginy{-2}
        
          \node[left=12pt] at (0, 0.5*\H) {$A\setminus B$};
          \draw (0,0) rectangle (\W, \H);
          \draw (\WX, 0) -- (\WX, \H);
          \draw[brace above] (0.05, \H+0.1) -- node[midway,above=6pt] {$X_1$} (\WX-0.05, \H+0.1);
          \draw[brace above] (\WX+0.05, \H+0.1) -- node[midway,above=6pt] {$X_2$} (\W-0.05, \H+0.1);
          \fill[gray!25] (0.07, 0.07) rectangle (0.07+\WUone, \H-0.07);
          \fill[gray!25] (\WX+0.07, 0.07) rectangle (\WX+\WUtwo-0.07, \H-0.07);
          \draw[myarrow] (\WUone, -\H+0.15) -- (\WX, -\H+0.15) node[midway, below] {$p$};
        
          \node[left=12pt] at (0, \Boriginy+0.5*\H) {$B\setminus A$};
          \draw (0, \Boriginy) rectangle (\W, \Boriginy+\H);
          \draw (\WYone, \Boriginy) -- (\WYone, \Boriginy+\H);
          
          \draw[brace below] (0.05, \Boriginy-0.10) -- node[midway,below=6pt] {$Y_1$} (\WYone-0.05, \Boriginy-0.10);
          \draw[brace below] (\WYone+0.05, \Boriginy-0.10) -- node[midway,below=6pt] {$Y_2$} (\W-0.05, \Boriginy-0.10);
          
          \fill[gray!25] (0.07, \Boriginy+0.07) rectangle (\WYone-0.07, \Boriginy+\H-0.07);
          \fill[gray!25] (\WYone+0.07, \Boriginy+0.07) rectangle (\WYone+0.07+\WVtwo, \Boriginy+\H-0.07);
          \draw[myarrow] (0.07, \Boriginy+\H+0.15) -- (\WYone+0.07+\WVtwo, \Boriginy+\H+0.15) node[midway, above, align=right] {$\le\min\{r(X_1+Y_1),\qquad\qquad\qquad\qquad\qquad\quad\;\;$\\$r(X_2+Y_2)-|X_2+Y_2|+|A\setminus B|\} - p$};
        \end{tikzpicture}
        \caption{\zcref{thm:with_ranks}~\ref{it:supersets_with_ranks}}
    \end{subfigure}
    \caption{Illustration of \zcref{thm:with_ranks}. Gray areas in $A\setminus B$ and $B\setminus A$ indicate $U$ and $V$, respectively.}\label{fig:with_ranks}
\end{figure}

\begin{restatable}{theorem}{withranks}\label{thm:with_ranks}
    Let $A$ and $B$ be bases of a matroid $\bM$ representable over a field $\F$ with rank function $r$, and $(X_1, X_2)$ and $(Y_1, Y_2)$ be bipartitions of $A\setminus B$ and $B\setminus A$, respectively.
    Let $p$ be an integer with $0 \le p \le \max\set{|X_1|-|Y_1|, 0}$ such that $p = 0$ or $\F$ has characteristic zero.
    \begin{enumerate}[label={\upshape(\alph*)}]
        \item \label{it:subsets_with_ranks} There exists $(U,V) \in \cU_{A,B}^{\bM}(\F)$ such that $U \subseteq X_1$, $|Y_2 \cap V| = p$, and
        \[
            |U|=|V| \ge \max\set[\big]{|A\setminus B|-r(X_2+Y_2),\ |X_1+Y_1|-r(X_1+Y_1)} + p.
        \]
        \item \label{it:supersets_with_ranks} There exists $(U,V) \in \cU_{A,B}^{\bM}(\F)$ such that $Y_1 \subseteq V$, $|X_1 \setminus U| = p$, and
        \[
            |U| = |V| \le \min\set[\big]{r(X_1+Y_1),\ r(X_2+Y_2)-|X_2+Y_2|+|A\setminus B|} - p.
        \]
    \end{enumerate}
\end{restatable}

\zcref{thm:with_ranks} generalizes both \zcref{thm:equitability-exchange} for matroids representable over a field of characteristic zero and \zcref{thm:main} for matroids representable over a field of any characteristic, together with their reconfiguration versions in the characteristic-zero case.
Indeed, if $|X_1| > |Y_1|$, then statement~\ref{it:subsets_with_ranks} for $p=1$ specializes to \zcref{thm:equitability-exchange} applied to $X \coloneqq X_1$ and the bipartition $(Y_1, Y_2)$ of $B \setminus A$, with an additional lower bound on the size of the exchange sets.
Statement~\ref{it:supersets_with_ranks} for $p=0$ gives \zcref{thm:main} applied to $X \coloneqq X_1$ and $Y \coloneqq Y_1$, where the improved upper bound on the sizes is the same as applying \zcref{thm:main} both to the matroid and its dual.

On the algorithmic side, via a reduction to linear matroid intersection with multiple exact weight constraints~\cite{camerini1992random,horsch2024problems}, we provide a randomized algorithm that finds, in expected polynomial time, exchangeable pairs $(U,V)$ witnessing the inequality of \zcref{thm:two-max}, and thus also the exchange properties of \zcref{thm:with_ranks}, except for their reconfiguration versions; designing algorithms for the reconfiguration versions is left open.

\begin{restatable}{theorem}{twomaxalgo}\label{thm:two-max-algo}
    Let $\bM$ be a matroid representable over a field $\F$.
    There is a randomized algorithm that, given a matrix $K \in \F^{r \times n}$ representing $\bM$, bases $A$ and $B$ of $\bM$, a subset $X \subseteq A\setminus B$, a bipartition $(Y_1, Y_2)$ of $B\setminus A$, and $p_1, p_2 \in \Z_{\ge 0}$ with $p_t \le \max\set{|X|-|Y_t|, 0}$ for $t \in [2]$ such that $p_1 = p_2 = 0$ or $\F$ has characteristic zero, outputs $(U_1, V_1), (U_2, V_2) \in \cX_{X, B\setminus A}$ such that $|V_t \cap Y_{3-t}| = p_t$ for $t \in [2]$, $|X| \le |V_1 \cap Y_1| + |V_2 \cap Y_2|$, and $A-U_t+V_t$ and $B+U_t-V_t$ are bases for $t \in [2]$ in polynomial time in expectation, assuming unit-time field operations.
\end{restatable}

Seeking more general exchange properties, using our framework, we further establish a generalization (\zcref{thm:gp-partition}) of \zcref{thm:two-max}, which allows partitioning $A\setminus B$ and $B\setminus A$ into a general fixed number of parts.
This general exchange property still admits a randomized polynomial-time algorithm and generates another common extension of \zcref{thm:main,thm:equitability-exchange} of the characteristic zero case (\zcref{thm:blue}).
Although these exchange properties do not yet yield direct applications, they unify several previously separate results under a common perspective and suggest unified conjectures for general matroids.

As another application of our framework, we also show a ``uniqueness version'' of \zcref{thm:main} for the characteristic zero case from the identity~\eqref{eq:gp-general} corresponding to \zcref{thm:main} with a weaker upper bound on the common size of $U$ and $V$ (\zcref{thm:main-unique-czero}).
This theorem asserts that $A$ and $B$ are bases if there exists a unique exchangeable pair $(U,V)$ with the conditions in \zcref{thm:main} under the weaker upper bound.
This relation between an exchange property and its uniqueness version is similar to that of the so-called perfect-matching and unique-matching lemmas~\cite{murota2010matrices} used in matroid intersection algorithms; see \zcref{sec:uniqueness} for further background of uniqueness versions.
We further observe that extensions of the Grassmann--Plücker identity generated by our framework yield exchange properties for representable valuated matroids, which are a quantitative generalization of representable matroids.

\subsection{Organization}
The rest of this paper is organized as follows.
\zcref{sec:preliminaries} provides preliminaries on matroids.
\zcref{sec:main} proves \zcref{thm:main} and derives consequences for existing exchange properties.
\zcref{sec:robustness} establishes the robustness property of local search based on \zcref{thm:main}.
\zcref{sec:harmonic} introduces harmonic functions and proves \zcref{thm:reconf-mep} and its consequences.
\zcref{sec:gp} develops the Grassmann--Plücker framework further and applies it to obtain additional exchange properties.
Finally, \zcref{sec:conclusion} concludes with open problems.

\section{Preliminaries}\label{sec:preliminaries}

\subsection{Notations}

Let $\N, \Z_{\ge 0}, \Z, \Q, \R$ and $\R_{\ge0}$ be the sets of natural numbers, nonnegative integers, integers, rationals, reals, and nonnegative reals, respectively.
For $n \in \Z_{\ge 0}$, let $[n] \coloneqq \set{1, \dotsc, n}$.
For a set $E$ and $k \in \Z$, let $\binom{E}{k} \coloneqq \Set{S \subseteq E}{|S|=k}$.
For $S \subseteq E$, let $\ones_S\colon E \to \set{0,1}$ denote the characteristic function of $S$ defined by $\ones_S(e) = 1$ if $e \in S$ and $0$ otherwise.
For $w\colon E \to \R$, let $w(S) \coloneqq \sum_{e \in S} w(e)$.
For a function $f\colon E \to \F$ from a set $E$ to a field $\F$, let $\supp(f) \coloneqq \Set{e \in E}{f(e) \ne 0}$.

\subsection{Matroids}
We collect basic notions of matroids used in this paper.
For further background on matroids, we refer readers to Oxley~\cite{oxley2011matroid}.

Let $\bM$ be a matroid with ground set $E \coloneqq E(\bM)$ and basis family $\cB \coloneqq \cB(\bM)$.
The \emph{rank function} $r\colon 2^E \to \Z_{\ge 0}$ of $\bM$ is defined as $r(S) \coloneqq \max\Set{|S \cap B|}{B \in \cB}$ for $S \subseteq E$.
The rank function is submodular; that is, $r(S)+r(T) \ge r(S \cap T) + r(S \cup T)$ holds for $S, T \subseteq E$.
The \emph{rank} of $\bM$ refers to $r(E)$.
A set $S \subseteq E$ is called \emph{independent} if $r(S) = |S|$, i.e., $S \subseteq B$ for some $B \in \cB$, and \emph{dependent} otherwise.
Likewise, we call a set $S \subseteq E$ \emph{spanning} if $r(S) = r(E)$, i.e., $B \subseteq S$ for some $B \in \cB$, and \emph{non-spanning}
otherwise.

The \emph{dual} of a matroid $\bM$ is a matroid $\bM^*$ with the same ground set $E \coloneqq E(\bM^*) = E(\bM)$ and the basis family $\cB(\bM^*) = \Set{E-B}{B \in \cB(\bM)}$.
If $r$ is the rank function of $\bM$, the rank function $r^*$ of $\bM^*$ is given as $r^*(S) = r(E-S)+|S|-r(E)$ for $S \subseteq E$.
For subsets $F_1, F_2\subseteq E$ with $F_1\cap F_2=\emptyset$ and $F_1\cup F_2\ne E$, the matroid $\bM/F_1\backslash F_2$  obtained from $\bM$ by \emph{contracting} $F_1$ and \emph{deleting} $F_2$ is the matroid on the ground set $E\setminus (F_1\cup F_2)$ with the rank function $r_{\bM/F_1\backslash F_2}(S)=r(S\cup F_1)-r(F_1)$ for $S\subseteq E\setminus (F_1\cup F_2)$.
Contraction and deletion are dual operations of each other; that is, ${(\bM/F_1\backslash F_2)}^* = \bM^* / F_2 \backslash F_1$ holds.

A matroid $\bM$ is said to be \emph{representable} over a field $\F$ if there exists a matrix $K \in \F^{r \times n}$ such that $\cB(\bM) = \Set{B \in \binom{[n]}{r}}{ \det K[B] \ne 0}$, where $E(\bM)$ is identified with $[n]$.
The representability of a matroid is preserved under taking the dual, contraction, and deletion.
Given a matrix representing $\bM$, we can compute representations of the dual, contraction, and deletion in polynomial time using Gaussian elimination.

The following theorem essentially states that the greedy algorithm works for linear optimization over matroid bases.

\begin{theorem}[{see~\cite[Theorem~40.2]{schrijver2003combinatorial}}]\label{thm:greedy}
    Let $\bM$ be a matroid with ground set $E = E(\bM)$ and rank function $r$.
    Let $c\colon E \to \R$ be a weight function, $e_1, \dots, e_n$ any ordering of $E$ with $n \coloneqq |E|$ such that $c(e_1) \le \dots \le c(e_n)$, and $S_i \coloneqq \set{e_1,\dots, e_i}$ for $0 \le i \le n$. Then,
    \[
        \min_{B \in \cB(\bM)} c(B) = \sum_{i=1}^n (r(S_i)-r(S_{i-1})) c(e_i).
    \]
\end{theorem}

The following is a min-max relation for weighted matroid intersection, called the \emph{weight splitting theorem}.

\begin{theorem}[Frank~\cite{frank1981weighted}]\label{thm:weight-splitting}
    Let $\bM_1$ and $\bM_2$ be matroids on the common ground set $E \coloneqq E(\bM_1) = E(\bM_2)$ and let $w\colon E \to \Z$ be an integral weight function. 
    Then,
    \begin{align}
        \min_{B \in \cB(\bM_1) \cap \cB(\bM_2)} w(B) = \max\Set*{\min_{B_1 \in \cB(\bM_1)} c(B_1) + \min_{B_2 \in \cB(\bM_2)} d(B_2)}{c, d\colon E \to \Z, w=c+d}.
    \end{align}
\end{theorem}

\subsection{Randomized Algorithm for Exact-weight Common Basis of Representable Matroids}

For a nonnegative integral weight function $w\colon E \to \Zp$ and a set of feasible weights $\cT\subseteq\Zp$, an \emph{exact weight constraint} refers to the condition that a set $S\subseteq E$ satisfies $w(S) \in \cT$.
If two representable matroids $\bM_1$ and $\bM_2$ having a common ground set $E$ are given as explicit matrix representations, it is possible to find their common basis $B$ subject to a fixed number of exact weight constraints in randomized pseudo-polynomial time~\cite{camerini1992random,horsch2024problems}.
We explain it through the formulation as \emph{group-labeled matroid intersection}.

Let $\psi\colon E \to \Gamma$ be a labeling from $E$ to an abelian group $\Gamma$ and $\mathcal{T} \subseteq \Gamma$.
The group-labeled matroid intersection problem asks for a common basis $B$ such that $\psi(B) \coloneqq \sum_{e \in B} \psi(e) \in \mathcal{T}$.
If the input matroids are representable and given by explicit matrix representations, the problem admits a polynomial-time \emph{Monte Carlo} algorithm via an algebraic technique: the algorithm runs in polynomial time in the worst case, and the error is one-sided.

\begin{theorem}[{\cite{horsch2024problems}}]\label{thm:group-label-algo}
    Let $\bM_1$ and $\bM_2$ be matroids representable over a field $\F$ with a common ground set $E$.
    There is a randomized algorithm that, given matrices $K_1$ and $K_2$ representing $\bM_1$ and $\bM_2$, respectively, a labeling $\psi \colon E \to \Gamma$ into an abelian group $\Gamma$, and a subset $\cT \subseteq \Gamma$, returns a common basis $B$ satisfying $\psi(B) \in \cT$ with probability at least $\frac12$ if such a common basis exists.
    The running time is polynomial in $|E|$ and $|\Gamma|$, assuming unit-time operations in $\F$ and $\Gamma$.\footnote{\cite{horsch2024problems} claims a stronger, expected polynomial-time \emph{Las Vegas} guarantee, i.e., the output is always correct while the running time is polynomial only in expectation.
    This would require a polynomially-sized certificate for NO instances computable in expected polynomial time, which is not known for matroid intersection with group-label constraints; thus, the claim should be weakened to a Monte Carlo guarantee.
    }
\end{theorem}

\begin{proposition}[{see~\cite{camerini1992random,horsch2024problems}}]\label{prop:exact}
    Let $\bM_1$ and $\bM_2$ be matroids of rank $r$ representable over a field $\F$ with a common ground set $E$, and let $k \in \Zp$.
    There is a randomized algorithm that, given matrices $K_1$ and $K_2$ representing $\bM_1$ and $\bM_2$, respectively, weight functions $w_i \colon E \to \set{0,1,\dotsc,W}$ $(i\in [k])$ for some $W \in \Zp$, and feasible weight sets $\cT \subseteq \set{0,1,\dotsc,rW}^k$, returns a common basis $B$ satisfying $(w_1(B), \dotsc, w_k(B)) \in \cT$ with probability at least $\frac12$ if such a common basis exists.
    The running time is polynomial in $|E|$ and ${(rW+1)}^k$, assuming unit-time operations in $\F$.
\end{proposition}

\begin{proof}
    Set $m \coloneqq rW+1$, and let $\Gamma \coloneqq (\Z/m\Z)^k$.
    Let $\pi \colon \Z^k \to \Gamma$ be the map defined by $\pi(a_1, \dotsc, a_k) \coloneqq (a_1 \bmod m, \dotsc, a_k \bmod m)$ for $(a_1, \dotsc, a_k) \in \Z^k$, which is a group homomorphism whose restriction to $\set{0,1,\dotsc,rW}^k$ is a bijection onto $\Gamma$.
    Define $\psi\colon E\to\Gamma$ by $\psi(e) \coloneqq \pi(w_1(e), \dotsc,w_k(e))$, and let $\cT' \subseteq \Gamma$ be the image of $\cT$ under $\pi$.
    For every common basis $B$, we have $\psi(B) = \pi(w_1(B), \dotsc, w_k(B))$ since $\pi$ is a homomorphism, and $0\le w_i(B)\le rW$ for $i\in[k]$ since $|B| = r$; hence, as $\pi$ is injective on $\set{0,1,\dotsc,rW}^k \supseteq \cT$, the condition $\psi(B) \in \cT'$ is equivalent to $(w_1(B), \dotsc, w_k(B)) \in \cT$.
    Thus, the claim follows by applying \zcref{thm:group-label-algo} with $\Gamma$, $\psi$, and $\cT'$, whose running time is polynomial in $|E|$ and $|\Gamma| = {(rW+1)}^k$.
\end{proof}

\section{Generalizing the Multiple Exchange Property via Weighted Matroid Intersection}\label{sec:main}

\subsection{Proof of \texorpdfstring{\zcref{thm:main}}{Theorem~\ref{thm:main}}}\label{sec:proof-main}

This section is devoted to showing \zcref{thm:main}.
We restate it here for readability.

\main*

\begin{proof}    
    Let $E \coloneqq E(\bM)$ denote the ground set of $\bM$.
    By contracting $A\cap B$ and deleting $E-(A\cup B)$, we may assume that $A\cap B = \emptyset$ and $A\cup B = E$.
    Using $A' \coloneqq A-U+(V-Y)$ and $B' \coloneqq B+(U-X)-V$, the statement of the theorem is equivalent to the existence of a bipartition $(A', B')$ of $E' \coloneqq E-(X+Y)$ such that $B'+X$ and $A'+Y$ are bases and $|B' \cap A| \le r(X+Y)-|X|$.
    Here, $B'+X$ being a basis of $\bM$ is equivalent to $B'$ being a basis of $\bM/X\backslash Y$, while $A'+Y = (E'-B')+Y$ being a basis of $\bM$ is equivalent to $B'$ being a basis of ${(\bM/Y\backslash X)}^* = \bM^*/X\backslash Y$. 
    Therefore, the statement of \zcref{thm:main}~(i) is equivalent to
    \begin{align}\label{eq:goal}
        \opt \coloneqq \min\Set{\ones_{A-X}(B')}{B' \in \cB(\bM/X\backslash Y) \cap \cB(\bM^*/X\backslash Y)} \le r(X+Y)-|X|.
    \end{align}
    Once~\eqref{eq:goal} is proved, \zcref{thm:main}~(ii) also follows since a minimum-weight common basis $B'$ obtained by any weighted matroid intersection algorithm gives the desired pair $(U,V)$ as $U = ((A\setminus B) \cap B')+X$ and $V = (B\setminus A) \setminus B'$.

    In what follows, we show~\eqref{eq:goal}.
    By \zcref{thm:weight-splitting}, there exist weight functions $c, d \colon E' \to \Z$ such that $\ones_{A-X} = c+d$ and
    \begin{align} \label{eq:weight_splitting}
        \opt = \min_{A' \in \cB(\bM/X\backslash Y)} c(A') + \min_{B' \in \cB(\bM^*/X\backslash Y)} d(B').
    \end{align}
    We may assume that $c\ge 0$ by considering $c-\min_{e \in E'} c(e)\cdot  \ones_{E'}$ and $d+\min_{e \in E'} c(e) \cdot \ones_{E'}$.
    Let $c' \coloneqq c -\frac12 \cdot \ones_{A-X}$ and let $e_1,\dots, e_m$ be an ordering of $E'$ such that $c'(e_1) \le \dots \le c'(e_m)$, where $m \coloneqq |E'|$.
    Since $c$ is integer-valued, this ordering satisfies $c(e_1) \le \dots \le c(e_m)$ and $(c-\ones_{A-X})(e_1) \le \dots \le (c-\ones_{A-X})(e_m)$.
    The latter is equivalent to $d(e_1) \ge \dots \ge d(e_m)$.
    Define $S_i \coloneq \set{e_1,\dots, e_i}$ for $0 \le i \le m$ and $T_j \coloneqq \set{e_m, e_{m-1}, \dots, e_{m-j+1}} = E'-S_{m-j}$ for $0 \le j \le m$. 
    Then, using \eqref{eq:weight_splitting} and \zcref{thm:greedy}, we get
    \begin{align} \label{eq:after_greedy}
        \opt & = \sum_{i=1}^m (r_{\bM/X\backslash Y}(S_i)-r_{\bM/X\backslash Y}(S_{i-1})) \cdot c(e_i)
        + \sum_{j=1}^m (r_{\bM^*/X\backslash Y}(T_j)-r_{\bM^*/X\backslash Y}(T_{j-1})) \cdot d(e_{m+1-j}),
    \end{align}
    where $r_{\bM/X\backslash Y}$ and $r_{\bM^*/X\backslash Y}$ denote the rank functions of $\bM/X\backslash Y$ and $\bM^*/X\backslash Y$, respectively.
    We have $r_{\bM/X\backslash Y}(S_i) = r(S_i+X)-|X|$ for $0 \le i \le m$
    and $r_{\bM^*/X\backslash Y}(T_j) = r_{\bM^*}(T_j +X)-|X| = r(E-(T_j + X))+|T_j + X| -r(E)-|X| = r(S_{m-j} + Y)+|T_j|-r(E)$ for $0 \le j \le m$, thus \eqref{eq:after_greedy} yields
    \begin{align}
        \opt & = \sum_{i=1}^m (r(S_i + X)-r(S_{i-1}+X)) \cdot c(e_i) \\
        & \quad + \sum_{j=1}^m (r(S_{m-j}+Y)-r(S_{m-j+1}+Y)+1)\cdot  d(e_{m+1-j}) \\
        & = \sum_{i=1}^m (r(S_i + X)-r(S_{i-1}+X)) \cdot c(e_i) + \sum_{i=1}^m (r(S_{i-1}+Y)-r(S_i+Y)+1)\cdot d(e_i). \label{eq:with_c_and_d}
    \end{align}
    Let $C\coloneqq \max\set{1, \max_{e \in E'} c(e)}$ and recall that we assumed that $c \ge 0$.
    For $t \in \Z$, define
    $E_t \coloneqq \Set{e \in E'}{c(e)=t}$,  $E_{\le t} \coloneqq \Set{e \in E'}{c(e)\le t}$, $A_t \coloneqq A\cap E_t$, $A_{\le t}\coloneqq A\cap E_{\le t}$, $B_t \coloneqq B \cap E_t$, and $B_{\le t} \coloneqq B\cap E_{\le t}$.
    Observe that
    \begin{align}
        \sum_{i=1}^m (r(S_i+X)-r(S_{i-1}+X))\cdot c(e_i) & = \sum_{t=0}^C (r(E_{\le t}+X)-r(E_{\le t-1}+X))\cdot t \\
        & = \sum_{t=0}^C r(E_{\le t}+X) \cdot t - \sum_{t=-1}^{C-1} r(E_{\le t}+X)\cdot (t+1) \\
        &= r(E_{\le C} + X) \cdot C - \sum_{t=0}^{C-1} r(E_{\le t}+X) \\
        & = r(E)\cdot C - \sum_{t=0}^{C-1} r(E_{\le t}+X). \label{eq:sumX}
    \end{align}
    As $d = \ones_{A-X}-c$, we have
    \begin{align}
        \sum_{i=1}^m (r(S_{i-1}+Y)-r(S_i+Y))\cdot d(e_i) & =\sum_{i=1}^m (r(S_i+Y)-r(S_{i-1}+Y))\cdot c(e_i) \\ 
        & {}+ \sum_{i \in [m] :\: e_i \in A-X} (r(S_{i-1}+Y)-r(S_i+Y)). \label{eq:sumd}
    \end{align}
    Here, similarly to \eqref{eq:sumX} we have
    \begin{align}
        \sum_{i=1}^m (r(S_i+Y)-r(S_{i-1}+Y))\cdot c(e_i) = r(E)\cdot C - \sum_{t=0}^{C-1} r(E_{\le t} + Y). \label{eq:sumY}
    \end{align}
    Since $(c-\ones_{A-X})(e_1)\le \dots \le (c-\ones_{A-X})(e_m)$, for any $0 \le t \le C$, the elements of $A_t$ precede the elements of $B_t$ in the ordering $e_1,\dots, e_m$.
    This implies that
    \begin{align}
        \sum_{i \in [m] :\: e_i \in A-X} (r(S_{i}+Y)-r(S_{i-1}+Y)) & = \sum_{t=0}^C (r(E_{\le t-1}+A_t+Y)-r(E_{\le t-1}+Y)) \\
        & = \sum_{t=0}^C r(E_{\le t-1}+A_t+Y) - \sum_{t=0}^{C-1} r(E_{\le t}+Y) - |Y|. \label{eq:sumAminusX}
    \end{align}
    The equations~\eqref{eq:sumd},~\eqref{eq:sumY}, and~\eqref{eq:sumAminusX} together imply
    \begin{align}
        \sum_{i=1}^m (r(S_{i-1}+Y)-r(S_i+Y))\cdot d(e_i) = r(E)\cdot C - \sum_{t=0}^C r(E_{\le t-1}+A_t+Y) + |Y|. \label{eq:final_d_sum}
    \end{align}
    Finally, we have
    \begin{align}
        \sum_{i=1}^m d(e_i) = \sum_{i = 1}^m (\ones_{A-X}-c)(e_i) = |A-X| - \sum_{i=1}^m c(e_i) = r(E)-|X|-\sum_{t=0}^C |E_t| t. \label{eq:d_weight_sum}
    \end{align}
    Using \eqref{eq:with_c_and_d}, \eqref{eq:sumX}, \eqref{eq:final_d_sum}, and \eqref{eq:d_weight_sum}, we conclude that
    \begin{align}
        \opt & = \left(r(E)\cdot C - \sum_{t=0}^{C-1} r(E_{\le t}+X)\right)+\left(r(E)\cdot C - \sum_{t=0}^C r(E_{\le t-1}+A_t+Y) + |Y|\right) \\ & \quad + \left(r(E)-|X|-\sum_{t=0}^C |E_t| t\right) \\
        & = (2C+1) r(E) -|X+Y|- \sum_{t=0}^{C-1} r(E_{\le t}+X) - \sum_{t=0}^C r(E_{\le t-1}+A_t+Y)-\sum_{t=0}^C |E_t| t. \label{eq:without_c_or_d}
    \end{align}

 Using the submodularity of $r$, we have 
\begin{align}
    \sum_{t=0}^{C-1} (r(E_{\le t} + X) + r(E_{\le t-1}+A_t+X+Y))
    & \ge \sum_{t=0}^{C-1} (r(E_{\le t-1}+A_t+X)+r(E_{\le t}+X+Y)),\label{eq:submod1} \\
    \sum_{t=0}^{C} (r(E_{\le t-1}+A_t+Y)+r(E_{\le t-1} + X+Y)) & \ge  \sum_{t=0}^C (r(E_{\le t-1}+Y)+r(E_{\le t-1}+A_t+X+Y))\label{eq:submod2}.
\end{align}
Summing these two inequalities and subtracting $\sum_{t=0}^{C-1} r(E_{\le t-1}+A_t+X+Y)+\sum_{t=1}^{C} r(E_{\le t-1}+X+Y)$ from both sides, we obtain
\begin{align}
    & \sum_{t=0}^{C-1} r(E_{\le t}+X) + \sum_{t=0}^C r(E_{\le t-1}+A_t+Y) + r(X+Y)\\
    & \ge \sum_{t=0}^{C-1} r(E_{\le t-1}+A_t+X) + \sum_{t=0}^C r(E_{\le t-1}+Y) + r(E_{\le C-1}+A_C+X+Y) \\
    & \ge \sum_{t=0}^{C-1} |A_{\le t}+X| + \sum_{t=0}^C |B_{\le t-1}+Y|+ r(E) \\
    & = C|X| + (C+1)|Y| + r(E) + \sum_{t=0}^{C-1} |E_{\le t}|.\label{eq:after_submod}
\end{align}

As we have 
\begin{align}
    \sum_{t=0}^{C-1} |E_{\le t}| = \sum_{t=0}^C (C-t) |E_t| = C |E'| -\sum_{t=0}^C |E_t| t = C (2r(E)-|X+Y|) - \sum_{t=0}^C |E_t| t, \label{eq:sumElet}
\end{align}
\eqref{eq:after_submod} implies 
\begin{align}
    \sum_{t=0}^{C-1} r(E_{\le t}+X) + \sum_{t=0}^C r(E_{\le t-1}+A_t+Y) \ge (2C+1) r(E)-r(X+Y)+|Y|-\sum_{t=0}^C |E_t|t. \label{eq:sum_ranks}
\end{align}
\eqref{eq:without_c_or_d} and \eqref{eq:sum_ranks} together give 
\begin{align*}
    \opt & \le \prn*{(2C+1) r(E) - |X+Y|-\sum_{t=0}^C |E_t| t} - \prn*{(2C+1) r(E) -r(X+Y)+|Y| -\sum_{t=0}^C |E_t| t} \\
    & = r(X+Y)-|X|,
\end{align*}
finishing the proof of~\eqref{eq:goal}.
\end{proof}

We remark that the upper bound on the sizes of $U$ and $V$ in \zcref{thm:main} can be improved to $r(X+Y+(A \cap B))-|A \cap B|$ if $A$ and $B$ are not disjoint, by considering the contraction of the matroid by $A \cap B$.

\subsection{Relation to Further Exchange Properties}\label{sec:relation}

In this section, we show how \zcref{thm:main} is related to other basis exchange properties from the literature such as the results of Greene~\cite{greene1974another} (\zcref{thm:greene}) and  Kung~\cite{kung1978alternating}.
We further explain how these results can be used to obtain a strengthening of Gabow's result~\cite{gabow1976decomposing}. 

First, we show that \zcref{thm:main} implies the following, seemingly stronger form.
For notational simplicity, we state the result only if the bases are disjoint.
In general, one can derive the analogous statements by contracting the intersection of the two bases and applying the theorem to the matroid thus obtained. 
Observe that statement \ref{it:ultra_kung_b} for $s=r(E(\bM))-r(X_2+Y_2)$ specializes to \zcref{thm:main} applied to the sets $X_2$ and $Y_2$, since the conditions $U\subseteq X_1$, $V\subseteq Y_1$, $r(E(\bM))-r(X_2+Y_2)\le |U|=|V|$, and $A-U+V$ and $B+U-V$ being bases translate to $A-U\supseteq X_2$, $B-V\supseteq Y_2$, $r(X_2+Y_2)\ge |A-U|=|B-V|$, and $B+(A-U)-(B-V)$ and $A-(A-U)+(B-V)$ being bases, respectively.

\begin{theorem} \label{thm:ultra_kung}
    Let $A$ and $B$ be disjoint bases of a matroid $\bM$ with rank function $r$.
    Let $(X_1,X_2)$ and $(Y_1,Y_2)$ be bipartitions of $A$ and $B$, respectively.
    \begin{enumerate}[label=(\alph*)]
        \item \label{it:ultra_kung_a} For any positive integer $s \le |X_1+Y_1|-r(X_1+Y_1)$, there exist subsets $U\subseteq X_1$ and $V\subseteq Y_1$ such that $s \le |U|=|V|\le s+r(X_1+Y_1)-\max\set{|X_1|, |Y_1|}$ and $A-U+V$ and $B+U-V$ are bases.
        \item \label{it:ultra_kung_b} For any positive integer $s \le r(E(\bM))-r(X_2+Y_2)$, there exist subsets $U\subseteq X_1$ and $V\subseteq Y_1$ such that $s \le |U|=|V|\le s+r(X_2+Y_2)-\max\set{|X_2|, |Y_2|}$ and $A-U+V$ and $B+U-V$ are bases.
    \end{enumerate}
\end{theorem}
\begin{proof}
    We may assume that $|X_2|\ge |Y_2|$ and let $E\coloneq E(\bM)$.

    First we show that \ref{it:ultra_kung_b} holds. 
    Let $s$ be a positive integer with $s \le r(E)-r(X_2+Y_2)$ and let $s'\coloneqq r(E)-r(X_2+Y_2)-s$.    
    Since $s' \ge 0$ and  $|X_2|+s' \le |X_2|+r(E)-r(X_2+Y_2) \le r(E)$, there exists a subset $X_2'\subseteq A$ with $X_2\subseteq X'_2$ and $|X'_2|=|X_2|+s'$.
    By \zcref{thm:main} applied to the bases $A$ and $B$ of the matroid $\bM$ and the subsets $X'_2\subseteq A$ and $Y_2\subseteq B$, there exist subsets $U'\subseteq A$ and $V'\subseteq B$ such that $X'_2\subseteq U'$, $Y_2\subseteq V'$, $A-U'+V'$ and $B+U'-V'$ are bases, and 
    \begin{align}
        |U'| =|V'| \le r(X'_2+Y_2) \le r(X_2+Y_2)+s'.  
    \end{align}
    Let $U\coloneq A-U'$ and $V\coloneqq B-V'$.
    Then, $A-U+V=B+U'-V'$ and $B+U-V=A-U'+V'$ are bases and
    \begin{align}
        |U|=|V| & \ge r(E)-r(X_2+Y_2)-s' = s. 
    \end{align}
    Since $|U|=|V|\le |A-X'_2| = r(E)-|X_2|-s' = s+r(X_2+Y_2)-|X_2|$, \ref{it:ultra_kung_b} follows.

    To show \ref{it:ultra_kung_a}, we apply statement \ref{it:ultra_kung_b} to the dual matroid $\bM^*$.
    Denote the rank function of $\bM^*$ by $r^*$, and note that 
    \begin{align}
        s & \le |X_1+Y_1|-r(X_1+Y_1) \\
        & = |X_1+Y_1|-r^*(X_2+Y_2)+|X_2+Y_2|-r(E) \\
        &= r^*(E)-r^*(X_2+Y_2).
    \end{align}
    Then, \ref{it:ultra_kung_b} implies that there exist subsets $U\subseteq X_1$ and $V\subseteq Y_1$ such that $A-U+V$ and $B+U-V$ are bases of $\bM^*$ and 
    \begin{align}
        s \le |U|=|V| & \le s+r^*(X_2+Y_2) - |X_2| = s + r(X_1+Y_1)-|Y_1|,
    \end{align}
    finishing the proof.
\end{proof}

By choosing $s=1$ in \zcref{thm:ultra_kung}\ref{it:ultra_kung_a} and \ref{it:ultra_kung_b}, we get the following result.
By omitting the upper bound on $|U|=|V|$, the statement is equivalent to the result of Kung~\cite[Theorem~4]{kung1978alternating}.
Nevertheless, it is not difficult to derive this stronger variant from Kung's result: e.g.\ to show \ref{it:kung_a} if $|X_1|\ge |Y_1|$, one can choose a smallest $Y'_1\subseteq Y_1$ such that $X_1+Y'_1$ is dependent.

\begin{corollary}[see Kung~\cite{kung1978alternating}] \label{cor:stronger_kung}
    Let $A$ and $B$ be disjoint bases of a matroid $\bM$ with rank function $r$.
    Let $(X_1,X_2)$ and $(Y_1,Y_2)$ be bipartitions of $A$ and $B$, respectively.
    \begin{enumerate}[label=(\alph*)]
        \item \label{it:kung_a} If $X_1+Y_1$ is dependent, then there exist subsets $U\subseteq X_1$ and $V\subseteq Y_1$ such that $1\le |U|=|V|\le r(X_1+ Y_1)-\max\set{|X_1|,|Y_1|}+1$ and $A-U+V$ and $B+U-V$ are bases.
        \item \label{it:kung_b} If $X_2+Y_2$ is non-spanning, then there exist subsets $U\subseteq X_1$ and $V\subseteq Y_1$ such that $1\le |U|=|V|\le r(X_2+Y_2)-\max\set{|X_2|,|Y_2|}+1$ and $A-U+V$ and $B+U-V$ are bases.
    \end{enumerate}
\end{corollary}

The following is an immediate consequence of \zcref{cor:stronger_kung}\ref{it:kung_b} by using the inequality $r(X_2+Y_2)\le |X_2+Y_2|$ and contracting $A\cap B$.
The statement without the upper bound on $|U|=|V|$ was first shown by Greene (\zcref{thm:greene}).
Nevertheless, one can also derive the form with $|U|=|V|\le |X_2|$ or $|U|=|V|\le |Y_2|$ from \zcref{thm:greene} by the appropriate change of the bipartitions.

\begin{corollary}[see Greene~\cite{greene1974another}] \label{cor:stronger_greene}
    Let $A$ and $B$ be bases of a matroid $\bM$.
    Let $\set{X_1,X_2}$ and $\set{Y_1,Y_2}$ be bipartitions of $A\setminus B$ and $B\setminus A$, respectively.
    If $|X_1+Y_1|\ge |A\setminus B|+1=|B\setminus A|+1$, then there exist subsets $U\subseteq X_1$ and $V\subseteq Y_1$ such that $1\le |U|=|V|\le \min\set{|X_2|,|Y_2|}+1$ and $A-U+V$ and $B+U-V$ are bases.
\end{corollary}

We observe that \zcref{cor:stronger_greene} can be used to show that all matroids satisfy relaxed forms of base orderability and Gabow's conjecture~\cite{gabow1976decomposing}.
A matroid is called \textit{base orderable} if for any bases $A$ and $B$, there exists a bijection $\varphi\colon A\setminus B\to B\setminus A$ such that $A-e+\varphi(e)$ and $B+e-\varphi(e)$ are bases for all $e \in A\setminus B$.
While not all matroids are base orderable, several matroids such as gammoids are base orderable, and this
property has important applications in fair division.
Observe that base orderability corresponds to the existence of subsets $U_1,\dots, U_k$ and $V_1,\dots, V_k$ in \ref{it:bo} with $k=|A\setminus B|=|B\setminus A|$ (and thus $|U_1|=\dots = |U_k|=|V_1|=\dots=|V_k|=1$).
While, as we mentioned, this is not a property of all matroids, Gabow's conjecture~\cite{gabow1976decomposing} says that each matroid satisfies \ref{it:gabow} with $k=|A\setminus B|=|B\setminus A|$.

\begin{corollary} \label{cor:relaxed}
    Let $A$ and $B$ be bases of a matroid, and $k\in \N$ with $2^{k}-1\le |A\setminus B|$.
    \begin{enumerate}[label=\upshape(\alph*)]
        \item \label{it:bo} There exist pairwise disjoint nonempty subsets $U_1,\dots, U_k\subseteq A\setminus B$ and $V_1,\dots, V_k\subseteq B\setminus A$ such that $|U_i|=|V_i|\le 2^{i-1}$ and $A-U_i+V_i$ and $B+U_i-V_i$ are bases for $i \in [k]$. 
        \item \label{it:gabow} There exist nonempty subsets $U_1 \subsetneq \dots \subsetneq U_k \subseteq A\setminus B$ and $V_1 \subsetneq \dots \subsetneq V_k \subseteq B\setminus A$ such that $|U_i|=|V_i|\le 2^{i}-1$ and $A-U_i+V_i$ and $B+U_i-V_i$ are bases for $i \in [k]$. 
    \end{enumerate}
\end{corollary}
\begin{proof}
    First we show \ref{it:bo} by induction on $k$.
    There is nothing to prove if $k=0$.
    Assume that $k\ge 1$ and there exist pairwise disjoint nonempty subsets $U_1,\dots, U_{k-1}\subseteq A\setminus B$ and $V_1,\dots, V_{k-1} \subseteq B\setminus A$ such that $|U_i|=|V_i|\le 2^{i-1}$ and $A-U_i+V_i$ and $B+U_i-V_i$ are bases for $i \in [k-1]$.
    Let $X_2\coloneqq U_1+ \dots + U_{k-1}$, $X_1\coloneqq (A\setminus B)-X_2$, $Y_2\coloneqq V_1+ \dots + V_{k-1}$, and $Y_1\coloneqq (B\setminus A)-Y_2$.
    Then, $|X_2|=|Y_2| \le \sum_{i=1}^{k-1} 2^{i-1} = 2^{k-1}-1$,
    thus 
    $|X_1+Y_1|\ge 2\cdot (|A\setminus B| - 2^{k-1}+1) \ge |A\setminus B|+(2^k-1)-2^k+2> |A\setminus B|$.
    Therefore, \zcref{cor:stronger_greene} implies that there exist nonempty subsets $U_k\subseteq X_1$ and $V_k\subseteq Y_1$ such that $A-U_k+V_k$ and $B+U_k-V_k$ are bases and
    $|U_k|=|V_k| \le \min \set{|X_2|,|Y_2|}+1 \le 2^{k-1}$.
    This finishes the proof of \ref{it:bo}.

    Next we show \ref{it:gabow} by induction on $k$.
    There is nothing to prove if $k=0$.
    Assume that $k\ge 1$ and there exist nonempty subsets $U_1\subsetneq \dots \subsetneq U_{k-1} \subseteq  A\setminus B$ and $V_1\subsetneq \dots \subsetneq V_{k-1}\subseteq B\setminus A$ such that $|U_i|=|V_i|\le 2^{i}-1$ and $A-U_i+V_i$ and $B+U_i-V_i$ are bases for $i \in [k-1]$.
    Let $A'\coloneqq A-U_{k-1}+V_{k-1}$, $B'\coloneqq B+U_{k-1}-V_{k-1}$.
    Since $|(A'-V_{k-1})+(B'-U_{k-1})|  \ge 2\cdot (|A'\setminus B'|-2^{k-1}+1) \ge |A'\setminus B'|+(2^k-1)-2^k+2 > |A'\setminus B'|$, \zcref{cor:stronger_greene} applied to the bipartitions $(A'-V_{k-1}, V_{k-1})$ and $(B'-U_{k-1}, U_{k-1})$ of $A'\setminus B'$ and $B'\setminus A'$, respectively, implies that there exist nonempty subsets $U'_k\subseteq A'-V_{k-1}=A-U_{k-1}$ and $V'_k\subseteq B'-U_{k-1}= B-V_{k-1}$ such that $|U'_k|=|V'_k| \le \min\set{|V_{k-1}|, |U_{k-1}|}+1 \le 2^{k-1}$ and 
    $A'-U'_k+V'_k = A-(U_{k-1}+U'_k)+(V_{k-1}+V'_k)$ and $B'+U'_k-V'_k = B+(U_{k-1}+U'_k)-(V_{k-1}+V'_k)$ are bases. 
    This finishes the proof of \ref{it:gabow} by setting $U_k\coloneqq U_{k-1}+U'_k$ and $V_k\coloneqq V_{k-1}+V'_k$.
\end{proof}

\begin{remark} \label{rem:gabow}
    \zcref{cor:relaxed}\ref{it:gabow} is closely related to the result of
    Gabow~\cite{gabow1976decomposing} who showed that if $X\subseteq A\setminus B$ and $Y\subseteq B\setminus A$ are nonempty subsets such that $A-X+Y$ and $B+X-Y$ are bases, then there exists a positive integer $k$, subsets $\emptyset =U_0\subsetneq U_1\subsetneq \dots \subsetneq U_k\subsetneq U_{k+1} = X$ and $\emptyset =V_0\subsetneq V_1\subsetneq \dots \subsetneq V_k \subsetneq V_{k+1} = Y$, and bijections $\varphi_i^1\colon U_i\setminus U_{i-1}\to V_i\setminus V_{i-1}$ and $\varphi_i^2\colon V_i\setminus V_{i-1}\to U_i\setminus U_{i-1}$ for $i \in [k+1]$ such that for each $i \in [k+1]$, $A-(U_{i-1}+Z)+(V_{i-1}+\varphi_i^1(Z))$ is a basis for each $Z\subseteq U_i\setminus U_{i-1}$ and $B+(U_{i-1}+\varphi^2_i(Z))-(V_{i-1}+Z)$ is a basis for each $Z\subseteq V_i\setminus V_{i-1}$.
    The statement usually referred to as Gabow's conjecture is his question whether $k$ can be chosen to be $|A\setminus B|-1$ if $X=A\setminus B$ and $Y=B\setminus A$.
    In this case, Gabow's result and \zcref{cor:relaxed}\ref{it:gabow} together imply that $k$ can be chosen to be at least $q\coloneqq \lfloor \log_2 |A\setminus B|\rfloor$.
    Indeed, as $2^q-1<|A\setminus B|$, the latter gives subsets $\emptyset \ne U_1 \subsetneq \dots \subsetneq U_q \subsetneq A\setminus B$ and $\emptyset \ne V_1 \subsetneq \dots \subsetneq V_q \subsetneq B\setminus A$ such that $A-U_i+V_i$ and $B+U_i-V_i$ are bases for $i \in [q]$.
    Setting $U_0\coloneqq \emptyset$, $U_{q+1}\coloneq A\setminus B$, $V_0\coloneqq \emptyset$, and $V_{q+1}\coloneqq B\setminus A$, the claim follows by applying Gabow's result to the subsets $U_i\setminus U_{i-1}$ and $V_i\setminus V_{i-1}$ of the bases $A-U_{i-1}+V_{i-1}$ and $B+U_{i-1}-V_{i-1}$ for $i \in [q+1]$.
\end{remark}

\begin{remark}
    \zcref{cor:relaxed} is also related to the results of \cite{horsch2024problems} who showed that for any prime power $q$, there exists a function $f_q\colon \N \to \N$ such that for any $k \in \N$ and any bases $A$ and $B$ of a $\GF(q)$-representable matroid with $|A\setminus B|\ge f_q(k)$, there exist pairwise disjoint subsets $X_1,\dots, X_k\subseteq A\setminus B$ and $Y_1,\dots, Y_k\subseteq B\setminus A$ such that $A-\bigcup_{i \in I} X_i+\bigcup_{i \in I} Y_i$ is a basis for any $I\subseteq [k]$.  
    In general, we are not aware of any counterexample to the stronger version for general matroids that there is a function $f\colon \N \to \N$ such that for any $k\in \N$ and any bases $A$ and $B$ of a matroid with $|A\setminus B|\ge f(k)$, there exist pairwise disjoint subsets $X_1,\dots, X_k\subseteq A\setminus B$ and $Y_1,\dots, Y_k\subseteq B\setminus A$ such that $A-\bigcup_{i \in I} X_i+\bigcup_{i \in I} Y_i$ and $B+\bigcup_{i \in I} X_i-\bigcup_{i \in I} Y_i$ are bases for any $I\subseteq [k]$. 
    If true, this would imply \zcref{cor:relaxed} up to the precise lower bound on $|A\setminus B|$, and it would also imply the polynomial solvability of finding a minimum-weight matroid basis in a group-labeled matroid with a fixed number of forbidden labels~\cite{horsch2024problems}.
\end{remark}

\section{Robustness of Local Search}\label{sec:robustness}
In this section, we prove \zcref{thm:robustness}, which concerns the robustness property of local search, using \zcref{thm:main}, restated below.

\robustness*

For notational simplicity, let $\mathcal{B}_k(B) \coloneqq \Set{B' \in \mathcal{B}}{|B' \setminus B| \le k}$ for $k \in \Z_{\ge 0}$ and $B \in \cB$.
The following lemma implies \zcref{thm:robustness}. 

\begin{lemma} \label{lem:robust}
    For any $B_0 \in \cB$, $B_1 \in \cB_1(B_0)$, and $k \in \N$, we have
    \begin{align} \label{eq:lem-robust}
        \max\Set{w(B)}{B \in \cB_{k-1}(B_1)}
        \ge
        \max\Set{w(B)}{B \in \cB_{k}(B_0)} - \err(B_1 \given B_0).
    \end{align}
\end{lemma}
\begin{proof}
    Take $A \in \argmax\Set{w(B')}{B' \in \mathcal{B}_k(B)}$. Using the definition of $\err(B_1\given B_0)$, we have that~\eqref{eq:lem-robust} is equivalent to
    \begin{align}\label{eq:robust}
        \max\Set{w(B)}{B \in \cB_{k-1}(B_1)} + \max\Set{w(B)}{B \in \cB_1(B_0)} \ge w(A)+w(B_1).
    \end{align}

    We first consider the case where $B_0 = B_1$.
    If $B_0 = B_1 = A$,~\eqref{eq:robust} trivially holds.
    If $B_0 = B_1 \ne A$, then by the symmetric exchange property for $A$ and $B_0=B_1$, there exist elements $a \in A\setminus B_1$ and $b \in B_1\setminus A$ such that $A-a+b$ and $B_1+a-b$ are bases.
    Then, $|(A-a+b)\setminus B_1| = |A\setminus B_0|-1 \le k-1$ and $|(B_1+a-b)\setminus B_0| = 1$, thus~\eqref{eq:robust} holds by \[w(A)+w(B_1) = w(A-a+b)+w(B_1+a-b).\]  
    
    Suppose $B_0 \ne B_1$ and let $x, y$ be the elements such that $B_0 \setminus B_1 = \set{x}$ and $B_1 \setminus B_0 = \set{y}$. 
    Assume first that $y \in A$.
    If $x \not \in A$, then $A \in \cB_{k-1}(B_1)$ and $B_1 \in \cB_1(B_0)$, thus~\eqref{eq:robust} clearly holds.
    Otherwise, if $x \in A$, then by the symmetric exchange property for $B_1$ and $A$, there exists $b \in B_1 \setminus A$ such that $A-x+b$ and $B_1+x-b$ are bases.
    As $(B_1+x-b)\setminus B_0 = \set{y}$, we have $B_1+x-b\in \cB_1(B_0)$.
    As $(A-x+b)\setminus B_1 = (A\setminus B_0)-y$ and $|A\setminus B_0|\le k$, we have $A-x+b\in \cB_{k-1}(B_1)$.
    Therefore,~\eqref{eq:robust} holds by \[w(A)+w(B_1) = w(A-x+b)+w(B_1+x-b).\]

    Assume now that $x \notin A \not\ni y$.
    By the symmetric exchange property for $B_1$ and $A$ again, there exists $a \in A\setminus B_1$ such that $B_1-y+a$ and $A+y-a$ are bases.
    As $(B_1-y+a)\setminus B_0 = \set{a}$, we have $B_1-y+a \in \cB_1(B_0)$.
    As $(A+y-a)\setminus B_1 = (A\setminus B_0)-a$ and $|A\setminus B_0| \le k$, we have $A+y-a \in \cB_{k-1}(B_1)$.
    Therefore,~\eqref{eq:robust} holds by
    \[w(A)+w(B_1) = w(A+y-a)+w(B_1-y+a).\]

    It remains to consider the case when $x \in A \not\ni y$.
    By \zcref{thm:main} applied to $A,B=B_1,X=\set{x}$, and $Y=\set{y}$, there exist $U \subseteq A \setminus B_1$ and $V \subseteq B_1 \setminus A$ such that $x \in U$, $y \in V$, $|U|=|V|\le 2$, and $A-U+V$ and $B_1+U-V$ are bases.

    Suppose $|U|=|V|=1$, i.e., $A-x+y \in \cB$.
    Let $a^*$ be an element of minimum weight in $A \setminus B_0 = (A \setminus B_1)-x$.
    Since $A-x+y$ and $B_1$ are bases and $a^* \in (A-x+y) \setminus B_1$, there exists $b \in B_1 \setminus (A-x+y)$ such that $A-\set{x,a^*}+\set{y,b}$ is a basis.
    As $\set{y, b} \subseteq B_1$ and $y \ne b$, we have $|(A-\set{x,a^*}+\set{y,b})\setminus B_1| \le |A\setminus B_1|-2 \le |A \setminus B_0|-1 \le k-1$, hence $A-\set{x,a^*}+\set{y,b} \in \cB_{k-1}(B_1)$. 
    As $A$ is a basis and $b \in B_1 \setminus (A-x+y) = B_0 \setminus A$, there exists $a \in A \setminus B_0$ such that $B_0-b+a$ is a basis.
    Using that $A-\set{x,a^*}+\set{y,b} \in \cB_{k-1}(B_1)$ and $B_0-b+a \in \cB_1(B_0)$, we get
    \begin{align}
        & \max\Set{w(B)}{B \in \cB_{k-1}(B_1)} + \max\Set{w(B)}{B \in \cB_1(B_0)} \\
        & \quad \ge w(A-\set{x,a^*}+\set{y,b}) + w(B_0-b+a) \\
        & \quad = w(A) + w(B_1) - w(a^*) + w(a) \\
        & \quad \ge w(A) + w(B_1),
    \end{align}
    proving \eqref{eq:robust}.

    Lastly, suppose the case when $|U|=|V|=2$.
    Note that $B_1+U-V = B_0+a-b \in \cB_1(B_0)$, where $a$ and $b$ are the elements such that $U = \set{x,a}$ and $V = \set{y,b}$.
    As $V \subseteq B_1$, we have $|(A-U+V) \setminus B_1| = |A \setminus B_1|-|V| \le k-1$, hence $A-U+V \in \cB_{k-1}(B_1)$.
    Thus, we get
    \begin{align}
        &\max\Set{w(B)}{B \in \cB_{k-1}(B_1)} + \max\Set{w(B)}{B \in \cB_1(B_0)} \\
        &\quad \ge w(A-U+V) + w(B_1+U-V) \\
        &\quad = w(A)+w(B_1),
    \end{align}
    showing~\eqref{eq:robust}.
\end{proof}

\begin{proof}[{Proof of \zcref{thm:robustness}}]
    By repeated applications of \zcref{lem:robust}, we have
    \begin{align}
        w(B_k)
        &= \max\Set{w(B)}{B \in \cB_0(B_k)} \\
        &\ge \max\Set{w(B)}{B \in \cB_1(B_{k-1})} - \err(B_k \given B_{k-1}) \\
        &\ge \max\Set{w(B)}{B \in \cB_2(B_{k-2})} - \err(B_{k-1} \given B_{k-2}) - \err(B_k \given B_{k-1}) \\
        &\ge \dotsb \ge \max\Set{w(B)}{B \in \cB_k(B_0)} - \sum_{i=1}^k \err(B_i \given B_{i-1}),
    \end{align}
    as required.
\end{proof}

\section{Harmonicity of Representable Matroids}\label{sec:harmonic}

In this section, we introduce harmonic functions in \zcref{subsec:harmonic} and show in \zcref{sec:harmonicity-of-products} that the products of determinants associated with pairs of bases are harmonic.
We then exploit this harmonicity to prove a reconfiguration version of the multiple exchange property for representable matroids in \zcref{sec:reconf-mep} and a weighted equitability theorem in \zcref{sec:weighted-equitability}.

\subsection{Harmonic Functions}\label{subsec:harmonic}

This section introduces harmonic functions of set functions.
In the literature, harmonic functions are defined for more general settings in terms of multilinear polynomials; see~\cite{filmus2016an,filmus2019harmonicity} for details.

Let $r \in \Z_{\ge 0}$, $E$ a set with $|E|=2r$, and $(A,B)$ an equal bipartition of $E$.
Recall that $\cX_{A,B}$ denotes $\Set{(X,Y)}{X\subseteq A,\ Y\subseteq B,\ |X|=|Y|}$.
Then, the map $(X,Y) \mapsto A-X+Y$ gives a bijection from  $\cX_{A,B}$ to $\tbinom{E}{r}$, with inverse $S \mapsto (A\setminus S,\ B\cap S)$.
Given a function $h\colon \cX_{A,B} \to \F$ that takes values in a field $\F$, we associate to it the function $\hat h \colon \tbinom{E}{r} \to \F$ defined by
\begin{align}
    \hat h(S)
    \coloneqq
    {(-1)}^{|B \cap S|}h(A \setminus S, B \cap S).
\end{align}

We say that a function $g \colon \binom{E}{r} \to \mathbb{F}$ is
\emph{harmonic} if
\begin{align}\label{def:harmonic}
    \sum_{e \in E \setminus T} g(T+e) = 0
\end{align}
for every $T \in \binom{E}{r-1}$ (see~\cite{filmus2016an,filmus2019harmonicity}).
We also say that a function $h\colon \cX_{A,B} \to \mathbb{F}$ is
\emph{harmonic} if $\hat h$ is harmonic; equivalently,
\begin{align}\label{def:harmonic-change}
    \sum_{a\in X} h(X-a,Y)
    =
    \sum_{b\in B\setminus Y} h(X,Y+b)
\end{align}
for every $X \subseteq A$ and $Y\subseteq B$ such that $|X|=|Y|+1$.

\subsection{Harmonicity of Products of Determinants}\label{sec:harmonicity-of-products}

In this section, we see that a harmonic function arises from certain products of determinants.
This fact is known as the \emph{Plücker relations}; see, e.g.,~\cite{gelfand2008discriminants}.
For completeness, we provide a proof with exact sign computations using our notations.

Let $K \in \F^{r \times n}$  be an $r \times n$ matrix over $\F$.
Recall that $K[A]$ denotes the submatrix of $K$ with columns indexed by $A \subseteq [n]$.
Let $X = \set{x_1, \dotsc, x_k} \subseteq A$ and $Y = \set{y_1, \dotsc, y_k} \subseteq [n] - A$ with $x_1 < \dotsb < x_k$ and $y_1 < \dotsb < y_k$.
Let $K[A; -X, +Y]$ and $K[A; +Y, -X]$ denote the matrix obtained from $K$ by selecting the columns indexed by $A-X+Y$, with column $y_i$ placed in the position of $x_i$ in $A$ for $i \in [k]$.
For $A, B \in \binom{[n]}{r}$, we define $\mu_{A,B}^K \colon \cX_{A \setminus B, B \setminus A} \to \F$ by
\begin{align}
    \mu^K_{A,B}(X, Y) \coloneqq \det K[A; -X, +Y] \det K[B; +X, -Y]
\end{align}
for $(X,Y) \in \cX_{A \setminus B, B \setminus A}$.
Note that $\mu^K_{A,B}(\emptyset, \emptyset) = \mu^K_{A,B}(A\setminus B, B\setminus A) = \det K[A] \det K[B]$, and $\hat \mu_{A,B}^K(S)$ is equal to $\det K[(A \cap B) + S] \det K[(A \cup B) - S]$ up to sign.

\begin{theorem}[{Plücker relations; see~\cite{gelfand2008discriminants}}]\label{thm:gp-harmonic}
    Let $K \in \F^{r \times n}$ be a matrix over a field $\F$ and $A,B \in \binom{[n]}{r}$.
    Then, $\mu^K_{A,B}$ is harmonic.
\end{theorem}

To show \zcref{thm:gp-harmonic}, we use the determinant identity called the \emph{generalized Laplace expansion}.
For $A = \set{a_1, \dotsc, a_r} \subseteq [n]$ with $a_1 < \dotsb < a_r$ and $U = \set{a_{i_1}, \dotsc, a_{i_k}} \subseteq A$ with $a_{i_1} < \dotsb < a_{i_k}$, let $\sgn_A(U) \in \set{+1, -1}$ denote the sign of the permutation
\begin{align}
    \begin{pNiceMatrix}
      a_1 & \Cdots & a_k & a_{k+1} & \Cdots & a_r \\
      a_{i_1} & \Cdots & a_{i_k} & a_{i_{k+1}} & \Cdots & a_{i_r}
    \end{pNiceMatrix},
\end{align}
where $A-U = \set{a_{i_{k+1}}, \dotsc, a_{i_r}}$ with $a_{i_{k+1}} < \dotsb < a_{i_r}$.

\begin{theorem}[{Generalized Laplace expansion; see~\cite[Proposition~2.1.3]{murota2010matrices}}] \label{thm:laplace}
    For a square matrix $L = \binom{L_1}{L_2} \in \F^{m \times m}$ with $L_1 \in \F^{k \times m}$ and $L_2 \in \F^{(m-k) \times m}$, we have
    \begin{align}
        \det L = \sum_{J \in \binom{[m]}{k}} \sgn_{[m]}(J) \det L_1[J] \det L_2[[m]-J].
    \end{align}
\end{theorem}

\begin{proof}[{Proof of \zcref{thm:gp-harmonic}}]
    For $S \subseteq [n]$, define $K_S \in \F^{r \times n}$ as the matrix obtained from $K$ by zeroing out all columns indexed by $[n]-S$.
    Let $X \subseteq A \setminus B$ and $Y \subseteq B \setminus A$ be sets such that $|X|=|Y|+1$.
    Applying the generalized Laplace expansion to the following $2r \times 2r$ matrix
    \begin{align}
        L \coloneqq \begin{pmatrix}
            K[A] & K[B] \\
            K_X[A] & K_{B-Y}[B]
        \end{pmatrix},
    \end{align}
    we obtain
    \begin{align}\label{eq:det-L}
        \det L =
        \sum_{J \in \binom{[2r]}{r}} \sgn_{[2r]}(J)
          \det \prn*{\begin{pmatrix} K[A] & K[B]\end{pmatrix}[J]}
          \det \prn*{\begin{pmatrix} K_X[A] & K_{B-Y}[B]\end{pmatrix}[[2r]-J]}.
    \end{align}
    For each $J$ in~\eqref{eq:det-L}, we let $U \coloneqq \Set{a_i}{i \in [r] \setminus J}$ and $V \coloneqq \Set{b_{j-r}}{j \in J \cap \set{r+1, \dotsc, 2r}}$, where $A = \set{a_1, \dotsc, a_r}$ and $B = \set{b_1, \dotsc, b_r}$.
    Note that $|U| = |V|$ holds.
    Then, we have
    \begin{gather}
        \sgn_{[2r]}(J) = {(-1)}^{|U||V|} \sgn_A(A-U) \sgn_B(V) = {(-1)}^{|U|} \sgn_A(A-U) \sgn_B(V), \\
        \begin{aligned}
        \det \prn*{\begin{pmatrix} K[A] & K[B]\end{pmatrix}[J]}
            &= \begin{cases}
              \det \begin{pmatrix} K[A-U] & K[V]\end{pmatrix} & ((A-U)\cap V = \emptyset), \\
              0 & (\text{otherwise}),
            \end{cases} \\
            \det \prn*{\begin{pmatrix} K_X[A] & K_{B-Y}[B]\end{pmatrix}[[2r]-J]}
            &= \det \begin{pmatrix} K_X[U] & K_{B-Y}[B-V]\end{pmatrix} \\
            &= \begin{cases}
              \det \begin{pmatrix} K[U] & K[B-V] \end{pmatrix} & (U \subseteq X, \, V \supseteq Y), \\
              0 & (\text{otherwise}).
            \end{cases}
        \end{aligned}
    \end{gather}
    We also have
    \begin{align}
        \sgn_A(A-U) \det \begin{pmatrix} K[A-U] & K[V]\end{pmatrix} = \det K[A; -U, +V], \\
        \sgn_B(V) \det \begin{pmatrix} K[U] & K[B-V]\end{pmatrix} = \det K[B; +U, -V].
    \end{align}
    Therefore,
    \begin{align}
        \det L
        &= \sum_{\substack{
            U \subseteq X,\\
            V \subseteq B \setminus A \::\: Y \subseteq V, |U|=|V|
        }} {(-1)}^{|U|} \mu_{A,B}^K(U, V) \\
        &= {(-1)}^{|X|-1} \sum_{a \in X} \mu_{A,B}^K(X-a, Y) + {(-1)}^{|X|} \sum_{b \in (B \setminus A)-Y} \mu_{A,B}^K(X, Y+b).
    \end{align}
    
    Finally, we show $\det L = 0$, completing the proof.
    By elementary row operations, we have
    \begin{align}
        \det L = \det \begin{pmatrix}
            K[A] & K[B] \\
            K_X[A] & K_{B-Y}[B]
        \end{pmatrix}
        = \det \begin{pmatrix}
            K_{A-X}[A] & K_Y[B] \\
            K_X[A] & K_{B-Y}[B] 
        \end{pmatrix}.
    \end{align}
    Since $\begin{pmatrix}
        K_{A-X}[A] & K_Y[B]
    \end{pmatrix}$ contains $|X|+|B-Y| = r+|X|-|Y| > r$ many zero columns, $L$ is singular, as required.
\end{proof}

\subsection{A Reconfiguration Version of the Multiple Exchange Property}\label{sec:reconf-mep}

In this section, we show \zcref{thm:reconf-mep}, extending the multiple exchange property with a statement on the existence of a certain sequence of symmetric exchanges under a characteristic constraint.
For the proof, we first show the following statement on harmonic functions.

\begin{lemma} \label{lem:harmonic-mep}
    Let $E$ be a set, $A, B \subseteq E$ with $|A|=|B|$ and $X\subseteq A\setminus B$.
    Let $\F$ be a field such that $\chr(\F)\not \in \set{2,3,\dots, |X|}$, and let $h\colon \cX_{A\setminus B, B\setminus A} \to \F$ be a harmonic function.
    Then, 
    \[h(\emptyset, \emptyset) = \sum_{Y \in \binom{B\setminus A}{|X|}} h(X,Y).\]
\end{lemma}
\begin{proof}
    We use induction on $k\coloneqq |X|$.
    There is nothing to prove if $k=0$.
    If $k \ge 1$, then by rearranging the sum, applying the harmonicity of $h$, and using the induction hypothesis for some sets of size $k-1$, we get 
    \begin{align}
        k \cdot \sum_{Y \in \binom{B\setminus A}{k}} h(X, Y) & = 
        \sum_{V \in \binom{B\setminus A}{k-1}}\sum_{b \in (B\setminus A) \setminus V} h(X, V+b)  = \sum_{V \in \binom{B\setminus A}{k-1}}\sum_{a \in X} h(X-a, V) \\  &= \sum_{a \in X} \sum_{V \in \binom{B\setminus A}{k-1}} h(X-a, V) = \sum_{a \in X} h(\emptyset, \emptyset) = k \cdot h(\emptyset, \emptyset).
    \end{align}
    Since $\chr(\F)=0$ or $\chr(\F)> k$, the desired equality follows.
\end{proof}

For $r\in \Z_{\ge 0}$ and a set $E$, the \emph{Johnson graph} $J(E,r)$ is defined as the graph on vertices $\binom{E}{r}$ with edge set $\Set{\set{S,S'}}{S, S' \in \binom{E}{r}, |S\setminus S'|=|S'\setminus S|=1}$.
We show the following lemma.

\begin{lemma}\label{lem:component-harmonicity}
    Let $r\in \Z_{\ge 0}$, $E$ a set with $|E|=2r$, $\F$ a field, and $g\colon \binom{E}{r}\to \F$ a harmonic function.
    Let $\cC\subseteq \supp(g)$ be the vertex set of a connected component of the subgraph of $J(E,r)$ induced by $\supp(g)$.
    Then, the function $g_{\cC}\colon \binom{E}{r}\to \F$ defined by
    \begin{align}\label{def:g-C}
        g_{\cC}(S)\coloneqq\begin{cases} g(S) & \text{($S\in \cC$)}, \\ 0 & \text{(otherwise)} \end{cases}
    \end{align}
    is harmonic.
\end{lemma}
\begin{proof}
    Let $T \in \binom{E}{r-1}$.
    If $\Set{T+e}{e\in E\setminus T}$ is disjoint from $\cC$, then $\sum_{e \in E\setminus T} g_\cC(T+e)=0$ as each term is zero.
    Otherwise, since $\Set{T+e}{e \in E\setminus T}$ induces a clique in $J(E, r)$, we have $\Set{T+e}{e \in E\setminus T, T+e\in \supp(g)}\subseteq \cC$, hence $\sum_{e \in E\setminus T} g_\cC(T+e)=\sum_{e \in E\setminus T} g(T+e)=0$.
    This shows the harmonicity of $g_\cC$.
\end{proof}

We are ready to prove \zcref{thm:reconf-mep} which we restate here for convenience.

\reconfmep*

\begin{proof}
    Let $K\in \F^{r \times n}$ be a matrix representing $\bM$.
    Define $h\colon \cX_{A\setminus B, B\setminus A} \to \F$ by $h(X,Y) = \mu^K_{A,B}(X,Y)$ if $(A-X+Y,B+X-Y)$ can be obtained from $(A,B)$ by a sequence of symmetric exchanges in $\bM$, and $h(X,Y) = 0$ otherwise.
    As $h_0\coloneqq \mu^K_{A,B}$ is harmonic by \zcref{thm:gp-harmonic}, $h$ is harmonic as well by \zcref{lem:component-harmonicity} applied to $\hat{h}_0$.
    Then, \zcref{lem:harmonic-mep} shows that $h(\emptyset, \emptyset) = \sum_{Y \in \binom{B\setminus A}{|X|}} h(X,Y)$, hence $h(\emptyset, \emptyset) \ne 0$ implies that $h(X, Y)\ne 0$ for some $Y\in \binom{B\setminus A}{|X|}$, finishing the proof by the definition of $h$.
\end{proof}

\begin{corollary} \label{cor:weak-gabow}
    Let $A$ and $B$ be bases of a matroid $\bM$ representable over a field $\F$ of characteristic zero.
    Then, the basis pair $(B,A)$ can be obtained from $(A,B)$ by a sequence of symmetric exchanges in $\bM$.
\end{corollary}

\subsection{An Application to Weighted Equitability}\label{sec:weighted-equitability}

This section shows \zcref{thm:weighted-equitability}, restated below for convenience.

\weightedequitability*

\begin{proof}
    Assume first that $k=2$ and let $(A,A')$ be a partition of $E$ into two bases.
    We may assume that $w(A)>w(A')$.
    By \zcref{cor:weak-gabow}, there exists a sequence $(A_0, A'_0), (A_1, A'_1), \dots, (A_l, A'_l)$ of basis pairs such that $(A_0, A'_0)=(A,A')$, $(A_l, A'_l) = (A',A)$, and consecutive pairs differ by a symmetric exchange.
    Let $i$ be the smallest index such that $w(A_i) \le w(A'_i)$.
    As $w(A_0)>w(A'_0)$ and $w(A_l)<w(A'_l)$, such an index $i$ exists and $i\ge 1$ holds.
    Let $x$ and $y$ be the unique elements such that $(A_i, A'_i) = (A_{i-1}-x+y, A'_{i-1}+x-y)$.
    If $w(A'_{i-1}) \ge w(A_{i-1}) - \max_{e \in A_{i-1}} w(e)$, then the partition $(A_{i-1}, A'_{i-1})$ satisfies $w(A_{i-1}) \ge w(A'_{i-1}) \ge w(A_{i-1}) - \max_{e \in A_{i-1}} w(e)$.
    Otherwise, we have $w(A'_{i-1}) < w(A_{i-1}) - \max_{e \in A_{i-1}} w(e) \le w(A_{i-1})-w(x)$, hence 
    \[0 \le w(A'_i)-w(A_i) = w(A'_{i-1})-w(A_{i-1})+2(w(x)-w(y)) < w(x)-2w(y) \le w(x) \le \max_{e \in A'_i} w(e).\]
    Therefore, the partition $(A'_i, A_i)$ satisfies the requirements of the theorem.
    
    Consider now the case $k \ge 3$.
    Let $\bar w \coloneqq w(E)/k$ and  $\cA\coloneqq (A_1,\dots, A_k)$ a partition into bases such that $\phi(\cA)\coloneqq \sum^k_{i=1}|w(A_i)-\bar{w}|$ is minimal.
    We may assume that $w(A_1)\ge \dots \ge w(A_k)$ and $\phi(\cA)>0$.
    By applying the previous $k=2$ case to the restriction of $\bM$ to $A_1\cup A_k$, there exists a partition $(\tilde{A}_1, \tilde{A}_k)$  of $A_1 \cup A_k$ such that $w(\tilde{A}_1) \ge w(\tilde{A}_k)\ge w(\tilde{A}_1)-\max_{e \in \tilde{A}_1} w(e)$.
    Let $\tilde{A}_i\coloneqq A_i$ for $i \in \set{2,3,\dots, k-1}$ and $\tilde{\cA}\coloneqq (\tilde{A}_1,\dots, \tilde{A}_k)$.
    If $w(\tilde{A}_1)\ge w(A_1)$, then 
    \[w(\tilde{A}_1)\ge w(A_1)\ge w(A_2)\ge \dots \ge w(A_k) \ge w(\tilde{A}_k) \ge w(\tilde{A}_1)-\max_{e \in \tilde{A}_1} w(e),\]
    hence $\tilde{\cA}$ satisfies the requirements of the theorem.
    Otherwise, we have $w(A_1) > w(\tilde{A}_1)$ and $w(\tilde{A}_k) > w(A_k)$.
    We claim that $\phi(\tilde{\cA}) < \phi(\cA)$, contradicting the choice of $\cA$.
    As $w(A_1)> \bar{w} > w(A_k)$, we have $|w(A_1)-\bar{w}| + |w(A_k)-\bar{w}| = w(A_1)-w(A_k)$.
    If $w(\tilde{A}_1) \ge \bar{w} \ge w(\tilde{A}_k)$, then $|w(\tilde{A}_1)-\bar{w}| + |w(\tilde{A}_k)-\bar{w}| = w(\tilde{A}_1)-w(\tilde{A}_k)< w(A_1)-w(A_k)$, showing $\phi(\tilde{\cA})<\phi(\cA)$.
    Otherwise, we may assume by symmetry that $w(\tilde{A}_1)\ge w(\tilde{A}_k) > \bar{w}$, in which case \[|w(\tilde{A}_1)-\bar{w}|+|w(\tilde{A}_k)-\bar{w}| = w(\tilde{A}_1)+w(\tilde{A}_k)-2\bar{w} = w(A_1)+w(A_k)-2\bar{w} < w(A_1)-w(A_k),\]
    showing $\phi(\tilde{\cA})<\phi(\cA)$.
    This finishes the proof of the existence statement (i).
    
    To give a randomized algorithm in (ii), first we observe that $k=2$ can be assumed.
    Indeed, starting from an arbitrary partition $\cA$ into $k$ bases found by a matroid union algorithm, the previous argument shows that the application of the $k=2$ case either yields a desired partition or a partition $\tilde{\cA}$ with $\phi(\tilde{\cA})<\phi(\cA)$.
    As $\phi(\cA)\le w(E)$ and $k\cdot \phi(\cA) \in \Z$, we need to apply the $k=2$ case at most $k\cdot  w(E)\le {|E|}^2 \cdot W$ times.
    Therefore, we are left to deal with the $k=2$ case.
    For each $e \in E$, the problem of finding a partition $(A_1, A_2)$ of $E$ with $e \in \argmax \Set{w(f)}{f \in A_1}$ and $w(A_1) \ge w(A_2) \ge w(A_1)-w(e)$ reduces to an instance of the problem given in \zcref{prop:exact}.
    Namely, set $F_e\coloneqq \Set{f \in E}{w(f) >w(e)}$ and observe that no desired partition $(A_1, A_2)$ exists if $F_e$ is dependent as $F_e\subseteq A_2$ must hold.
    If $F_e$ is independent, then the problem is equivalent to finding a common basis of $\bM/\set{e}\backslash F_e$ and $((\bM/F_e)\backslash \set{e})^*$ such that its weight with respect to the restriction of $w$ to $E-F_e-e$ is in the set $\cT \coloneqq \Set{\alpha \in \Z_{\ge 0}}{\alpha + w(e) \ge w(E)-(\alpha+w(e))\ge \alpha}$.
    Collecting this instance for each $e\in E$ with $F_e$ being independent, we have at most $|E|$ instances such that at least one of them is solvable by (i) and its solution yields a feasible partition $(A_1, A_2)$.
    Therefore, \zcref{prop:exact} yields a randomized algorithm that finds a feasible partition in polynomial time in $|E|$ and $W$ with probability at least $\frac12$.
    To obtain a Las Vegas algorithm, we repeatedly run this randomized algorithm until it returns a feasible partition.
    Since a feasible partition is guaranteed to exist, the expected number of runs is at most $2$.
\end{proof}

\section{Generalizing the Grassmann--Plücker Identity}\label{sec:gp}

In this section, we provide a framework for deriving extensions of the Grassmann--Plücker identity~\eqref{eq:multiple-gp}.
\zcref{sec:preliminaries-harmonic} describes additional preliminaries on harmonic functions.
Building on this, we present the framework in \zcref{sec:deriving}, and in \zcref{sec:applications}, we use it to provide several exchange properties and their reconfiguration versions.
Algorithmic aspects are discussed in \zcref{sec:rand-algo}.
Finally, \zcref{sec:gp-applications} presents several applications.

In this section, we define binomial coefficients in a generalized sense, i.e., 
\[
    \tbinom{n}{k} \coloneqq \begin{cases}
        \frac{n (n-1) \dotsm (n-k+1)}{k!} & (k \ge 0), \\
        0 & (k < 0)
    \end{cases}
\]
for $n, k \in \Z$.
Note that $\binom{n}{k} = 0$ if and only if $k < 0$ or $0 \le n < k$.

\subsection{Basis of the Harmonic Functions}\label{sec:preliminaries-harmonic}

The proofs of \zcref{lem:dimension} and the first part of \zcref{lem:basis-expansion} largely follow~\cite{filmus2016an,filmus2019harmonicity}.
We provide proofs specializing to our cases for completeness.

Let $r \in \Zp$, $E$ a set with $|E|=2r$, and $(A,B)$ be an equal bipartition of $E$.
Let $\F$ be a field.
Recall that $g \colon \tbinom{E}{r} \to \F$ and $h \colon \cX_{A,B} \to \F$ are harmonic if they satisfy~\eqref{def:harmonic} and~\eqref{def:harmonic-change}, respectively.

We first introduce a family of harmonic functions arising from bijections from $A$ to $B$, which play a key role in our framework.
Let $\Phi_{A, B}$ denote the set of all bijections from $A$ to $B$.
For a bijection $\varphi \in \Phi_{A,B}$, define $\delta_\varphi\colon \cX_{A,B} \to \F$ by
\begin{align}
    \delta_\varphi(X,Y)
    \coloneqq
    \begin{cases}
        1 & (\varphi(X)=Y),\\
        0 & (\text{otherwise})
    \end{cases}
\end{align}
for $(X,Y) \in \cX_{A,B}$.
It is easy to see that $\delta_\varphi$ is harmonic.
The associated function $\hat{\delta}_\varphi\colon \tbinom{E}{r} \to \F$ is given by
\begin{align}\label{def:matching-harmonic}
    \hat{\delta}_\varphi(S) = \begin{cases}
        {(-1)}^{|S \cap B|} & (\text{$|S \cap \set{a, \varphi(a)}| = 1$ for all $a \in A$}), \\
        0 & \text{(otherwise)}
    \end{cases}
\end{align}
for $S \in \binom{E}{r}$.

Let $\cH_E(\F)$ be the set of all harmonic functions $g\colon \tbinom{E}{r} \to \F$ on $\tbinom{E}{r}$ with values in $\F$.
We see that if $\F$ has characteristic zero, then the dimension of $\cH_E(\F)$ as a vector space over $\F$ is the $r$th Catalan number.
Note that $\tbinom{0}{-1}$ is defined to be $0$.

\begin{lemma}[{see~\cite{filmus2016an}}]\label{lem:dimension}
    Let $\F$ be a field of characteristic zero, $r \in \Zp$, and $E$ a set with $|E|=2r$.
    Then, we have $\dim_\F \cH_E(\F) = \tbinom{2r}{r} - \tbinom{2r}{r-1} = \frac{1}{r+1}\tbinom{2r}{r}$.
\end{lemma}
\begin{proof}
    The claim is clear when $r=0$.
    Suppose $r \ge 1$.
    Since the vector space of functions $\binom{E}{r} \to \F$ has dimension $\tbinom{2r}{r}$ and there are $\binom{2r}{r-1}$ equations of the form~\eqref{def:harmonic}, we have $\dim_\F \cH_E(\F) \ge \tbinom{2r}{r} - \tbinom{2r}{r-1}$.
    In fact, the equations are linearly independent.
    Let $W$ be the inclusion matrix between $(r-1)$-subsets and $r$-subsets of $E$.
    That is, the rows and columns of $W$ are indexed by $\binom{E}{r-1}$ and $\binom{E}{r}$, respectively, and the $(T,S)$ entry is $1$ if $T\subseteq S$ and $0$ otherwise.
    Then,~\eqref{def:harmonic} can be identified with the linear system with the coefficient matrix $W$.
    Since $W$ has full-row rank as a matrix over a field of characteristic zero~\cite{gottlieb1966a}, the equations are linearly independent, and hence we obtain $\dim_\F \cH_E(\F) = \tbinom{2r}{r} - \tbinom{2r}{r-1}$. 
\end{proof}

Frankl and Graham~\cite{frankl1989old} showed that the set of functions $\hat{\delta}_\varphi$ defined by~\eqref{def:matching-harmonic} spans $\cH_E(\F)$ in the zero-characteristic case.
This immediately yields expansions of harmonic functions on $\cX_{A,B}$ in terms of $\delta_\varphi$.
Here, we prove a slightly stronger form, which also keeps track of the coefficient ring.

\begin{lemma}\label{lem:basis-expansion}
    Let $\F$ be a field of characteristic zero and $A$ and $B$ be disjoint sets with $|A|=|B|$.
    Then, for any harmonic function $h \colon \cX_{A,B} \to \F$, there exist $\alpha_{\varphi} \in \F$ $(\varphi \in \Phi_{A,B})$ such that $h = \sum_{\varphi \in \Phi_{A,B}} \alpha_{\varphi} \delta_\varphi$.
    In particular, if $h$ takes values of a subring $R$ of $\F$, we can choose $\alpha_{\varphi}$ from $R$ for all $\varphi \in \Phi_{A,B}$.
\end{lemma}

\begin{proof}
    To show the first part of the claim, we first show that $\Set{\hat{\delta}_\varphi}{\varphi \in \Phi_{A,B}}$ spans $\cH_E(\F)$, where $E\coloneqq A+B$, following the proof in~\cite{filmus2016an}.
    Fix orderings of the elements of $A$ and $B$ as $A = \set{a_1, \dotsc, a_r}$ and $B = \set{b_1, \dotsc, b_r}$.
    We say that a bijection $\varphi \in \Phi_{A,B}$ is \emph{non-crossing} if there do not exist two pairs $(a_i,b_j),(a_k,b_l)\in \Set{(a,\varphi(a))}{a \in A}$ such that $i<k \le j<l$.
    Equivalently, when the elements of $A\cup B$ are arranged in the order $a_1,b_1,a_2,b_2,\dots,a_r,b_r$ on a line and each pair is drawn as an arc above the line, no two arcs intersect.
    Let $\Phi^\mathrm{nc} \subseteq \Phi_{A,B}$ be the set of all non-crossing bijections.
    
    We show that $\Set{\hat{\delta}_\varphi}{\varphi \in \Phi^\mathrm{nc}}$ forms a basis of $\cH_E(\F)$, proving the claim.
    It is well-known that the number of non-crossing bijections is equal to the Catalan number.
    Therefore, by \zcref{lem:dimension}, it suffices to show the linear independence of the functions.
    
    We arrange the elements of $E$ in the order $a_1,b_1,a_2,b_2,\dots,a_r,b_r$ and consider the lexicographic order on $\binom{E}{r}$ induced by this ordering.
    For $\varphi \in \Phi^\mathrm{nc}$, define
    \begin{align}\label{eq:leading}
        L_{\varphi} \coloneqq \Set{a_i}{\varphi(a_i)=b_j,\ i \le j} \cup \Set{b_j}{\varphi(a_i)=b_j,\ i>j} \subseteq E.
    \end{align}
    Equivalently, \(L_{\varphi}\) is obtained by choosing, from each pair $(a,\varphi(a))$ $(a\in A)$, the element that appears earlier in the above order.
    Then, $L_{\varphi}$ is the lexicographically smallest set
    $S\in\binom{E}{r}$ such that $\hat{\delta}_\varphi(S)\ne 0$.
    Moreover, distinct non-crossing bijections yield distinct sets
    $L_{\varphi}$.
    Indeed, a non-crossing bijection is uniquely determined by its set of
    left endpoints: given $L_{\varphi}$, regard the elements of $L_{\varphi}$
    as left endpoints and the elements of $E-L_{\varphi}$ as right
    endpoints. Scanning the elements of $E$ from left to right, match each
    right endpoint with the most recent unmatched left endpoint. This usual
    stack-matching procedure recovers the bijection.
    
    Now, let $\varphi, \varphi' \in \Phi^\mathrm{nc}$ such that $L_{\varphi'}$ is lexicographically smaller than $L_{\varphi}$.
    Since $L_{\varphi}$ is the lexicographically smallest set on which $\hat{\delta}_\varphi$ is nonzero, we must have $\hat{\delta}_\varphi(L_{\varphi'})=0$.
    Thus, the matrix $L\coloneqq \prn[\big]{\hat{\delta}_\varphi(L_{\varphi'})}_{\varphi,\varphi' \in \Phi^\mathrm{nc}}$ is triangular with nonzero diagonal entries.
    Therefore, $\Set{\hat{\delta}_\varphi}{\varphi \in \Phi^\mathrm{nc}}$ is linearly independent, and thus $\Set{\hat{\delta}_\varphi}{\varphi \in \Phi_{A,B}}$ spans $\cH_E(\F)$.
    
    This basis yields an expansion of $h$ in terms of $\Set{\delta_\varphi}{\varphi \in \Phi^{\mathrm{nc}}} \subseteq \Set{\delta_\varphi}{\varphi \in \Phi_{A,B}}$.
    More precisely, the coefficients are given by the unique solution ${(\alpha_{\varphi})}_{\varphi \in \Phi^{\mathrm{nc}}}$ of the linear system
    \begin{align}
    \sum_{\varphi' \in \Phi^{\mathrm{nc}}}
    \hat{\delta}_{\varphi'}(L_{\varphi}) \alpha_{\varphi'}
    =
    \hat{h}(L_{\varphi})
    \qquad
    (\varphi \in \Phi^{\mathrm{nc}}),
    \end{align}
    whose coefficient matrix is $L$.
    Moreover, the matrix $L$ is triangular with diagonal entries $\pm 1$.
    Hence $L$ is unimodular, and $L^{-1}$ is again an integer matrix.
    Therefore, each $\alpha_{\varphi}$ can be chosen as an integer linear combination of the values of $h$.
    This proves the second assertion.
\end{proof}

Combining \zcref{lem:basis-expansion,thm:gp-harmonic}, we obtain the following expansion of $\mu_{A,B}^K$ in terms of $\delta_\varphi$.
Note that the following lemma does not impose a condition on the characteristic of $\F$.

\begin{lemma}\label{lem:basis-expansion-mu}
    Let $K \in \F^{r \times n}$ be a matrix over a field $\F$ and $A,B \in \binom{[n]}{r}$.
    Then, there exists $\alpha_{\varphi} \in \F$ $(\varphi \in \Phi_{A\setminus B,B\setminus A})$ such that
    \begin{align}\label{eq:basis-expansion}
        \mu_{A,B}^K(X, Y) = \sum_{\varphi \in \Phi_{A \setminus B, B \setminus A}} \alpha_{\varphi} \delta_\varphi(X, Y)
    \end{align}
    for all $X \subseteq A \setminus B$ and $Y \subseteq B \setminus A$ with $|X|=|Y|$.
\end{lemma}

\begin{proof}
    First, consider a symbolic matrix $K = {(k_{ij})}_{1 \le i \le r, 1 \le j \le n}$.
    We regard $K$ as a matrix over the polynomial ring $R \coloneqq \Z[k_{1,1}, \dotsc, k_{r,n}]$, which is a subring of the rational function field $\Q(k_{1,1}, \dotsc, k_{r,n})$.
    Since $\mu^K_{A,B}$ is harmonic by \zcref{thm:gp-harmonic}, \zcref{lem:basis-expansion} implies that there exists $\alpha_{\varphi} \in R$ $(\varphi \in \Phi_{A \setminus B, B\setminus A})$ such that
    \begin{align}\label{eq:basis-expansion-c}
        \mu^K_{A,B}(X,Y)
        = \sum_{\varphi \in \Phi_{A\setminus B, B\setminus A}} \alpha_{\varphi} \delta_{\varphi}(X, Y)
    \end{align}
    for all $(X,Y) \in \cX_{A \setminus B, B \setminus A}$.
    Then,~\eqref{eq:basis-expansion-c} can be seen as an identity of polynomials in $k_{1,1}, \dotsc, k_{r,n}$ with $\Z$ coefficients.
    Hence, the same identities hold over $\F[k_{1,1}, \dotsc, k_{r,n}]$ for any field $\F$ as well, proving the claim.
\end{proof}

\subsection{Deriving Exchange Properties via Harmonicity}\label{sec:deriving}

We now prove \zcref{thm:gp-main,thm:gp-matroids}, which we restate below.
The former provides a framework for deriving extensions of the Grassmann--Plücker identity, while the latter describes its consequences for exchange properties, including a reconfiguration statement in the characteristic-zero case.

\gpmain*
\gpmatroids*

The following lemma immediately implies \zcref{thm:gp-main} given \zcref{lem:basis-expansion-mu}.

\begin{lemma}\label{lem:h-c}
    Let $\F$ be a field and $A,B$ sets with $|A|=|B|$, and $c\colon \cX_{A\setminus B, B\setminus A} \to \F$ be such that~\eqref{eq:c} holds for every bijection $\varphi\colon A\setminus B \to B\setminus A$.
    Then, for any function $h\colon \cX_{A \setminus B, B \setminus A} \to \F$ written as $h = \sum_{\varphi \in \Phi_{A\setminus B,B\setminus A}} \alpha_\varphi \delta_\varphi$ for some $\alpha_\varphi \in \F$ $(\varphi \in \Phi_{A\setminus B,B\setminus A})$, we have
    \begin{align}
        h(\emptyset, \emptyset) = \sum_{(X,Y) \in \cX_{A\setminus B,B\setminus A}} c(X,Y) h(X,Y).
    \end{align}
\end{lemma}
\begin{proof}
    Since~\eqref{eq:c} can be rewritten as
    \begin{align}
        \sum_{(X,Y) \in \cX_{A\setminus B, B\setminus A}} c(X, Y) \delta_\varphi(X,Y) = 1 = \delta_\varphi(\emptyset, \emptyset),
    \end{align}
    we have
    \begin{align}
        h(\emptyset, \emptyset)
        &= \sum_{\varphi \in \Phi_{A\setminus B, B\setminus A}} \alpha_\varphi \delta_\varphi(\emptyset, \emptyset) \\
        &= \sum_{\varphi \in \Phi_{A\setminus B, B\setminus A}} \alpha_\varphi \sum_{(X,Y) \in \cX_{A\setminus B, B\setminus A}} c(X, Y) \delta_\varphi(X,Y) \\
        &= \sum_{(X,Y) \in \cX_{A\setminus B, B\setminus A}} c(X, Y) \sum_{\varphi \in \Phi_{A\setminus B, B\setminus A}} \alpha_\varphi \delta_\varphi(X,Y) \\
        &= \sum_{(X,Y) \in \cX_{A\setminus B,B\setminus A}} c(X,Y) h(X,Y).
        \qedhere
    \end{align}
\end{proof}

By \zcref{lem:basis-expansion,lem:basis-expansion-mu}, any harmonic function over a field of characteristic zero and $\mu_{A,B}^K$ over a field of any characteristic admit the expansion required in \zcref{lem:h-c}.
In particular, \zcref{thm:gp-main} is proved by \zcref{lem:basis-expansion-mu,lem:h-c}.

Applying \zcref{lem:h-c} to the restriction of $\mu_{A,B}^K$ to a connected component, we further obtain \zcref{thm:gp-matroids}.

\begin{proof}[{Proof of \zcref{thm:gp-matroids}}]
    The first claim is an immediate consequence of \zcref{thm:gp-main}.
    To prove the second claim, let $h \coloneqq \mu_{A,B}^K$, $g \coloneqq \hat h$, and consider the subgraph of the Johnson graph
    $J(A \symdif B, |A\setminus B|)$ induced by $\supp(g)$.
    As with \zcref{lem:component-harmonicity}, let $\cC \subseteq \tbinom{A\symdif B}{|A\setminus B|}$ be the vertex set of the connected component of this subgraph containing $A \setminus B$, and let $g_\cC$ be the function defined by~\eqref{def:g-C} from $g$ and $\cC$.
    By \zcref{lem:component-harmonicity}, $g_\cC$ is harmonic.
    Thus, the function $h_\cC$ such that $\hat h_\cC = g_\cC$ is also harmonic.
    If $\chr(\F)=0$, then \zcref{lem:h-c} applies to $h_\cC$ by \zcref{lem:basis-expansion}, asserting
    \begin{align}
        0 \ne h(\emptyset, \emptyset)
        = h_\cC(\emptyset, \emptyset)
        = \sum_{(X,Y) \in \cX_{A\setminus B,B\setminus A}} c(X,Y) h_\cC(X,Y).
    \end{align}
    This implies that $h_\cC(U,V) \ne 0$ for some $(U,V) \in \supp(c)$.
    Then, $A-U+V$ and $B+U-V$ are bases and $(A\setminus B)-U+V\in \cC$, proving the second claim.
\end{proof}

\subsection{Applications of the Harmonic Framework}\label{sec:applications}

In this section, we apply \zcref{thm:gp-main,thm:gp-matroids} to derive known and new exchange properties for representable matroids, in some cases under characteristic conditions, as well as their reconfiguration versions.
First, in \zcref{sec:known}, we derive Grassmann--Plücker identities corresponding to \zcref{thm:multiple,thm:equitability-exchange,thm:main}, with a weaker upper bound on the common size of $U$ and $V$ for \zcref{thm:main}.
In \zcref{sec:two-max}, we prove \zcref{thm:two-max} without explicitly computing the coefficients of the corresponding Grassmann--Plücker identity and further discuss the consequences of the theorem.
In \zcref{sec:general}, we prove a more general exchange property (\zcref{thm:gp-partition}).

\subsubsection{Deriving Grassmann--Plücker Identities of Exchange Properties}\label{sec:known}

To build an intuition and demonstrate the power of the harmonic framework, we first derive Grassmann--Plücker identities corresponding to \zcref{thm:multiple,thm:equitability-exchange,thm:main} (with a weaker upper bound $|U|=|V|\le|X+Y|$) using \zcref{thm:gp-main}.
By \zcref{thm:gp-matroids}, these imply not only exchange properties but also their reconfiguration extensions for the $\chr(\F)=0$ case.

We first observe that the multiple Grassmann--Plücker identity~\eqref{eq:multiple-gp} is an immediate consequence of \zcref{thm:gp-main}.
For $A,B \in \tbinom{[n]}{r}$ and $X \subseteq A\setminus B$, define $c\colon \cX_{A\setminus B, B\setminus A} \to \F$ by
\begin{align}
    c(U,V) \coloneqq
    \begin{cases}
        1 & (U=X),\\
        0 & (\text{otherwise}).
    \end{cases}
\end{align}
The function $c$ clearly satisfies~\eqref{eq:c} for every bijection $\varphi \in \Phi_{A\setminus B, B\setminus A}$.
Thus, \zcref{thm:gp-main} yields
\begin{align}
    \mu_{A,B}^K(\emptyset, \emptyset)
= \sum_{Y \in \tbinom{B\setminus A}{|X|}} \mu_{A,B}^K(X,Y),
\end{align}
which is precisely~\eqref{eq:multiple-gp}, including the correct signs.
When $\chr(\F)=0$, \zcref{thm:gp-matroids} also implies the reconfiguration property stated in \zcref{thm:reconf-mep}.

We next show the following, which derives \zcref{thm:equitability-exchange} as well as its reconfiguration version for matroids representable over a field of characteristic zero.

\begin{lemma}\label{lem:c-equitability}
    Let $\F$ be a field of characteristic zero, $A$ and $B$ sets with $|A|=|B|$, $X \subseteq A\setminus B$, and $(Y_1, Y_2)$ a bipartition of $B\setminus A$ with $|X|>|Y_1|$.
    Then, a function $c\colon \cX_{A\setminus B,B\setminus A}\to \F$ defined by
    \begin{align}
        c(U,V) \coloneqq \begin{cases}
            \frac{1}{|U|\tbinom{|X|}{|U|}} & (U \subseteq X,\, |V \cap Y_2|=1), \\
            0 & (\text{otherwise})
        \end{cases}
    \end{align}
    for $(U,V) \in \cX_{A\setminus B,B\setminus A}$ satisfies~\eqref{eq:c} for every $\varphi \in \Phi_{A\setminus B,B\setminus A}$.
\end{lemma}

\begin{proof}
    The conditions $U \subseteq X$ and $|\varphi(U) \cap Y_2|=1$ are equivalent to choosing one element from $\varphi(X) \cap Y_2$ and an arbitrary number of elements from $\varphi(X) \cap Y_1$.
    Thus, letting $k \coloneqq |\varphi(X)\cap Y_1|$ and $u \coloneqq |U|$, the left-hand side of~\eqref{eq:c} becomes
    \begin{align}\label{eq:equitability-gp-2}
        \sum_{u=1}^{k+1} (|X|-k) \tbinom{k}{u-1} \frac{1}{u\tbinom{|X|}{u}}
        = \sum_{i=0}^{k} \frac{|X|-k}{|X|-i} \cdot \frac{\tbinom{k}{i}}{\tbinom{|X|}{i}},
    \end{align}
    where we used $\tbinom{n}{m} = \frac{n-m+1}{m}\tbinom{n}{m-1}$ for $n,m \in \Z$ with $m \ne 0$.
    By
    \begin{align}
        \frac{\tbinom{k}{i}}{\tbinom{|X|}{i}}
        - \frac{\tbinom{k}{i+1}}{\tbinom{|X|}{i+1}}
        = \frac{\tbinom{k}{i}}{\tbinom{|X|}{i}}
        - \frac{(k-i)\tbinom{k}{i}}{(|X|-i)\tbinom{|X|}{i}}
        = \frac{|X|-k}{|X|-i} \cdot \frac{\tbinom{k}{i}}{\tbinom{|X|}{i}}
    \end{align}
    for $0 \le i < |X|$, we can rewrite~\eqref{eq:equitability-gp-2} using telescoping as   \begin{align}
        \sum_{i=0}^{k} \frac{|X|-k}{|X|-i} \cdot \frac{\tbinom{k}{i}}{\tbinom{|X|}{i}}
        = \sum_{i=0}^{k} \prn*{\frac{\tbinom{k}{i}}{\tbinom{|X|}{i}} - \frac{\tbinom{k}{i+1}}{\tbinom{|X|}{i+1}}}
        = \frac{\tbinom{k}{0}}{\tbinom{|X|}{0}} - \frac{\tbinom{k}{k+1}}{\tbinom{|X|}{k+1}}
        = 1.
        &\qedhere
    \end{align}
\end{proof}

\begin{corollary}
    Let $K \in \F^{r \times n}$ be a matrix over a field of characteristic zero, $A,B \in \tbinom{[n]}{r}$, $X \subseteq A\setminus B$, and $(Y_1, Y_2)$ a bipartition of $B\setminus A$ with $|X|>|Y_1|$.
    Then,
    \begin{align}
        \mu_{A,B}^K(\emptyset, \emptyset)
        = \sum_{(U,V) \in \cX_{X, B\setminus A} \::\: |V\cap Y_2|=1} \frac{1}{|U|\tbinom{|X|}{|U|}}\mu_{A,B}^K(U,V).
    \end{align}
\end{corollary}

\begin{corollary}\label{cor:reconf-arlv}
  Let $A$ and $B$ be bases of a matroid $\bM$ representable over a field of characteristic zero, $X\subseteq A\setminus B$, and  $(Y_1, Y_2)$ bipartition of $B\setminus A$ with $|X|>|Y_1|$.
  Then, there exist $U \subseteq X$ and $V \subseteq B \setminus A$ such that $|V \cap Y_2|=1$, $A-U+V$ and $B+U-V$ are bases of $\bM$, and $(A-U+V,B+U-V)$ is obtained from $(A,B)$ using symmetric exchanges in $\bM$.
\end{corollary}

Third, we give a function $c$ that yields \zcref{thm:main} with a weaker upper bound on the sizes of $U$ and $V$, and its reconfiguration version for the $\chr(\F)=0$ case.
Note that $\tbinom{|X+Y|-|A\setminus B|}{|X+Y|-|U|} = 0$ if $|U|>|X+Y|$.

\begin{lemma}\label{lem:c-main}
    Let $\F$ be a field, $A$ and $B$ sets with $|A|=|B|$, $X \subseteq A \setminus B$, and $Y \subseteq B \setminus A$.
    Then, a function $c\colon \cX_{A\setminus B,B\setminus A}\to \F$ defined by
    \begin{align}
        c(U,V) \coloneqq \begin{cases}
            \tbinom{|X+Y|-|A\setminus B|}{|X+Y|-|U|} & (X \subseteq U,\,\ Y \subseteq V), \\
            0 & (\text{otherwise})
        \end{cases}
    \end{align}
    for $(U,V) \in \cX_{A\setminus B,B\setminus A}$ satisfies~\eqref{eq:c} for every $\varphi \in \Phi_{A\setminus B,B\setminus A}$.
\end{lemma}

\begin{proof}
    The conditions $X \subseteq U$ and $Y \subseteq \varphi(U)$ are equivalent to $X\cup \varphi^{-1}(Y)\subseteq U$.
    Thus, letting $k \coloneqq |X\cup \varphi^{-1}(Y)|$ and $u \coloneqq |U|$, the left-hand side of~\eqref{eq:c} becomes
    \begin{align}\label{eq:main-gp-vandermonde}
        \sum_{u=k}^{|X+Y|}
        \tbinom{|A\setminus B|-k}{u-k}
        \tbinom{|X+Y|-|A\setminus B|}{|X+Y|-u}
        = \sum_{i=0}^{|X+Y|-k}
        \tbinom{|A\setminus B|-k}{i}
        \tbinom{|X+Y|-|A\setminus B|}{|X+Y|-k-i}.
    \end{align}
    We use Vandermonde's convolution identity, stating $\sum_{i=0}^m \binom{n}{i}\binom{m}{r-i} = \binom{n+m}{r}$ for $n,m,r \in \Z$ (see~\cite{Spivey2019-kl}).
    This and $k \le |X+Y|$ give that~\eqref{eq:main-gp-vandermonde} is equal to $\tbinom{|X+Y|-k}{|X+Y|-k} = 1$, as required.
\end{proof}

\begin{corollary}
    Let $K \in \F^{r \times n}$ be a matrix over a field $\F$, $A, B \in \tbinom{[n]}{r}$, $X \subseteq A \setminus B$, and $Y \subseteq B \setminus A$.
    Then,
    \begin{align}\label{eq:main-gp}
        \mu_{A,B}^K(\emptyset, \emptyset)
        = \sum_{(U,V) \in \cX_{A\setminus B, B\setminus A}\::\: X \subseteq U, Y \subseteq V, |U|\le|X+Y|} \tbinom{|X+Y|-|A\setminus B|}{|X+Y|-|U|} \mu_{A,B}^K(U,V).
    \end{align}
\end{corollary}

\begin{corollary}\label{cor:reconf-main}
  Let $A$ and $B$ be bases of a matroid $\bM$ representable over a field of characteristic zero, $X\subseteq A\setminus B$, and  $Y \subseteq B\setminus A$.
  Then, there exist $U \subseteq A \setminus B$ and $V \subseteq B \setminus A$ such that $X \subseteq U$, $Y \subseteq V$, $|U|=|V| \le |X+Y|$, $A-U+V$ and $B+U-V$ are bases of $\bM$, and $(A-U+V,B+U-V)$ is obtained from $(A,B)$ using symmetric exchanges in $\bM$.
\end{corollary}

Since \zcref{thm:gp-main} yields only polynomial identities in the entries of the matrix and does not take into account algebraic relations among the entries, such as $\rank K[X+Y] < |X+Y|$, the stronger form of \zcref{thm:main} with the upper bound $|U|=|V| \le r(X+Y)$ cannot be obtained by a na\"ive application of \zcref{thm:gp-main}.
Nevertheless, \zcref{thm:with_ranks} proved in \zcref{sec:two-max} below recovers \zcref{thm:main} with the stronger upper bound for representable matroids.

\subsubsection{A Common Generalization of \texorpdfstring{\zcref{thm:main,thm:equitability-exchange}}{Theorems~\ref{thm:main} and~\ref{thm:equitability-exchange}} for Representable Cases}\label{sec:two-max}

In this section, we apply the harmonic framework to derive further generalizations beyond \zcref{thm:main,thm:equitability-exchange} for representable matroids, in some cases under characteristic conditions.

Let $A$ and $B$ be sets with $|A|=|B|$. To avoid complicated calculations of binomial coefficients, we first describe a useful lemma ensuring the existence of $c$ with~\eqref{eq:c} and $\supp(c) \subseteq \cU$ for given $\cU \subseteq \cX_{A\setminus B,B\setminus A}$  without computing its exact values.
For a bijection $\varphi \in \Phi_{A\setminus B, B\setminus A}$ and $\cU \subseteq \cX_{A\setminus B,B\setminus A}$, define 
\begin{align}
    \tau(\varphi, \cU)
    \coloneqq \sum_{(U, V) \in \cU} \delta_\varphi(U, V)
    = |\Set{U\subseteq A\setminus B}{(U, \varphi(U))\in \cU}|.
\end{align}
Let $\bmPhi = (\Phi_1, \dotsc, \Phi_k)$ be a partition of $\Phi_{A \setminus B, B \setminus A}$ with $\Phi_i \ne \emptyset$ for $i \in [k]$, and $\bmcU = (\cU_1, \dotsc, \cU_l)$ be a partition of a subset $\cU \subseteq \cX_{A\setminus B,B\setminus A}$.
We say that the pair $(\bmPhi, \bmcU)$ is \emph{compatible} if, for any $i \in [k]$ and $j \in [l]$, the value $\tau(\varphi, \cU_j)$ is constant over all $\varphi \in \Phi_i$.
For compatible $(\bmPhi, \bmcU)$, let $C_{\bmPhi, \bmcU} \in \Z^{k \times l}$ denote the matrix such that its $(i,j)$ entry is $\tau(\varphi, \cU_j)$ with arbitrary $\varphi \in \Phi_i$.

\begin{lemma}\label{lem:matrix}
    Let $A$ and $B$ be sets with $|A|=|B|$, $\bmPhi = (\Phi_1, \dotsc, \Phi_k)$ a partition of $\Phi_{A \setminus B, B \setminus A}$, and $\bmcU = (\cU_1, \dotsc, \cU_k)$ a partition of $\cU \subseteq \cX_{A\setminus B,B\setminus A}$ such that $(\bmPhi, \bmcU)$ is compatible.
    If $C_{\bmPhi, \bmcU}$ is nonsingular as a matrix over a field $\F$, then there exists $c\colon \cX_{A\setminus B,B\setminus A} \to \F$ such that~\eqref{eq:c} holds for every $\varphi \in \Phi_{A\setminus B,B\setminus A}$ and $\supp(c) \subseteq \cU$.
\end{lemma}
\begin{proof}
    By the nonsingularity of $C_{\bmPhi, \bmcU}$, we can take $x \in \F^{k}$ with $C_{\bmPhi, \bmcU}x = \ones$, where $\ones$ is the vector of ones.
    This equality means
    \begin{align}
        1
        =\sum_{j=1}^k \tau(\varphi, \cU_j) x_j
        =\sum_{j=1}^k \sum_{(U,V) \in \cU_j} \delta_\varphi(U,V) x_j
        =\sum_{j=1}^k \sum_{(U,V) \in \cU_j \::\: \varphi(U)=V} x_j
    \end{align}
    for any $i \in [k]$ and $\varphi \in \Phi_i$.
    Therefore, the function $c\colon \cX_{A\setminus B,B\setminus A} \to \F$ defined by $c(X,Y) \coloneqq x_j$ for $(X,Y) \in \cU_j$ and $0$ otherwise satisfies~\eqref{eq:c} as required.
\end{proof}

Using \zcref{lem:matrix}, we derive the following; see \zcref{fig:bipartition-gp-exchange} for illustration.

\begin{figure}
    \centering
    \begin{tikzpicture}[
      font=\small,
      line width=0.7pt,
      line cap=round,
      line join=round,
      brace above/.style={decorate,decoration={brace,amplitude=5pt}},
      brace below/.style={decorate,decoration={brace,mirror,amplitude=5pt}},
      myarrow/.style={<->, >=Stealth, line width=0.7pt}
    ]
      \def\W{5}
      \def\H{0.35}
      \def\WX{3.5}
      \def\WYone{2.9}
      \def\WU{2}
      \def\WVone{1.5}
      \def\WVtwo{0.5}
      \def\Boriginy{-1.5}
    
      \node[left=12pt] at (0, 0.5*\H) {$A\setminus B$};
      \draw (0,0) rectangle (\W, \H);
      \draw (\WX, 0) -- (\WX, \H);
      \draw[brace above] (0.05, \H+0.1) -- node[midway,above=6pt] {$X$} (\WX-0.05, \H+0.1);
      \fill[gray!25] (0.07, 0.07) rectangle (0.07+\WU, \H-0.07);
    
      \node[left=12pt] at (0, \Boriginy+0.5*\H) {$B\setminus A$};
      \draw (0, \Boriginy) rectangle (\W, \Boriginy+\H);
      \draw (\WYone, \Boriginy) -- (\WYone, \Boriginy+\H);
      
      \draw[brace below] (0.05, \Boriginy-0.10) -- node[midway,below=6pt] {$Y_t$} (\WYone-0.05, \Boriginy-0.10);
      \draw[brace below] (\WYone+0.05, \Boriginy-0.10) -- node[midway,below=6pt] {$Y_{3-t}$} (\W-0.05, \Boriginy-0.10);
      
      \fill[gray!25] (0.07, \Boriginy+0.07) rectangle (\WVone+0.07, \Boriginy+\H-0.07);
      \draw[myarrow] (0.07, \Boriginy+\H+0.15) -- (\WVone+0.07, \Boriginy+\H+0.15) node[midway, above] {$\ge\max\set{d_t,|X|-|Y_{3-t}|}$};
      
      \fill[gray!25] (\WYone+0.07, \Boriginy+0.07) rectangle (\WYone+0.07+\WVtwo, \Boriginy+\H-0.07);
      \draw[myarrow] (\WYone+0.07, \Boriginy+\H+0.15) -- (\WYone+0.07+\WVtwo, \Boriginy+\H+0.15) node[midway, above] {$p_t$};
    \end{tikzpicture}
    \caption{Illustration of $(U,V) \in \cU^t$ in \zcref{lem:bipartition-gp}. Gray areas in $A\setminus B$ and $B\setminus A$ indicate $U$ and $V$, respectively.}\label{fig:bipartition-gp-exchange}
\end{figure}

\begin{lemma}\label{lem:bipartition-gp}
    Let $\F$ be a field, $A$ and $B$ be sets with $|A|=|B|$, $X\subseteq A\setminus B$, $(Y_1, Y_2)$ a bipartition of $B\setminus A$, $d_1, d_2\in \Z_{\ge 0}$ with $d_1+d_2 = |X|+1$, and $p_1, p_2\in \Z_{\ge 0}$ with $p_t \le \max\set{|X|-|Y_t|, 0}$ for $t \in [2]$.
    If $p_1 = p_2 = 0$ or $\F$ has characteristic zero, then there exists $c\colon \cX_{A\setminus B,B\setminus A} \to \F$ such that~\eqref{eq:c} holds for every $\varphi \in \Phi_{A\setminus B,B\setminus A}$ and $\supp(c) \subseteq \cU^1 \cup \cU^2$, where
    \begin{align}\label{def:cU-t}
        \cU^t &\coloneqq \Set{(U,V) \in \cX_{X, B\setminus A}}{|U| \ge \max\set{d_t, |X|-|Y_{3-t}|}+p_t,\, |V \cap Y_{3-t}|=p_t}
    \end{align}
    for $t \in [2]$.
\end{lemma}
\begin{proof}
    Observe that for $t \in [2]$, we have 
    \begin{align} \label{eq:possible-sizes-new}
        \Set{|\varphi(X)\cap Y_t|}{\varphi \in \Phi_{A\setminus B, B\setminus A}}
        = \Set{i \in \Z}{\max\set{0, |X|-|Y_{3-t}|} \le i \le \min\set{|X|, |Y_t|}}.
    \end{align}
    For $t \in [2]$, let $a_t \coloneqq \max\set{d_t, |X|-|Y_{3-t}|}$ and $b_t \coloneqq \min \set{|X|, |Y_t|}$, and for $a_t \le i \le b_t$, define \[\Phi^{t}_i \coloneqq \Set{\varphi \in \Phi_{A\setminus B, B\setminus A}}{|\varphi(X)\cap Y_t|=i}.\]
    Since $d_1+d_2=|X|+1$, each bijection $\varphi$ satisfies exactly one of $d_1 \le |\varphi(X)\cap Y_1|$ and $d_2 \le |\varphi(X)\cap Y_2|$, hence~\eqref{eq:possible-sizes-new} implies that $\bmPhi \coloneqq (\Phi^1_{a_1}, \dots, \Phi^1_{b_1}, \Phi^2_{a_2},\dots, \Phi^2_{b_2})$ is a partition of $\Phi_{A\setminus B, B\setminus A}$ into nonempty parts.
    
    For $t \in [2]$ and $a_t \le j \le b_t$, define
    \begin{align}
        \cU^t_{j} & \coloneqq \Set{(U, V) \in \cX_{X,B\setminus A}}{|U|=j+p_t, |V\cap Y_{3-t}|=p_t}.
    \end{align}
    Since each $(U,V) \in \cU^t$ satisfies $|V \cap Y_t| = |U|-p_t \le \min\set{|X|-p_t, |Y_t|} \le b_t$, the tuple $(\cU^t_{a_t}, \dotsc, \cU^t_{b_t})$ is a partition of $\cU^t$.
    Moreover, $\cU^1_{j_1}$ and $\cU^2_{j_2}$ are disjoint whenever $a_1 \le j_1 \le b_1$ and $a_2 \le j_2 \le b_2$: a common element $(U,V)$ would satisfy $j_1 = |V \cap Y_1| = p_2$ and $j_2 = |V \cap Y_2| = p_1$, hence
    \[
        |X|+1 = d_1+d_2 \le a_1+a_2 \le j_1+j_2 = p_1+p_2 \le \max\set{|X|-|Y_1|,0} + \max\set{|X|-|Y_2|,0} \le |X|,
    \]
    a contradiction.
    Thus, $\bmcU \coloneqq (\cU_{a_1}^1,\dots, \cU^1_{b_1}, \cU_{a_2}^2, \dots, \cU_{b_2}^2)$ consists of pairwise disjoint sets.
    
    Observe that for each bijection $\varphi \in \Phi_{A\setminus B, B\setminus A}$, each $t \in [2]$, and any integer $j$ with $a_t \le j \le b_t$, we have $\tau(\varphi, \cU^t_{j}) = \binom{|\varphi(X)\cap Y_t|}{j} \binom{|\varphi(X)\cap Y_{3-t}|}{p_t}$.
    This shows that if $a_t \le i_t \le b_t$, $a_t \le j_t \le b_t$, and $\varphi^t_{i_t} \in \Phi^t_{i_t}$ for $t \in [2]$, then 
     \begin{align} \label{eq:CMU}
    \begin{alignedat}{2} 
        \tau(\varphi^1_{i_1}, \cU^1_{j_1}) & = \tbinom{i_1}{j_1}  \tbinom{|X|-i_1}{p_1}, & \qquad  \tau(\varphi^1_{i_1}, \cU^2_{j_2}) & = \tbinom{|X|-i_1}{j_2} \tbinom{i_1}{p_2} = 0, \\
        \tau(\varphi^2_{i_2}, \cU^1_{j_1}) & = \tbinom{|X|-i_2}{j_1} \tbinom{i_2}{p_1} = 0, &  \tau(\varphi^2_{i_2}, \cU^2_{j_2}) &= \tbinom{i_2}{j_2} \tbinom{|X|-i_2}{p_2},
    \end{alignedat} \end{align}
    where $\tbinom{|X|-i_1}{j_2}= \tbinom{|X|-i_2}{j_1} = 0$ since 
    \[|X|-i_t \le |X|-a_t = |X|-\max\set{d_t, |X|-|Y_{3-t}|} = \min \set{d_{3-t}-1,|Y_{3-t}|} \le a_{3-t}-1 < j_{3-t}.\]
    This shows that $(\bmPhi, \bmcU)$ is compatible, and the entries of $C_{\bmPhi, \bmcU}$ are given by~\eqref{eq:CMU}.
    We further see from~\eqref{eq:CMU} that $C_{\bmPhi, \bmcU}$ is a lower triangular matrix, since $\binom{i_t}{j_t} = 0$ if $j_t > i_t$.  
    Its diagonal entries are $\tbinom{i_t}{i_t}\tbinom{|X|-i_t}{p_t} = \tbinom{|X|-i_t}{p_t}$ for $t \in [2]$ and $i_t \in \set{a_t,\dots, b_t}$.
    Since $|X|-i_t \ge |X|-b_t = |X|-\min\set{|X|, |Y_t|} = \max\set{0, |X|-|Y_t|} \ge p_t$, we have $\binom{|X|-i_t}{p_t} \ne 0$, thus all diagonal entries of $C_{\bmPhi, \bmcU}$ are nonzero integers, and are indeed $1$ if $p_1 = p_2 = 0$.
    Therefore, \zcref{lem:matrix} gives the desired $c$.
\end{proof}

\zcref{thm:gp-matroids,lem:bipartition-gp} imply \zcref{thm:two-max}, restated below for readability.

\twomax*

\begin{proof}
    We first verify that the maximum in~\eqref{def:m-t} is taken over a nonempty set.
    Without loss of generality, let $t=1$.
    We apply \zcref{thm:gp-matroids,lem:bipartition-gp} with $d_1 \coloneqq 0$ and $d_2 \coloneqq |X|+1$, and let $\cU^t$ be defined by~\eqref{def:cU-t} for $t \in [2]$.
    Since every $(U,V) \in \cU^2$ would satisfy $|U| \ge d_{2} + p_{2} > |X| \ge |U|$, we have $\cU^{2} = \emptyset$; hence the obtained pair $(U,V) \in \cU_{A,B}^{\bM}(\F) \cap \cU^1$ belongs to the set defining $m_1$.
    
    We apply \zcref{thm:gp-matroids,lem:bipartition-gp} again with $d_1 \coloneqq |X|-m_2$ and $d_2 \coloneqq m_2+1$, where $\cU^1$ and $\cU^2$ below refer to~\eqref{def:cU-t} for these new $d_1$ and $d_2$.
    Since every $(U,V) \in \cU^2$ satisfies $|V \cap Y_2| = |U| - p_2 \ge d_2 > m_2$, contradicting the maximality of $m_2$, we have $\cU_{A,B}^\bM(\F) \cap \cU^2 = \emptyset$; hence the obtained pair lies in $\cU_{A,B}^\bM(\F) \cap \cU^1$.
    Any $(U,V) \in \cU^1$ satisfies $|V \cap Y_1| = |U| - p_1 \ge d_1$, so $m_1 \ge d_1 = |X|-m_2$, finishing the proof.
\end{proof}

As a consequence of \zcref{thm:two-max}, we obtain \zcref{thm:with_ranks}, restated below, which is a common generalization of \zcref{thm:main, thm:equitability-exchange} for matroids representable over a field of characteristic zero.

\withranks*

\begin{proof}
    We first show~\ref{it:subsets_with_ranks}.
    Applying \zcref{thm:two-max} to the subset $X_1 \subseteq A\setminus B$, the bipartition $(Y_1, Y_2)$ of $B\setminus A$, and $(p, 0)$ in place of $(p_1, p_2)$, we have $|X_1| \le m_1 + m_2$, where $m_t$ is defined by~\eqref{def:m-t} for $t \in [2]$.
    Observe that if $(U,V) \in \cU_{A,B}^{\bM}(\F)$ satisfies $U\subseteq X_1$ and $V\subseteq Y_2$, then $X_2+V$ and $Y_1+U$ are independent subsets of $X_2+Y_2$ and $X_1+Y_1$, respectively, thus $|V| \le r(X_2+Y_2)-|X_2|$ and $|U|\le r(X_1+Y_1)-|Y_1|$, hence $m_2 \le \min \set{r(X_2+Y_2)-|X_2|,\, r(X_1+Y_1)-|Y_1|}$.
    Therefore,
    \[
        m_1 \ge |X_1|-m_2 \ge \max \set[\big]{|A\setminus B| - r(X_2+Y_2),\ |X_1+Y_1|-r(X_1+Y_1)},
    \]
    and any $(U,V) \in \cU_{A,B}^{\bM}(\F)$ attaining the maximum of $m_1$ proves the claim.

    Next we show~\ref{it:supersets_with_ranks}.
    We apply~\ref{it:subsets_with_ranks} to the basis pair $(B,A)$, the bipartitions $(Y_2, Y_1)$ of $B\setminus A$ and $(X_2, X_1)$ of $A\setminus B$, and the integer $p$; the assumption on $p$ is met since $|Y_2|-|X_2| = |X_1|-|Y_1|$ by $|X_1|+|X_2|=|Y_1|+|Y_2|$.
    This yields $(U',V') \in \cU_{B,A}^{\bM}(\F)$ such that $U' \subseteq Y_2$,
    \[
        |X_2 \cap V'| \ge \max\set[\big]{|A\setminus B|-r(X_1+Y_1),\ |X_2+Y_2|-r(X_2+Y_2)},
    \]
    and $|X_1 \cap V'| = p$.
    Let $U \coloneqq (A\setminus B)-V'$ and $V\coloneqq (B\setminus A)-U'$.
    Then, $A-U+V = B-U'+V'$ and $B+U-V = A+U'-V'$ are bases, and if $\chr(\F)=0$, then $(A-U+V, B+U-V)$ is obtained from $(B,A)$ using symmetric exchanges, hence also from $(A,B)$ by \zcref{cor:weak-gabow}.
    Therefore, $(U,V) \in \cU_{A,B}^{\bM}(\F)$.
    Moreover, $|X_1 \setminus U| = |X_1 \cap V'| = p$ holds, and $U' \subseteq Y_2$ implies $Y_1 \subseteq V$.
    Finally, we have
    \begin{align}
        |U| = |A \setminus B| - |V'|
        &= |A\setminus B| - p - |X_2 \cap V'| \\
        &\le |A\setminus B| - p - \max\set[\big]{|A\setminus B|-r(X_1+Y_1),\ |X_2+Y_2|-r(X_2+Y_2)} \\
        &= \min\set[\big]{r(X_1+Y_1),\ r(X_2+Y_2)-|X_2+Y_2|+|A\setminus B|} - p,
    \end{align}
    finishing the proof.
\end{proof}

\subsubsection{An Exchange Property with General Partitions} \label{sec:general}

For a matroid representable over a field of characteristic zero, this section describes an exchange property with partitioning $A \setminus B$ and $B\setminus A$ into a general number of parts.
For $k,l \in \Z_{\ge0}$, let $\cP_{k,l} \coloneqq \Set{(u, v) \in \Z_{\ge0}^k \times \Z_{\ge0}^l}{\sum_{i=1}^k u_i = \sum_{j=1}^l v_j}$.
For $(x,y) \in \cP_{k,l}$, let $\cQ_{x,y} \coloneqq \Set{Q \in \Z_{\ge0}^{k \times l}}{Q\ones = x,\, Q^\top\ones = y}$, where $\ones$ denotes the all-one vector.
For $Q = {(q_{i,j})}_{1\le i\le k, 1\le j\le l} \in \Z_{\ge0}^{k \times l}$, we say that $(u, v) \in \cP_{k,l}$ is \emph{$Q$-feasible} if the following two conditions, which are equivalent (see \zcref{lem:b-matching-equiv} below), are met.
\begin{enumerate}
    \item[(P)] \customlabel{item:primal}{(P)} There is $Q' = (q'_{i,j})_{1 \le i \le k, 1 \le j \le l}\in \cQ_{u,v}$ such that $q'_{i,j} \le q_{i,j}$ for $i \in [k]$ and $j \in [l]$.
    \item[(D)] \customlabel{item:dual}{(D)} $\displaystyle \sum_{i \in I} u_i- \sum_{j \in J} v_j \le \sum_{i \in I} \sum_{j \in [l]\setminus J} q_{i,j}$ for all $I\subseteq [k]$ and $J\subseteq [l]$.
\end{enumerate}

\begin{lemma}\label{lem:b-matching-equiv}
    For $k,l \in \Z_{\ge0}$, $(u,v) \in \cP_{k,l}$, and $Q = {(q_{i,j})}_{1\le i\le k, 1\le j\le l} \in \Z_{\ge0}^{k \times l}$, the two conditions \ref{item:primal} and \ref{item:dual} are equivalent.
    Moreover, the conditions can be checked in polynomial time.
\end{lemma}
\begin{proof}
    The claim is a rephrasing of the dual characterization of the existence of a perfect $b$-matching with capacity constraints~\cite[Corollary~21.7a]{schrijver2003combinatorial}, stating that the desired $Q'$ in \ref{item:primal} exists if and only if
    \begin{align}
        \sum_{i \in I} u_i + \sum_{j \in J} v_j - \frac12 \prn*{\sum_{i=1}^k u_i + \sum_{j=1}^l v_j} \le \sum_{i \in I} \sum_{j \in J} q_{i,j}
    \end{align}
    holds for $I \subseteq [k]$ and $J \subseteq [l]$.
    Swapping $J$ and $[l] \setminus J$ and using $\sum_{i=1}^k u_i = \sum_{j=1}^l v_j$, we obtain the desired characterization.
    The polynomial solvability is also obtained via reduction to minimum-cost flow; see, e.g.,~\cite[Chapter~21]{schrijver2003combinatorial}.
\end{proof}

\begin{restatable}{theorem}{gppartition}\label{thm:gp-partition}
    Let $A$ and $B$ be bases of a matroid representable over a field $\F$ of characteristic zero, and $(X_1, \dots, X_k)$ and $(Y_1, \dots, Y_l)$ be partitions of $A\setminus B$ and $B\setminus A$, respectively.
    Let $\cP\subseteq \cP_{k,l}$ be a subset, and assume that there exists a unique bijection $\sigma\colon \cP \to \mathcal{Q}_{x,y}$ with $x \coloneqq (|X_1|, \dotsc, |X_k|)$ and $y \coloneqq (|Y_1|, \dotsc, |Y_l|)$ such that $(u,v)$ is $\sigma(u,v)$-feasible for each $(u,v)\in \cP$.
    Then, there exist $(U, V) \in \cU_{A,B}^\bM(\F)$ such that $\big((|U\cap X_1|, \dots, |U\cap X_k|), (|V\cap Y_1|, \dots, |V \cap Y_l|)\big) \in \cP$.
\end{restatable}

\begin{proof}
    For each $Q=(q_{i,j})_{1\le i\le k,\,1\le j\le l}\in\cQ_{x,y}$, define
    \[
        \Phi_Q
        \coloneqq
        \Set{
            \varphi\in\Phi_{A\setminus B,B\setminus A}
        }{
            |\varphi(X_i)\cap Y_j|=q_{i,j}
            \ (i\in[k], j\in[l])
        }.
    \]
    Each $\Phi_Q$ is nonempty, and the sets $\Phi_Q$ for $Q\in\cQ_{x,y}$ form a partition of $\Phi_{A\setminus B,B\setminus A}$.
    For each $(u,v)\in\cP$, define
    \[
        \cU_{u,v}
        \coloneqq
        \Set{
            (U,V)\in\cX_{A\setminus B,B\setminus A}
        }{
            |U\cap X_i|=u_i\ (i\in[k]),\
            |V\cap Y_j|=v_j\ (j\in[l])
        }.
    \]
    These sets are pairwise disjoint. Moreover, each $\cU_{u,v}$ is nonempty because $(u,v)$ is $\sigma(u,v)$-feasible.
    Write $\cP=\set{(u^1,v^1),\dotsc,(u^m,v^m)}$, where $m\coloneqq|\cP|$, and set $Q^s\coloneqq\sigma(u^s,v^s)$ for $s\in[m]$.
    Define $\bmPhi\coloneqq(\Phi_{Q^1},\dotsc,\Phi_{Q^m})$ and $\bmcU \coloneqq(\cU_{u^1,v^1},\dotsc,\cU_{u^m,v^m})$.
    
    Fix $Q=(q_{i,j})\in\cQ_{x,y}$, $\varphi\in\Phi_Q$, and $(u,v)\in\cP$.
    Choosing $U\subseteq A\setminus B$ such that $(U,\varphi(U))\in\cU_{u,v}$ amounts to choosing, for each $i\in[k]$ and $j\in[l]$, some number $q'_{i,j}$ of the $q_{i,j}$ elements of $X_i$ mapped by $\varphi$ into $Y_j$, where $Q'=(q'_{i,j}) \in \cQ_{u,v}$
    satisfies $q'_{i,j}\le q_{i,j}$. Consequently,
    \begin{align}\label{eq:gp-partition-tau}
        \tau(\varphi,\cU_{u,v})
        =
        \sum_{\substack{
            Q'=(q'_{i,j})\in\cQ_{u,v}\\
            q'_{i,j}\le q_{i,j}\ 
            (i\in[k],\,j\in[l])
        }}
        \prod_{i=1}^k\prod_{j=1}^l
        \binom{q_{i,j}}{q'_{i,j}}.
    \end{align}
    In particular,~\eqref{eq:gp-partition-tau} depends only on $Q$, and not on the
    choice of $\varphi\in\Phi_Q$. Hence, $(\bmPhi,\bmcU)$ is
    compatible.
    Furthermore,~\eqref{eq:gp-partition-tau} is nonzero if and only if $(u,v)$ is $Q$-feasible.
    Thus, by the assumption on $\sigma$, the matrix $C_{\bmPhi,\bmcU}$ has a unique nonzero transversal, and \zcref{lem:h-c,lem:matrix} yield the desired sets $U$ and $V$.
\end{proof}

We note that, for a given $(x,y) \in \cP_{k,l}$ and $\cP \subseteq \cP_{k,l}$, the condition of $\cP$ in \zcref{thm:gp-partition} can be checked in time polynomial in $k$, $l$, and $|\cP|$.
For this, we first check if $|\cP|=|\cQ_{x,y}|$ is true or not; if not, $\cP$ cannot satisfy the condition.
If $|\cP|=|\cQ_{x,y}|$, we construct a 0-1 matrix $F \in \set{0,1}^{\cP \times \cQ_{x,y}}$ such that the $((u,v),Q)$ entry is $1$ if $(u,v)$ is $Q$-feasible and $0$ otherwise for each $(u,v) \in \cP$ and $Q \in \cQ_{x,y}$.
This matrix can be constructed in polynomial time by \zcref{lem:b-matching-equiv}.
Then, the uniqueness of a nonzero transversal of this matrix can be checked via the Dulmage--Mendelsohn decomposition; see~\cite[Theorem~2.2.22]{murota2010matrices}.

For the case where $k=l=2$, \zcref{thm:gp-partition} can be specialized as follows.

\begin{restatable}{corollary}{gpbipartition}\label{cor:gp-bipartition}
    Let $A$ and $B$ be bases of a matroid representable over a field $\F$ of characteristic zero, and $(X_1, X_2)$ and $(Y_1, Y_2)$ be bipartitions of $A\setminus B$ and $B\setminus A$, respectively.
    Let $s \coloneqq \min\set{|X_1|,|X_2|,|Y_1|,|Y_2|}$ and $\cP \subseteq \cP_{2,2}$ be a subset with $|\cP| = s+1$ such that $u_i \le |X_i|$ and $v_j \le |Y_j|$ hold for each $(u,v)\in \cP$ and $i,j \in [2]$.
    Assume that there exists a unique ordering $(u^0, v^0), \dotsc, (u^s, v^s)$ of $\cP$ such that
    \begin{align}\label{eq:gp-bipartition}
        \begin{aligned}
            u^t_1 - v^t_1 &\le \min\set{|X_1|, |Y_2|} - t, &
            u^t_1 - v^t_2 &\le \max\set{0, |X_1|-|Y_2|} + t, \\
            u^t_2 - v^t_1 &\le \max\set{0, |Y_2|-|X_1|} + t, &
            u^t_2 - v^t_2 &\le \min\set{|X_2|, |Y_1|} - t
        \end{aligned}
    \end{align}
    for each $t \in \set{0, \dotsc, s}$.
    Then, there exist $(U, V) \in \cU_{A,B}^\bM(\F)$ such that $((|U \cap X_1|, |U \cap X_2|), (|V \cap Y_1|, |V \cap Y_2|)) \in \cP$.
\end{restatable}
\begin{proof}
    Let $x \coloneqq (|X_1|, |X_2|)$ and $y \coloneqq (|Y_1|, |Y_2|)$.
    Then, each $(u,v) \in \cP_{2,2}$ is $Q$-feasible for $Q = {(q_{i,j})}_{1\le i,j \le 2} \in \cQ_{x,y}$ if and only if $u_i \le x_i$, $v_j \le y_j$, and $u_i-v_j \le q_{i,3-j}$ for each $i,j \in [2]$.
    Furthermore, $\cQ_{x,y}$ is explicitly written down as
    \begin{align}
        \cQ_{x,y} &= \Set*{\begin{pmatrix}
            \max\set{0, x_1-y_2} & \min\set{x_1, y_2} \\
            \min\set{x_2, y_1} & \max\set{0, y_2-x_1}
        \end{pmatrix} + t\begin{pmatrix} 1 & -1 \\ -1 & 1 \end{pmatrix}}{0 \le t \le \min\set{x_1,x_2,y_1,y_2}}.
    \end{align}
    Thus, \zcref{thm:gp-partition} proves the claim.
\end{proof}

We note that, for matroids representable over a field of characteristic zero, the exchange property obtained by applying \zcref{thm:gp-matroids} to the function $c$ of \zcref{lem:bipartition-gp} can also be derived as an application of \zcref{cor:gp-bipartition} with $(X_1, X_2)\coloneqq (X, (A\setminus B)-X)$ and
\begin{align}
    \cP \coloneqq & \bigcup_{t=1}^2 \Set{((u+p_t, 0), (u, p_t))}{u \in \Z,\ \max\set{d_t, |X|-|Y_{3-t}|} \le u \le  \min\set{|X|, |Y_t|}}.
\end{align}

\subsubsection{Yet Another Common Generalization of \texorpdfstring{\zcref{thm:main,thm:equitability-exchange}}{Theorems~\ref{thm:main} and~\ref{thm:equitability-exchange}}}

\zcref{cor:gp-bipartition} can generate yet another common generalization of \zcref{thm:main,thm:equitability-exchange} for matroids representable over a field of characteristic zero as follows, where an illustration is given \zcref{fig:blue}.

\begin{figure}
    \centering
    \begin{subfigure}{0.5\linewidth}
        \centering
        \begin{tikzpicture}[
          font=\small,
          line width=0.7pt,
          line cap=round,
          line join=round,
          brace above/.style={decorate,decoration={brace,amplitude=5pt}},
          brace below/.style={decorate,decoration={brace,mirror,amplitude=5pt}},
          myarrow/.style={<->, >=Stealth, line width=0.7pt}
        ]
          \def\W{5}
          \def\H{0.35}
          \def\WX{2.6}
          \def\WYone{3.2}
          \def\WU{2}
          \def\WVone{2}
          \def\Boriginy{-1.5}
        
          \node[left=12pt] at (0, 0.5*\H) {$A\setminus B$};
          \draw (0,0) rectangle (\W, \H);
          \draw (\WX, 0) -- (\WX, \H);
          \draw[brace above] (0.05, \H+0.1) -- node[midway,above=6pt] {$X$} (\WX-0.05, \H+0.1);
          \fill[gray!25] (0.07, 0.07) rectangle (0.07+\WU, \H-0.07);
        
          \node[left=12pt] at (0, \Boriginy+0.5*\H) {$B\setminus A$};
          \draw (0, \Boriginy) rectangle (\W, \Boriginy+\H);
          \draw (\WYone, \Boriginy) -- (\WYone, \Boriginy+\H);
          
          \draw[brace below] (0.05, \Boriginy-0.10) -- node[midway,below=6pt] {$Y_1$} (\WYone-0.05, \Boriginy-0.10);
          \draw[brace below] (\WYone+0.05, \Boriginy-0.10) -- node[midway,below=6pt] {$Y_2$} (\W-0.05, \Boriginy-0.10);
          
          \fill[gray!25] (0.07, \Boriginy+0.07) rectangle (0.07+\WVone, \Boriginy+\H-0.07);
          \draw[myarrow] (0.07, \Boriginy+\H+0.15) -- (0.07+\WVone, \Boriginy+\H+0.15) node[midway, above] {$\ge |X|-p+1$};
        \end{tikzpicture}
        \caption{Case~\ref{it:blue-first}}
    \end{subfigure}%
    \begin{subfigure}{0.5\linewidth}
        \centering
        \begin{tikzpicture}[
          font=\small,
          line width=0.7pt,
          line cap=round,
          line join=round,
          brace above/.style={decorate,decoration={brace,amplitude=5pt}},
          brace below/.style={decorate,decoration={brace,mirror,amplitude=5pt}},
          myarrow/.style={<->, >=Stealth, line width=0.7pt}
        ]
          \def\W{5}
          \def\H{0.35}
          \def\WX{2.6}
          \def\WYone{3.2}
          \def\WU{1.5}
          \def\WVone{1.0}
          \def\WVtwo{0.5}
          \def\Boriginy{-1.5}
        
          \node[left=12pt] at (0, 0.5*\H) {$A\setminus B$};
          \draw (0,0) rectangle (\W, \H);
          \draw (\WX, 0) -- (\WX, \H);
          \draw[brace above] (0.05, \H+0.1) -- node[midway,above=6pt] {$X$} (\WX-0.05, \H+0.1);
          \fill[gray!25] (0.07, 0.07) rectangle (0.07+\WU, \H-0.07);
        
          \node[left=12pt] at (0, \Boriginy+0.5*\H) {$B\setminus A$};
          \draw (0, \Boriginy) rectangle (\W, \Boriginy+\H);
          \draw (\WYone, \Boriginy) -- (\WYone, \Boriginy+\H);
          
          \draw[brace below] (0.05, \Boriginy-0.10) -- node[midway,below=6pt] {$Y_1$} (\WYone-0.05, \Boriginy-0.10);
          \draw[brace below] (\WYone+0.05, \Boriginy-0.10) -- node[midway,below=6pt] {$Y_2$} (\W-0.05, \Boriginy-0.10);
          
          \fill[gray!25] (0.07, \Boriginy+0.07) rectangle (\WVone+0.07, \Boriginy+\H-0.07);
          
          \fill[gray!25] (\WYone+0.07, \Boriginy+0.07) rectangle (\WYone+0.07+\WVtwo, \Boriginy+\H-0.07);
          \draw[myarrow] (\WYone+0.07, \Boriginy+\H+0.15) -- (\WYone+0.07+\WVtwo, \Boriginy+\H+0.15) node[midway, above] {$p$};
        \end{tikzpicture}
        \caption{Case~\ref{it:blue-second}}
    \end{subfigure}%
    \caption{Illustration of \zcref{thm:blue}. Gray areas in $A\setminus B$ and $B\setminus A$ indicate $U$ and $V$, respectively.}\label{fig:blue}
\end{figure}

\begin{restatable}{theorem}{blue}\label{thm:blue}
    Let $A$ and $B$ be bases of a matroid $\bM$ representable over a field $\F$ of characteristic zero.
    For any subset $X\subseteq A\setminus B$, bipartition $(Y_1, Y_2)$ of $B\setminus A$, and integer $p\ge 0$, there exists $(U,V) \in \cU_{A,B}^\bM(\F)$ such that $U\subseteq X$ and either of the following two is met:
    \begin{enumerate}[label={\upshape(\roman*)}] \itemsep0em
        \item \label{it:blue-first} $V\subseteq Y_1$ and $|U|=|V|\ge |X|-p+1$, or
        \item \label{it:blue-second} $|V\cap Y_2| = p$.
    \end{enumerate}
\end{restatable}

\begin{proof}
    Let $s \coloneqq \min\set{|X|, |A\setminus B|-|X|, |Y_1|, |Y_2|}$, $m \coloneqq \max\set{0, |X|-|Y_2|}$, and
    \begin{align}
        H\coloneqq \Set{m+t}{t = 0, \dotsc, s} = \set{m, m+1, \dotsc, \min\set{|X|, |Y_1|}}.
    \end{align}
    We show that the claimed exchange property corresponds to \zcref{cor:gp-bipartition} for the bipartitions $(X, (A\setminus B)-X)$ and $(Y_1,Y_2)$ of $A\setminus B$, $B\setminus A$, respectively, and
    \[
        \cP \coloneqq \Set{((h+p, 0), (h, p))}{h \in H, h \le |X|-p} \cup \Set{((h,0),(h,0))}{h \in H, h \ge |X|-p+1} \subseteq \cP_{2,2}.
    \]
    To see that $\cP$ satisfies the condition in \zcref{cor:gp-bipartition}, assume that an ordering $(u^0, v^0), \dotsc, (u^s, v^s)$ of $\cP$ satisfies the inequalities~\eqref{eq:gp-bipartition} for each $t \in \set{0, \dotsc, s}$.
    Let $h \colon \set{0, \dotsc, s} \to H$ be such that $v^t_1 = h(t)$ for $t \in \set{0, \dotsc, s}$.
    By the inequality $u^t_1 - v^t_2 \le m+t$ coming from~\eqref{eq:gp-bipartition} and $u_1^t - v_2^t = h(t)$ for each $t$, we must have $h(t) = m+t$; i.e., $u^t = (m+t+p, 0)$ and $v^t = (m+t, p)$ if $m+t \le |X|-p$ and $u^t = v^t = (m+t,0)$ if $m+t \ge |X|-p+1$.
    This shows that no ordering other than this satisfies~\eqref{eq:gp-bipartition}, and it can be checked that one indeed satisfies them.
    Indeed, for each $t \in \set{0, \dotsc, s}$, both $u_2^t-v_1^t$ and $u_2^t-v_2^t$ are nonpositive by $u_2^t=0$, hence the inequalities are satisfied.
    For the remaining case, $u_1^t - v_1^t = p \le |X|-m-t = \min\set{|X|,|Y_2|}-t$ if $m+t \le |X|-p$ and $u_1^t - v_1^t = 0 \le \min\set{|X|,|Y_2|}-t$ otherwise.
    Therefore, $\cP$ satisfies the condition in \zcref{cor:gp-bipartition} and the result follows from \zcref{cor:gp-bipartition}.
\end{proof}

Observe that if $|X|>|Y_1|$ and $p=1$, then \ref{it:blue-first} is never satisfied, hence the theorem specializes to \zcref{thm:equitability-exchange} for matroids representable over a field of characteristic zero.
Another notable special case is when $p=r(X_2+Y_2)-|X_2|+1$ where $X_1 \coloneqq X$ and $X_2 \coloneqq (A\setminus B)-X_1$.
In this case, $X_2+(Y_2\cap V)$ is dependent for any $V\subseteq B\setminus A$ with $|Y_2\cap V| = p$, thus \ref{it:blue-second} is never satisfied.
\ref{it:blue-first} gives $U\subseteq X_1$ and $V\subseteq Y_1$ with $|U|=|V| \ge |A\setminus B|-r(X_2+Y_2)$, or equivalently, for $U'\coloneqq (A\setminus B)-U$ and $V'\coloneqq (B\setminus A)-V$ we get $U'\supseteq X_2$ and $V'\supseteq Y_2$ with $|U'|=|V'| \le r(X_2+Y_2)$.
This shows that for matroids representable over a field of characteristic zero, \zcref{thm:blue} is a common generalization of \zcref{thm:main,thm:equitability-exchange}, with the additional property of reconfiguration.

We also note the relation between \zcref{thm:two-max,thm:blue}.
If $p \le \max\set{|X| - |Y_1|, 0}$, then \zcref{thm:two-max} applied with $(p_1, p_2)=(p, 0)$ shows that a pair satisfying~\ref{it:blue-second} always exists; hence, in this range, \zcref{thm:blue} follows from \zcref{thm:two-max}.
However, for $p > \max\set{|X| - |Y_1|, 0}$,~\ref{it:blue-second} may fail, and \zcref{thm:blue} guarantees~\ref{it:blue-first} instead.

\subsection{Randomized Algorithm for Finding Exchangeable Sets}\label{sec:rand-algo}

Our proofs of exchange properties using \zcref{cor:reconf-main} do not yield a polynomial-time algorithm for finding the subsets $U$ and $V$.
Nevertheless, for exchange properties derived from \zcref{thm:gp-partition}, we show that a randomized polynomial-time algorithm of Las Vegas type is available via \zcref{prop:exact}.

\begin{restatable}{theorem}{gpalgo}\label{thm:gp-algo}
    Let $\bM$ be a matroid representable over a field $\F$.
    There is a randomized algorithm that, given a matrix $K \in \F^{r \times n}$ representing $\bM$, bases $A$ and $B$ of $\bM$, partitions $(X_1, \dotsc, X_k)$ and $(Y_1, \dotsc, Y_l)$ of $A\setminus B$ and $B\setminus A$, respectively, and $\cP \subseteq \cP_{k,l}$, outputs a pair $(U, V)$ such that $A-U+V$ and $B+U-V$ are bases of $\bM$ and $\prn[\big]{(|U\cap X_1|, \dots, |U\cap X_k|), (|V\cap Y_1|, \dots, |V \cap Y_l|)} \in \cP$, provided that such a pair exists.
    The algorithm runs in expected time polynomial in $n$ and ${(r+1)}^{k+l}$, assuming unit-time field operations.
\end{restatable}

\begin{proof}
    We may assume $A \cup B = E(\bM) \eqqcolon E$ and $A \cap B = \emptyset$ by contracting $A \cap B$ and deleting $E - (A \cup B)$; the corresponding matrix representation is obtained by applying Gaussian elimination to $K$ and removing rows and columns.
    Note that a matrix representation of the dual matroid $\bM^*$ can also be obtained from $K$ by Gaussian elimination.
    Since $B = E - A$, a set $B' \subseteq E$ is of the form $B' = A - U + V$ with $U \subseteq A \setminus B$ and $V \subseteq B \setminus A$ if and only if $U = A - B'$ and $V = B' - A$, and then $B + U - V = E - B'$.
    Hence, $(U, V)$ is a pair such that $A-U+V$ and $B+U-V$ are bases if and only if $B' \coloneqq A - U + V$ is a common basis of $\bM$ and $\bM^*$, and in this case $|B' \cap X_i| = |X_i| - |U \cap X_i|$ for $i \in [k]$ and $|B' \cap Y_j| = |V \cap Y_j|$ for $j \in [l]$.
    For $i \in [k]$ and $j \in [l]$, define weight functions $w_i \coloneqq \ones_{X_i}$ and $w_{k+j} \coloneqq \ones_{Y_j}$, and let
    \[
        \cT \coloneqq \Set{(|X_1| - u_1, \dotsc, |X_k| - u_k, v_1, \dotsc, v_l)}{(u, v) \in \cP}.
    \]
    Then, by the above correspondence, a desired pair exists if and only if $\bM$ and $\bM^*$ have a common basis $B'$ with $(w_1(B'), \dotsc, w_{k+l}(B')) \in \cT$, and $(U, V)$ is recovered from $B'$ in linear time.
    The algorithm repeats applying the algorithm of \zcref{prop:exact} to $\bM$ and $\bM^*$ with these weight functions and $\cT$ until a common basis is found.
    As a feasible pair is guaranteed to exist, this yields the Las Vegas guarantee by the same argument as in the proof of \zcref{thm:weighted-equitability}.
    Since all weights are $0$-$1$, each application uses a number of field operations polynomial in $n$ and ${(r+1)}^{k+l}$ by \zcref{prop:exact}, and the claimed expected running time follows.
\end{proof}

\zcref{thm:gp-algo} applies to \zcref{thm:blue,thm:with_ranks}.
Moreover, we show that \zcref{thm:gp-algo} also yields an algorithm for \zcref{thm:two-max}, as claimed in \zcref{thm:two-max-algo}, which is restated here for readability.

\twomaxalgo*

\begin{proof}
    For each $d \in \set{0, \dotsc, |X|+1}$, we apply \zcref{thm:gp-algo} to \zcref{lem:bipartition-gp} with $(d_1, d_2) \coloneqq (d, |X|+1-d)$ to obtain $U_d \subseteq X$ and $V_d \subseteq B \setminus A$ such that $A-U_d+V_d$ and $B+U_d-V_d$ are bases,  and, for some $t \in [2]$, both $|V_d \cap Y_{3-t}| = p_t$ and $|V_d \cap Y_t| \ge d_t$ hold; a pair with these properties exists by \zcref{thm:gp-matroids,lem:bipartition-gp}.
    Since $|V_d \cap Y_t| \le |V_d| = |U_d| \le |X|$, the alternative $t=2$ cannot occur for $d = 0$, and the alternative $t=1$ cannot occur for $d = |X|+1$.
    Hence, the largest value $d_0$ of $d$ for which the obtained pair satisfies the alternative $t=1$ exists and satisfies $0 \le d_0 \le |X|$.
    Then, $(U_1, V_1) \coloneqq (U_{d_0}, V_{d_0})$ satisfies $|V_1 \cap Y_2| = p_1$ and $|V_1 \cap Y_1| \ge d_0$, and $(U_2, V_2) \coloneqq (U_{d_0+1}, V_{d_0+1})$ satisfies $|V_2 \cap Y_1| = p_2$ and $|V_2 \cap Y_2| \ge |X| + 1 - (d_0 + 1) = |X| - d_0$, whence $|V_1 \cap Y_1| + |V_2 \cap Y_2| \ge |X|$.
\end{proof}

\subsection{Applications}\label{sec:gp-applications}

\subsubsection{Uniqueness Versions of Exchange Properties} \label{sec:uniqueness}

In some contexts, converses of exchange properties are also of interest.
For example, consider the \emph{perfect-matching lemma}~\cite{brualdi1969comments}, which states that for a basis $A$ and a subset $B$ of the ground set, if $B$ is also a basis, then there exists a bijection $\varphi\colon A \setminus B \to B \setminus A$ such that $A - x + \varphi(x)$ is a basis for every $x \in A \setminus B$.
The converse of this property does not hold in general; that is, the existence of such a bijection does not necessarily imply that $B$ is a basis.
Nevertheless, if the bijection is unique, then $B$ is indeed a basis~\cite{krogdahl1977dependence}.
This property, sometimes referred to as the \emph{unique-matching lemma}~\cite[Lemma~2.3.18]{murota2010matrices}, turns out to be a key ingredient in the augmenting-path algorithm for matroid intersection~\cite{krogdahl1974combinatorial,krogdahl1976combinatorial}.
While the perfect-matching and unique-matching lemmas hold for general matroids, simple linear-algebraic proofs are available for representable cases (see~\cite[Remark~2.3.23]{murota2010matrices}).

Analogously, it is natural to formulate a uniqueness version of the multiple exchange property.
The Grassmann--Plücker identity \eqref{eq:multiple-gp} directly implies that for a representable matroid $\bM$ and subsets $A, B \subseteq E(\bM)$ with $|A| = |B|$ and $X \subseteq A \setminus B$, if there uniquely exists $Y \subseteq B \setminus A$ such that both $A - X + Y$ and $B + X - Y$ are bases, then $A$ and $B$ are bases as well.
It was first shown by Bowler, Ding, and Lowen in an unpublished work that this property holds for general matroids as well.
We observe that it also follows directly from the result of Kung~\cite{kung1978alternating} (\zcref{cor:stronger_kung}).

\begin{corollary}\label{cor:unique-multiple}
    Let $\bM$ be a matroid, $A, B \subseteq E(\bM)$ subsets with $|A|=|B|$, and $X\subseteq A\setminus B$.    
    If there uniquely exists $Y\subseteq B\setminus A$ such that $A-X+Y$ and $B+X-Y$ are bases, then $A$ and $B$ are bases as well. 
\end{corollary}
\begin{proof}
    Let $A'\coloneqq A-X+Y$ and $B'\coloneqq B+X-Y$.
    Observe that the uniqueness of $Y$ implies that there exist no nonempty subsets $U\subseteq Y$ and $V\subseteq (B\setminus A)-Y$ such that $A'-U+V = A-X+(Y-U+V)$ and $B'+U-V = B+X-(Y-U+V)$ are bases.
    Then, the contraposition of \zcref{cor:stronger_kung} applied to the bases $A'$ and $B'$ and the bipartitions $(Y, (A\setminus B)-X)$ and $((B\setminus A)-Y, X)$ of $A'\setminus B'$ and $B'\setminus A'$, respectively, implies that $Y+((B\setminus A)-Y)+(A\cap B) = B$ is independent and $((A\setminus B)-X)+X+(A\cap B) = A$ is spanning.
    As $A$ and $B$ have size $r$, they are necessarily bases as well.
\end{proof}

We note that the previous proof shows that the condition $|A|=|B|$ can be relaxed to $|A|\le |B|$.

Motivated by these, it is natural to ask whether similar uniqueness versions of more general exchange properties hold.
For matroids representable over a field of characteristic zero, a uniqueness version of \zcref{thm:main} with the weaker upper bound on the sizes of $U$ and $V$ follows from the corresponding Grassmann--Plücker identity~\eqref{eq:main-gp}, in the same way as the case of the multiple exchange property.

\begin{restatable}{theorem}{mainuniquezero}\label{thm:main-unique-czero}
    Let $\bM$ be a matroid representable over a field of characteristic zero, with ground set $E$ and rank function $r$.
    Let $A,B \in \binom{E}{r(E)}$ and $X \subseteq A \setminus B$ and $Y \subseteq B \setminus A$ be subsets such that $|X+Y| < |A \setminus B|$.
    If there uniquely exist $U \subseteq A \setminus B$ and $V\subseteq B \setminus A$ such that $X \subseteq U$, $Y \subseteq V$, $|U| = |V| \le |X+Y|$, and $A-U+V$ and $B+U-V$ are bases, then $A$ and $B$ are bases as well.
\end{restatable}
\begin{proof}
    Let $K \in \F^{r \times n}$ be a matrix representing $\bM$.
    We check the coefficients of~\eqref{eq:main-gp} for $K$.
    For an integer $u$ with $\max\set{|X|,|Y|} \le u \le |X+Y|$, the coefficient of $(U,V) \in \cX_{A\setminus B,B\setminus A}$ with $|U|=|V|=u$, $X \subseteq U$, and $Y \subseteq V$ is $\tbinom{|X+Y|-|A \setminus B|}{|X+Y|-u}$,
    which is zero if and only if $0 \le |X+Y|-|A \setminus B| < |X+Y|-u$, that is,  $u < |A \setminus B| \le |X+Y|$.
    As $|X+Y| < |A \setminus B|$ is assumed, the coefficient is nonzero for all $u$.
    Hence, the unique existence of the subsets $U$ and $V$ implies that the right-hand side of~\eqref{eq:main-gp} is nonzero, and thus $A$ and $B$ are bases.
\end{proof}

\zcref{thm:main-unique-czero} recovers \zcref{cor:unique-multiple} for the zero-characteristic case if $Y = \emptyset$; note that the $X=A \setminus B$ and $Y=\emptyset$ case does not satisfy the requirement $|X+Y| < |A \setminus B|$ of \zcref{thm:main-unique-czero}, whereas this case is trivial.
The uniqueness version with the stronger upper bound $|U|=|V| \le r(X+Y)$ can also be obtained from the corresponding Grassmann--Plücker identity by explicitly computing the coefficients and verifying that they are nonzero.

We also remark that the condition $|X+Y| < |A\setminus B|$ cannot be relaxed to $|X+Y| \le |A \setminus B|$ in general, as the following counterexample shows: consider the matroid represented by $K = \begin{psmallmatrix} 1 & 1 & 0 & 0 \\ 0 & 0 & 1 & 1 \end{psmallmatrix}$ and let $A = \set{1, 2}$, $B = \set{3, 4}$, $X = \set{1}$, and $Y = \set{3}$.
Then, $U=\set{1}$ and $V=\set{3}$ is the unique pair such that $X \subseteq U$, $Y \subseteq V$, $|U|=|V| \le |X+Y| = 2$, and $A-U+V$ and $B+U-V$ are bases, whereas $A$ and $B$ are not bases.

It is also worth mentioning that the same proof as \zcref{thm:main-unique-czero} does not work for matrices over a field of positive characteristic, as the binomial coefficient might become zero.
In fact, somewhat interestingly, \zcref{thm:main-unique-czero} fails for binary matroids.
Let $e_i \in \GF(2)^r$ denote the $i$th unit vector and $\ones \in \GF(2)^r$ the all ones vector.
Let $r \ge 4$ be even, $A = \set{\ones, e_2, \dots, e_r}$, $B=\set{e_1, \ones-e_2,\ones-e_3,\dots, \ones-e_r}$, $X=\set{\ones}$, $Y=\set{e_1}$, and consider the binary matroid on $A+B$ defined by these vectors.
Then, $A-X+Y$ and $B+X-Y$ are bases.
For integers $i$ and $j$ with $2 \le i, j \le r$, the set $B+\set{\ones,e_i}-\set{e_1,\ones-e_j}$ cannot be a basis unless $i=j$ because $\set{\ones, e_i, \ones - e_i}$ is dependent.
Moreover, $A-\set{\ones,e_i}+\set{e_1,\ones-e_i}$ is not a basis as $r$ is even.
Therefore, $U=X$ and $V=Y$ is the unique pair such that $X \subseteq U$, $Y \subseteq V$, $|U|=|V|\le 2$, and $A-U+V$ and $B+U-V$ are bases.
However, $B$ is not a basis as $r$ is even.
Thus, \zcref{thm:main-unique-czero} fails for this instance.

\subsubsection{Exchange Properties of Representable Valuated Matroids} \label{sec:valuated}

A further implication of Grassmann--Plücker-type identities lies in \emph{valuated matroids}. 
A valuated matroid~\cite{Dress1990-uc} is a function $\omega\colon \binom{E}{r} \to \R \cup \set{-\infty}$ satisfying the following property: for any $A, B \in \binom{E}{r}$ and $x \in A \setminus B$, there exists $y \in B \setminus A$ such that $\omega(A) + \omega(B) \le \omega(A - x + y) + \omega(B + x - y)$. 
This property can be viewed as a quantitative generalization of the symmetric exchange property of matroids, and a multiple-exchange version also holds~\cite{murota2018a}.
Valuated matroids appear in many research areas, such as combinatorial optimization~\cite{Dress1990-uc,murota2003discrete}, tropical geometry~\cite{maclagan2015introduction}, systems analysis~\cite{murota2010matrices}, and economics~\cite{fujishige2003a}.

An extension of the representability of matroids to valuated matroids is as follows.
Let $\F$ be a \emph{valuated field}, which is a field equipped with a function $v\colon \F \to \R\cup\set{+\infty}$, called a (non-Archimedean) \emph{valuation}, such that $v(0) = +\infty$, $v(1) = 0$, $v(a+b) \ge \min\set{v(a), v(b)}$, and $v(ab) = v(a)+v(b)$ for $a,b \in \F$.
A valuation $v$ is \emph{trivial} over a subfield $\K \subseteq \F$ if $v(a) = 0$ for $a \in \K \setminus \set{0}$.
A typical example of a valuated field is the univariate rational function field $\K(t)$ with valuation $v(p/q) \coloneqq -\deg_t p/q = \deg_t q - \deg_t p$ for $p, q \in \K[t]$ with $q \ne 0$, which is trivial over $\K$.
For a matrix $K \in \F^{r \times n}$ over a valuated field $\F$ with valuation $v$, a function $\omega_K\colon \binom{[n]}{r} \to \R \cup \set{-\infty}$ defined by $\omega_K(A) \coloneqq -v(\det K[A])$ forms a valuated matroid~\cite{Dress1990-uc}.
This is an immediate consequence of the Grassmann--Plücker identity~\eqref{eq:multiple-gp} for $|X|=1$.
Valuated matroids constructed in this way, sometimes called \emph{representable valuated matroids}, form a quantitative generalization of representable matroids.

The Grassmann--Plücker identities shown in this section imply exchange properties for representable valuated matroids beyond the matroid case.
As an example, we present the analog of \zcref{thm:main} with the weaker upper bound for representable valuated matroids.

\begin{proposition} \label{prop:valuated-main}
    Let $\F$ be a valuated field equipped with a valuation $v$, $K \in \F^{r \times n}$ a matrix, and $A, B \in \binom{[n]}{r}$.
    For any $X\subseteq A\setminus B$ and $Y\subseteq B\setminus A$, there exist $U\subseteq A\setminus B$ and $V\subseteq B\setminus A$ such that $X\subseteq U$, $Y\subseteq V$, $|U|=|V|\le |X+Y|$ and
    $\omega_K(A) + \omega_K(B) \le \omega_K(A-U+V)+\omega_K(B+U-V)$.
\end{proposition}
\begin{proof}
    Note that $v(a) \ge 0$ for nonzero $a \in \Z$.
    Applying $-v$ to \eqref{eq:main-gp}, we get
    \begin{align}
        \omega_K(A)+\omega_K(B)
        &= -v(\mu_{A,B}^K(\emptyset, \emptyset)) \\
        &= -v\prn*{\sum_{(U,V) \in \cX_{A\setminus B, B\setminus A}\::\: X \subseteq U, Y \subseteq V, |U|\le|X+Y|} \tbinom{|X+Y|-|A\setminus B|}{|X+Y|-|U|} \mu_{A,B}^K(U,V)} \\
        &\le -\min_{(U,V) \in \cX_{A\setminus B, B\setminus A}\::\: X \subseteq U, Y \subseteq V, |U|\le|X+Y|} v(\mu_{A,B}^K(U,V)) \\
        &= \max_{(U,V) \in \cX_{A\setminus B, B\setminus A}\::\: X \subseteq U, Y \subseteq V, |U|\le|X+Y|} (\omega_K(A-U+V) + \omega_K(B+U-V)).
        \qedhere
    \end{align}
\end{proof}

It remains an open problem whether \zcref{prop:valuated-main} can be generalized to arbitrary valuated matroids; our proof of \zcref{thm:main} does not seem to generalize.

\section{Conclusion}\label{sec:conclusion}
We close the paper by mentioning some open problems.

\begin{enumerate}[label=\arabic*.]
    \item For proving the exchange properties in \zcref{sec:harmonic,sec:gp} for matroids representable over a field $\F$ of characteristic zero, we relied on the fact that   
    for such a rank-$r$ matroid $\bM$ on a ground set $E$ with $|E|=2r$, there exists a harmonic function $g \colon \binom{E}{r} \to \F$ that is nonzero on a set $A\in \binom{E}{r}$ exactly if both $A$ and $E\setminus A$ are bases. 
    It remains an intriguing open problem to find further classes of matroids for which such a harmonic function, over, say, $\F=\Q$ exists, as all exchange properties shown in \zcref{sec:gp} extend to those classes.

    \item For matroids representable over a field of characteristic zero, we  proved several exchange properties in which the new basis pair can be obtained
    from the original one by a sequence of symmetric exchanges. It remains open whether such sequences can always be chosen to have polynomial length. In
    particular, it is open whether they can be found in (randomized) polynomial time.
    
    \item In \zcref{thm:weighted-equitability}, we showed a weighted version of the equitability theorem for matroids representable over a field of characteristic zero.
    This directly includes the existence of a matroid-constrained EF1 allocation for two agents with identical additive valuations.
    It remains open whether our framework could be used to show the analogous EF1 result for an arbitrary number of agents.
\end{enumerate}

\section*{Acknowledgments}

The authors are grateful to Kristóf Bérczi, Shinsaku Sakaue, and Yutaro Yamaguchi for initial discussions on the problem.
Taihei Oki was supported by JST FOREST Grant Number JPMJFR232L, JST BOOST Grant Number JPMJBY25A6, and JSPS KAKENHI Grant Number JP18J22141.
This research has been implemented with the support provided by the Lend\"ulet Programme of the Hungarian Academy of Sciences -- grant number LP2021-1/2021. This work was supported in part by EPSRC grant EP/X030989/1.

\paragraph{Data availability.}  
No data are associated with this article. Data sharing is not applicable to this article.

\paragraph{AI disclosure.}
We used GPT-5.5 Pro to explore possible proof strategies for \zcref{thm:reconf-mep,lem:c-equitability,lem:c-main}, and used Claude Fable 5 to strengthen \zcref{thm:with_ranks} to its current form.
The authors independently verified and completed the resulting proofs.

\printbibliography[heading=bibintoc]

\end{document}